\documentclass[11pt]{amsart}
\usepackage{amsfonts,amssymb,amsthm}
\usepackage{amsmath,amscd}
\usepackage{pstricks}
\usepackage{pstricks,pst-node}
\usepackage{mathrsfs}
\usepackage[all]{xy}
\usepackage{mathtools}

\DeclareMathAlphabet{\mathpzc}{OT1}{pzc}{m}{it}

\theoremstyle{plain}
\newtheorem{Thm}{Theorem}[section]
\newtheorem{Prop}[Thm]{Proposition}

\newtheorem{Lem}[Thm]{Lemma}
\newtheorem{Coro}[Thm]{Corollary}
\theoremstyle{definition}
\newtheorem{Def}[Thm]{Definition}

\newtheorem{Rem}[Thm]{Remark}

\numberwithin{equation}{section}

\def\bfUC#1{\bfU_{\mathbf v}({#1})}
\newcommand{\fsl}{{\frak{sl}_n}}

\newcommand{\Pn}{\mathcal P(n)}
\newcommand{\Qn}{\mathcal Q(n)}

\newcommand{\sfF}{\mathsf F}
\newcommand{\gl}{\frak{gl}_n}

\newcommand{\glm}{\frak{gl}_n^-}
\newcommand{\msC}{\mathcal G}

\newcommand{\msCc}{{\mathcal G}_{\mbc}}
\newcommand{\dmsC}{\dot{\mathcal G}}
\newcommand{\msCr}{{\mathcal G}_r}

\newcommand{\msPc}{{\mathcal P}_{\mbc}}
\newcommand{\dmsP}{\dot{\mathcal P}}
\newcommand{\msPr}{\mathcal P_r}

\newcommand{\dev}{\dot{\mathsf{ev}}}
\newcommand{\dtau}{\dot{\tau}}
\newcommand{\evr}{\mathsf{ev}_r}
\newcommand{\ev}{\mathsf{ev}}
\newcommand{\evc}{\mathsf{ev}_\mbc}
\newcommand{\ddHa}{\dot{{\mathfrak D}_\vtg}(n)_\sZ}
\newcommand{\ddbfHa}{\dot{\boldsymbol{\mathfrak D}_\vtg}(n)}

\def\su#1{^{#1}}
\newcommand{\Kbfj}{K^{\bfj}}
\newcommand{\afPin}{\Pi_\vtg(n)}

\newcommand{\afWn}{\mathcal W(n)_\sZ}

\newcommand{\afbfWn}{\boldsymbol{\mathcal W}(n)}

\newcommand{\tA}{{}^t\!A}
\newcommand{\tAm}{{}^t\!(A^-)}
\newcommand{\tB}{{}^t\!B}

\newcommand{\UglC}{\bfU_{\mathbf v}({\frak{gl}}_n)}
\newcommand{\ttv}{{v}}
\newcommand{\ttz}{\mathbf{v}}
\newcommand{\afUsl}{{\bfU}(\widehat{\frak{sl}}_n)}

\newcommand{\ms}{\mathscr}

\newcommand{\bin}{\bigcup}

\newcommand{\han}{\subseteq}

\newcommand{\scr}[1]{\mathscr #1}

\newcommand{\bsa}{{\boldsymbol{a}}}
\newcommand{\bsb}{\boldsymbol{b}}

\newcommand{\bsH}{{\boldsymbol{\sH}}}
\newcommand{\bsS}{{\boldsymbol{\sS}}}

\newcommand{\lan}{\langle}
\newcommand{\ran}{\rangle}
\newcommand{\dleb}{\left[\!\!\left[}

\newcommand{\leb}{\left[}
\newcommand{\drib}{\right]\!\!\right]}
\newcommand{\rib}{\right]}

\def\lr#1{\langle #1\rangle}

\def\ddet#1{|\!| #1 |\!|}

\def\ggp#1#2{\left[\kern-3.2pt\left[{#1\atop #2}\right]\kern-3.2pt\right]}

\def\fS{{\frak S}}

\def\fka{{\frak a}}

\def\sfz{{\mathsf z}}

\def\sfF{{\mathsf F}}
\def\sfG{{\mathsf G}}

\newcommand{\afmsD}{{\mathscr D}^\vtg}
\newcommand{\fSr}{\fS_r}
\newcommand{\affSr}{{\fS_{\vtg,r}}}

\newcommand{\afHr}{\widehat{\sH}(r)_\sZ}
\newcommand{\afbfHr}{\widehat{\boldsymbol{\mathcal H}}(r)}

\def\sB{{\mathcal B}}

\def\sH{{\mathcal H}}

\def\sO{{\mathcal O}}
\def\sP{{\mathcal P}}
\def\sQ{{\mathcal Q}}

\def\sS{{\mathcal S}}

\def\sU{{\mathcal U}}

\def\sX{{\mathcal X}}
\def\ttX{{\mathtt x}_n^+}
\def\ttY{{\mathtt x}_n^-}

\def\sZ{{\mathcal Z}}

\newcommand{\vtg}{{\!\vartriangle\!}}

\newcommand{\dbfHa}{{\boldsymbol{\mathfrak D}}_\vtg(n)}

\newcommand{\dbfHap}{{\boldsymbol{\mathfrak D}^+_\vtg}(n)}
\newcommand{\dbfHam}{{\boldsymbol{\mathfrak D}^-_\vtg}(n)}
\newcommand{\dbfHapm}{{\boldsymbol{\mathfrak D}^\pm_\vtg}(n)}

\def\field{{\mathbb F}}

\def\deg{{\rm deg}}

\newcommand{\mbzn}{\mathbb Z^{n}}
\newcommand{\mbnn}{\mathbb N^{n}}

\newcommand{\mnmod}{\!\!\!\mod\!}

\newcommand{\Lanr}{\Lambda(n,r)}
\newcommand{\tiThn}{\widetilde\Th(n)}
\newcommand{\Thn}{\Th(n)}
\newcommand{\Thnr}{\Th(n,r)}

\newcommand{\tri}{\triangle(n)}

\newcommand{\cgl}{{\frak{gl}}_n[t]}

\newcommand{\afsl}{\widehat{\frak{sl}}_n}
\newcommand{\afgl}{\widehat{\frak{gl}}_n}
\newcommand{\afglp}{\widehat{\frak{gl}}_n^+}
\newcommand{\afglm}{\widehat{\frak{gl}}_n^-}
\newcommand{\fh}{\frak{h}}

\newcommand{\bfUgln}{\bfU(\frak{gl}_n)}
\newcommand{\bfUglnp}{\bfU^+(\frak{gl}_n)}
\newcommand{\bfUglnm}{\bfU^-(\frak{gl}_n)}
\newcommand{\bfUglnz}{\bfU^0(\frak{gl}_n)}

\newcommand{\dbfUgln}{\dot{\bfU}(\frak{gl}_n)}
\newcommand{\bfUn}{\bfU(\frak{gl}_n[t])}
\newcommand{\bfUnC}{\bfU_{\mathbf v}(\frak{gl}_n[t])}
\newcommand{\afUslC}{{\bfU}_{\mathbf v}(\widehat{\frak{sl}}_n)}
\newcommand{\afUglC}{{\bfU}_{\mathbf v}(\widehat{\frak{gl}}_n)}
\newcommand{\dbfUn}{\dot{\bfU}(\frak{gl}_n[t])}
\newcommand{\dUn}{\dot U(\frak{gl}_n[t])_\sZ}
\newcommand{\bfUnp}{\bfU_{+}(\frak{gl}_n[t])}
\newcommand{\bfUnm}{\bfU_-(\frak{gl}_n[t])}
\newcommand{\bfUnz}{\bfU_0(\frak{gl}_n[t])}
\newcommand{\Un}{U(\frak{gl}_n[t])_\sZ}

\newcommand{\ddHac}{\dot{{\mathfrak D}_\vtg}(n)_\mbc}
\newcommand{\Unc}{U_\mbc(\frak{gl}_n[t])}
\newcommand{\dUnc}{\dot U_\mbc(\frak{gl}_n[t])}
\newcommand{\barUnc}{\bar U_\mbc(\frak{gl}_n[t])}

\newcommand{\bfUnr}{{\mathscr K}(n,r)}
\newcommand{\Unr}{{\mathscr K}(n,r)_\sZ}
\newcommand{\UpA}{{\mathscr K}(n,pn+\sg(A))_\sZ}

\newcommand{\msbfHr}{{\mathscr H}(r)}
\newcommand{\msHr}{{\mathscr H}(r)_\sZ}

\newcommand{\aftiLn}{{\ti{\mathscr K}(n)_{\sZ_1}}}
\newcommand{\aftiLnt}{{\ti{\mathscr K}(n)_{\sZ_2}}}
\newcommand{\afLn}{{\mathscr K}(n)_\sZ}

\newcommand{\afbfLn}{{\mathscr K}(n)}
\newcommand{\afhbfLn}{{\widehat{{\mathscr K}}(n)}}

\newcommand{\tiXin}{{\ti\Xi}(n)}
\newcommand{\Xin}{{\Xi}(n)}
\newcommand{\Xinpm}{{\Xi}^\pm(n)}
\newcommand{\Xinr}{{\Xi}(n,r)}
\newcommand{\XinpA}{{\Xi}(n,pn+\sg(A))}

\newcommand{\afg}{\nu}
\newcommand{\afFn}{{\mathscr F_\vtg}}

\newcommand{\afE}{E^\vartriangle}

\newcommand{\afSr}{\widehat{\mathcal S}(n,r)_\sZ}
\newcommand{\afSpA}{\widehat{\mathcal S}(n,pn+\sg(A))_\sZ}
\newcommand{\bfSr}{{\boldsymbol{\mathcal S}}(n,r)}

\newcommand{\afbfSr}{\widehat{\boldsymbol{\mathcal S}}(n,r)}

\newcommand{\afal}{{\boldsymbol\alpha}^\vartriangle}

\newcommand{\afbse}{\boldsymbol e^\vartriangle}
\newcommand{\afmbnn}{\mathbb N_\vtg^{n}}
\newcommand{\afmbzn}{\mathbb Z_\vtg^{n}}

\newcommand{\afLa}{\Lambda_\vtg}
\newcommand{\afLanr}{\Lambda_\vtg(n,r)}

\newcommand{\afThn}{\Theta_\vtg(n)}
\newcommand{\aftiThn}{\widetilde\Theta_\vtg(n)}

\newcommand{\afThnpm}{\Theta_\vtg^\pm(n)}
\newcommand{\afThnp}{\Theta_\vtg^+(n)}

\newcommand{\Thnpm}{\Theta^\pm(n)}
\newcommand{\Thnp}{\Theta^+(n)}
\newcommand{\Thnm}{\Theta^-(n)}
\newcommand{\afThnm}{\Theta_\vtg^-(n)}

\newcommand{\afThnr}{\Theta_\vtg(n,r)}

\newcommand{\afMnc}{M_{\vtg,n}(\mathbb C)}

\newcommand{\dzr}{\dot{\zeta}_r}

\newcommand{\ttk}{\mathtt{k}}
\newcommand{\tth}{\mathtt{h}}

\def\leq{\leqslant}\def\geq{\geqslant}
\def\le{\leqslant}\def\ge{\geqslant}

\newcommand{\Th}{\Theta}
\newcommand{\dt}{\delta}

\newcommand{\Dt}{\Delta}
\newcommand{\lm}{\longmapsto}

\newcommand{\og}{\omega}

\newcommand{\vi}{\varphi}

\newcommand{\up}{v}
 
 \newcommand{\ep}{\varepsilon}
  
 \newcommand{\al}{\alpha}
 \newcommand{\bt}{\beta}
 \newcommand{\h}{\widehat}
 \newcommand{\ti}{\widetilde}

\newcommand{\zr}{\zeta_r}

\newcommand{\sg}{\sigma}
\newcommand{\Sg}{\Sigma}
\def\th{\theta}

\newcommand{\p}{\prec}
\newcommand{\pr}{\preccurlyeq}
\newcommand{\bop}{\bigoplus}

\newcommand{\ot}{\otimes}

\newcommand{\bfl}{\mathbf{0}}

\newcommand{\Ar}{{A,r}}
\newcommand{\mpm}{\mathpzc m}
\newcommand{\mpM}{\mathpzc M}
\newcommand{\mc}{\mathcal}

\newcommand{\ol}{\overline}
\newcommand{\lra}{\longrightarrow}
\newcommand{\ra}{\rightarrow}
 \newcommand{\la}{{\lambda}}

 \newcommand{\La}{\Lambda}

 \newcommand{\mbn}{\mathbb N}
 \newcommand{\mbq}{\mathbb Q}
 \newcommand{\mbc}{\mathbb C}
 \newcommand{\mbz}{\mathbb Z}

  \newcommand{\bfd}{{\mathbf{d}}}
 
 \newcommand{\bfj}{{\mathbf{j}}}

\newcommand{\bfa}{{\boldsymbol{a}}}
\newcommand{\bfb}{{\boldsymbol{b}}}

\newcommand{\bfU}{{\mathbf{U}}}
\newcommand{\bfL}{{\mathbf{L}}}

\newcommand{\bfv}{{\mathbf{v}}}

\newcommand{\ga}{{\gamma}}
\newcommand{\Ga}{{\Gamma}}
\newcommand{\bfB}{\mathbf{B}}

\newcommand{\bfS}{\mathbf{S}}

\newcommand{\Irr}{\operatorname{Irr}}
\newcommand{\Aut}{\operatorname{Aut}}
\newcommand{\End}{\operatorname{End}}
\newcommand{\Hom}{\operatorname{Hom}}

\newcommand{\spann}{\operatorname{span}}
\newcommand{\diag}{\operatorname{diag}}

\newcommand{\Lcp}{\bar L}
\newcommand{\Mcp}{\bar M}
\newcommand{\Icp}{\bar I}

\def\ro{\text{\rm ro}}
\def\co{\text{\rm co}}
\def\wh{\widehat}

\newcommand{\bfsg}{{\boldsymbol\sigma}}

\def\afsygr{{\fS_{\vtg,r}}}

\def\ttx{{\tt x}}
\def\ttg{{\tt g}}

\def\ttk{{\tt k}}

\def\hmod{{\text-}{\mathsf{mod}}}

\newcommand{\afUgl}{{\bfU}(\widehat{\frak{gl}}_n)}

\newcommand{\afUglp}{{\bfU}^+(\widehat{\frak{gl}}_n)}
\newcommand{\afUglm}{{\bfU}^-(\widehat{\frak{gl}}_n)}
\newcommand{\afUglz}{{\bfU}^0(\widehat{\frak{gl}}_n)}
\newcommand{\afUglpz}{{\bfU}^{\geq 0}(\widehat{\frak{gl}}_n)}

\newcommand{\bfUnslC}{{\bfU}_{\mathbf v}({\frak{sl}}_n[t])}

\newcommand{\afUslCpz}{{\bfU}_{\mathbf v}^{\geq 0}(\widehat{\frak{sl}}_n)}

\newcommand{\bfP}{\mathbf{P}}
\newcommand{\bfQ}{\mathbf{Q}}

\def\dob{{[}}
\def\dcb{{]}}
\def\dop{{(}}
\def\dcp{{)}}

\newcommand{\dotBn}{\dot{\mathscr B}(n)}
\newcommand{\dotBnc}{\dot{\mathscr B}(n)_{\mathbb C}}
\newcommand{\Bnr}{{\mathscr B}(n,r)}

\begin{document}
\title{Quantum current algebra \(\bfUn\): canonical bases and polynomial representations}

\author{Qiang Fu}
\address{School of Mathematical Sciences,
Key Laboratory of Intelligent Computing and Applications (Ministry of Education),
Tongji University, Shanghai, 200092, China.}
\email{q.fu@hotmail.com, q.fu@tongji.edu.cn}


\thanks{Supported by the National Natural Science Foundation
of China (12371032, 12431002)}

\begin{abstract}
We introduce a quantum deformation $\mathbf{U}(\mathfrak{gl}_n[t])$ of the universal enveloping algebra of the current algebra $\mathfrak{gl}_n[t]$, realized as a Hopf parabolic subalgebra of quantum affine $\mathfrak{gl}_n$. Unlike the Yangian---the standard quantization of the current algebra---our algebra admits a canonical basis.
We give a BLM-type realization of \(\mathbf{U}(\mathfrak{gl}_n[t])\) via certain subalgebras of affine quantum Schur algebras, and then
construct canonical bases for the modified quantum current algebra $\dot{\mathbf{U}}(\mathfrak{gl}_n[t])$ and for its finite dimensional irreducible graded modules. Moreover, we prove a rigidity theorem: every finite dimensional polynomial irreducible module for quantum affine $\mathfrak{gl}_n$ remains irreducible when restricted to ${\bf U}_{\mathbf v}(\frak{gl}_n[t])$ (the specialization of ${\bf U}(\frak{gl}_n[t])$  at a non-root-of-unity complex number $\mathbf{v}$); conversely, every finite dimensional affine polynomial irreducible ${\bf U}_{\mathbf v}(\frak{gl}_n[t])$-module extends uniquely to a polynomial irreducible module for quantum affine $\mathfrak{gl}_n$.
More generally, we classify all finite dimensional polynomial irreducible \({\bf U}_{\mathbf v}(\frak{gl}_n[t])\)-modules by \(\mathcal Q(n)\times\Lambda^+(n)\). As an application,  we establish a bijection between
the isomorphism classes of finite dimensional polynomial irreducible modules of \({\bf U}_{\mathbf v}(\frak{gl}_n[t])\) and those of the Yangian \(Y(\mathfrak{gl}_n)\).
Thus \({\bf U}(\mathfrak{gl}_n[t])\) admits canonical bases for its modified
form and irreducible graded modules, and its finite dimensional
polynomial irreducible modules are in correspondence with those of the Yangian.
\end{abstract}
 \sloppy \maketitle

\section{Introduction}

The modified quantum group $\dot\bfU(\frak g)$ is a variant of the quantum group
$\bfU(\frak g)$  associated with a symmetrizable Kac--Moody Lie algebra $\frak{g}$. In their seminal work,
Beilinson--Lusztig--MacPherson (BLM) \cite{BLM} provided a geometric realization of quantum $\mathfrak{gl}_n$ via quantum Schur algebra. In particular, they introduced a topological construction of a canonical basis for $\dot\bfU(\frak{gl}_n)$. Subsequently,
Lusztig \cite{Lubk} extended this construction to arbitrary symmetrizable Kac--Moody Lie algebras $\frak g$, establishing canonical bases for $\dot\bfU(\frak g)$.
These bases form a cornerstone of modern representation theory, offering a powerful combinatorial and algebraic framework for understanding the structure and representations of quantum groups, with far-reaching applications in combinatorics, geometry, and mathematical physics.

A natural question arises: can one construct a canonical basis for Lie algebras that are not of Kac--Moody type? For a finite dimensional simple Lie algebra $\frak{g}$ over $\mbc$, the polynomial current algebra
$\frak{g}[t]=\frak{g}\ot\mbc[t]$  plays a significant role in quantum field theory and is intimately  connected to  various problems in mathematical physics such as the the $X = M$ conjectures (see \cite{AK,FK,Naoi}). Although current algebras are not Kac--Moody algebras, they are deeply intertwined with affine Lie algebras and quantum affine algebras. This raises the question of whether a theory of canonical bases can be developed for current algebras, a problem deeply tied to the quantization of classical current algebras.

The universal enveloping algebra of the current algebra $\mathfrak{g}[t]$ admits two important  quantizations. The first is the  {Yangian} $Y(\mathfrak{g})$, introduced by Drinfeld, which has been extensively studied due to its applications in integrable systems and mathematical physics \cite{BGX}, \cite{CPbk}. However the Yangian lacks a well-developed theory of canonical bases  that has been instrumental in understanding quantum groups and their representations. The second quantization arises from a different perspective: viewing the current algebra $\mathfrak{g}[t]$ as a parabolic subalgebra of the loop algebra $L(\mathfrak{g}) = \mathfrak{g} \otimes \mathbb{C}[t,t^{-1}]$, one may consider the corresponding parabolic subalgebra of the  {quantum affine algebra} $\mathbf{U}(L(\mathfrak{g}))$.

In this paper, we introduce and systematically study this second quantization, denoted by $\mathbf{U}(\mathfrak{gl}_n[t])$, as a quantum deformation of the universal enveloping algebra of $\mathfrak{gl}_n[t]$. It is naturally a parabolic Hopf subalgebra of quantum affine \(\mathfrak{gl}_n\), and its specialization at \(v=1\) recovers the  universal enveloping algebra of the $\cgl$.

Our main results are as follows. First, we establish a BLM-type realization of \(\mathbf{U}(\mathfrak{gl}_n[t])\) via certain subalgebras of affine quantum Schur algebras (Theorem \ref{realization of Un}). Building on this realization, we then construct canonical bases for its modified form $\dot{\mathbf{U}}(\mathfrak{gl}_n[t])$ (Theorem \ref{canonical basis for afLn}) and for its finite dimensional irreducible graded modules (Theorem  \ref{canonical basis for graded modules of quantum current algebra}).
We next study finite dimensional polynomial representations after specializing \(v\) to a non-root-of-unity complex number $\ttz$. We first prove a rigidity result for affine-polynomial irreducible modules (Theorem \ref{classification bfUnC}): every finite dimensional polynomial irreducible module for quantum affine $\mathfrak{gl}_n$ remains irreducible when restricted to $\bfUnC$; conversely, every finite dimensional affine-polynomial irreducible $\bfUnC$-module  extends uniquely to a polynomial irreducible module for quantum affine $\mathfrak{gl}_n$.
We then classify all finite dimensional polynomial irreducible  $\bfUnC$-modules.
Let \(\mathcal Q(n)\) denote the set of dominant \(n\)-tuples of polynomials and let \(\Lambda^+(n)\) denote the set of dominant polynomial weights.
We prove that the finite dimensional polynomial irreducible $\bfUnC$-modules are parametrized by
by $\mc Q(n)\times\La^+(n)$
(Theorem \ref{classification polynomial bfUnC}).
The subclass corresponding to \(\mu=0\) is precisely the affine-polynomial subclass occurring in the rigidity theorem.
Finally, using this classification together with the classification of polynomial representations of the Yangian,
we establish a bijection between the finite dimensional polynomial irreducible modules of \(\bfUnC\) and those of the Yangian \(Y(\mathfrak{gl}_n)\) (Theorem~\ref{thm:bijection}).

Our approach has several structural  advantages compared with the Yangian:
\begin{itemize}
\item
We obtain explicit formulas for the comultiplication acting on the generators of the quantum current algebra $\bfUn$ (see Proposition \ref{Hopf subalgebra}), whereas no comparably simple formula is known for the standard current-type generators of the Yangian.
\item
The quantum current algebra $\bfUn$ is naturally  a Hopf subalgebra  (in fact, a parabolic subalgebra) of quantum affine $\mathfrak{gl}_n$, whereas the Yangian is \emph{not} a Hopf subalgebra of quantum affine $\mathfrak{gl}_n$. This embedding provides a direct connection between $\bfUn$ and the
structure theory of quantum affine $\mathfrak{gl}_n$.
\item
Most importantly, the quantum current algebra $\bfUn$ inherits the combinatorial structure of the quantum affine $\frak{gl}_n$, allowing us to construct canonical bases for both the algebra and its irreducible graded modules.
The construction of a canonical basis for a quantum group relies on the existence of an integral form over  \( \mathbb{Z}[v, v^{-1}] \), where \( v \) is an indeterminate, together with a compatible bar involution.
For the Yangian, no analogous integral structure and bar involution leading to a Lusztig-type canonical basis are currently available in this generality.
Moreover, the absence of an embedding into a quantum affine algebra as a Hopf subalgebra means that the above construction of canonical bases does not directly apply to Yangians.
\end{itemize}

In addition to these advantages, our results clarify the relation
between $\bfUnC$ and the Yangian $Y(\mathfrak{gl}_n)$.
The two quantizations have the same classification of finite
dimensional polynomial irreducible modules, while $\bfUnC$
additionally admits an integral form and a canonical basis.
Our work thus provides a new combinatorial perspective on quantum current algebras.

We organize this paper as follows.  In \S2 we recall necessary background on quantum affine $\frak{gl}_n$
and the affine quantum Schur algebras $\afbfSr$.
In \S3, we introduce the quantum current algebra $\bfUn$, establish its Hopf algebra structure, and provide a presentation of $\bfUn$ by generators and relations.
In \S4,  we construct  certain subalgebras
$\bfUnr$ of affine quantum Schur algebras and establish a Schur--Weyl reciprocity  between $\bfUn$ and these algebras. We further show that, for $n\geq r$, the categories
$\bfUnr\hmod$ and $\msbfHr\hmod$ are equivalent, where $\msbfHr$ is a subalgebra of the extended affine Hecke algebra.
Using the algebras $\bfUnr$, we give a BLM realization of $\bfUn$ in Theorem \ref{realization of Un}. Let $\bfB(n,r)$ be the canonical basis of the affine quantum Schur algebra $\afbfSr$ defined by Lusztig \cite{Lu99}. In Proposition \ref{canonical basis for Unr} we prove that
the set $\Bnr:=\bfB(n,r)\cap\bfUnr$ forms a basis of $\bfUnr$. Furthermore, we prove in Theorem \ref{canonical basis for afLn} that these bases $\Bnr$  can be ``glued together'' to form a canonical basis, denoted by $\dotBn$, of the modified quantum current algebra $\dbfUn$.
In addition we show that the finite dimensional irreducible graded $\cgl$-module $\bar L(\la,m)$ admits a quantum deformation: there exists a finite dimensional irreducible  graded $\bfUn$-module $L_v(\la,m)$ such that $L_v(\la,m)$ specializes to $\bar L(\la,m)$ as $v$ tends to $1$.
Let $w_{\la,m}$ be a highest weight vector of $L_v(\la,m)$. We show in Theorem \ref{canonical basis for graded modules of quantum current algebra} that
the set $\dotBn w_{\la,m}-\{0\}$ is a $\mbq(v)$-basis of $L_v(\la,m)$.
Specializing $v$ to $1$, we obtain canonical bases for the modified enveloping algebra $\dot\sU(\cgl)$ and for the irreducible modules $\bar L(\la,m)$ (see Theorem \ref{canonical basis for current algebra} and \ref{canonical basis for bar L(la,m)}).
Finally, in \S9 we study finite dimensional polynomial
representations of $\bfUnC$. We first prove a rigidity result for
affine-polynomial irreducible modules, then classify all finite
dimensional polynomial irreducible modules by
$\mathcal Q(n)\times\Lambda^+(n)$, and establish a correspondence
with finite dimensional polynomial irreducible modules of the Yangian
$Y(\mathfrak{gl}_n)$.

We now fix some general notation. Let $\mbq(\up)$ be the fraction field
of $\sZ=\mbz[\up,\up^{-1}]$, where $\up$ is an indeterminate. For integers $N,t$ with $t\geq 0$, let
$$
\dleb{N\atop t}\drib=\prod\limits_{1\leq
i\leq t}\frac{\up^{2(N-i+1)}-1}{\up^{2i}-1}
\,\,\text{ and }\,\,\leb{N\atop t}\rib=\up^{-t(N-t)}\dleb{N\atop t}\drib.
$$
For a positive integer $n$, let
 $\afMnc$  be the set of all $\mbz\times\mbz$ complex matrices
$A=(a_{i,j})_{i,j\in\mbz}$ with $a_{i,j}\in\mbc$ such that
\begin{itemize}
\item[(a)]$a_{i,j}=a_{i+n,j+n}$ for $i,j\in\mbz$, and \item[(b)] for
every $i\in\mbz$, the set $\{j\in\mbz\mid a_{i,j}\not=0\}$ is finite.
\end{itemize}
Furthermore let $\afThn=\{A\in\afMnc\mid a_{i,j}\in\mbn,\,\forall i,j\}$ and
$\aftiThn=\{A\in\afMnc\mid a_{i,j}\in\mbn,\,a_{i,i}\in\mbz,\,\forall i\not=j\}.$

\section{The quantum affine $\frak{gl}_n$ and affine quantum Schur algebras}
\subsection{The quantum affine $\frak{gl}_n$}\label{quantum affine gln}

For a positive integer $n$, let $\frak{gl}_n$ be the complex general linear Lie algebra, and let
$\afgl:={\frak{gl}_n}\ot\mbc[t,t^{-1}]$
be the loop algebra of $\frak{gl}_n$.
The set $\{E_{i,j}\ot t^m\mid 1\le i,j\le n,\,m\in\mbz\}$ forms
a $\mbc$-basis of $\afgl$,
where $E_{i,j}$ is the $n\times n$ matrix $(\delta_{k,i}\delta_{j,l})_{1\le
k,l\le n}$. Clearly we have
the following triangular decomposition
\begin{equation}\label{tri decomposition}
\afgl=\afglp\oplus\fh\oplus\afglm
\end{equation}
where $\afglp=\spann\{E_{i,j}\ot t^s\mid 1\leq i,j\leq n,\,s>0\}\oplus\spann\{E_{i,j}\mid 1\leq i<j\leq n\}$, $\fh=\spann\{E_{i,i}\mid 1\leq i\leq n\}$, and
$\afglm=\spann\{E_{i,j}\ot t^s\mid 1\leq i,j\leq n,\,s<0\}\oplus\spann\{E_{i,j}\mid 1\leq j<i\leq n\}$.

For $i,j\in\mbz$, let $\afE_{i,j}\in\afMnc$ be the matrix
$(e^{i,j}_{k,l})_{k,l\in\mbz}$ defined by
\begin{equation*}e_{k,l}^{i,j}=
\begin{cases}1,&\text{if $k=i+sn,l=j+sn$ for some $s\in\mbz$;}\\
0,&\text{otherwise}.\end{cases}
\end{equation*}
Clearly the map
$$\afMnc\lra\afgl,\,\,\,\afE_{i,j+ln}\longmapsto E_{i,j}\ot t^l, \,1\le i,j\le n,l\in\mbz $$
is a Lie algebra isomorphism. We will identify the loop algebra
$\afgl$ with $\afMnc$ in the sequel.

Let $C=(c_{i,j})$ be the Cartan matrix of affine type $A_{n-1}$. We recall the Drinfeld's new realization of quantum affine $\frak{gl}_n$ as follows.

\begin{Def}\label{QLA}
 The {\it quantum loop algebra} $\afUgl$  (or {\it
quantum affine $\mathfrak {gl}_n$})  is the $\mbq(v)$-algebra generated by $\ttx^\pm_{i,s}$
($1\leq i<n$, $s\in\mbz$), $\ttk_i^{\pm1}$ and $\ttg_{i,t}$ ($1\leq
i\leq n$, $t\in\mbz\backslash\{0\}$) with the following relations:
\begin{itemize}
 \item[(QLA1)] $\ttk_i\ttk_i^{-1}=1=\ttk_i^{-1}\ttk_i,\,\;[\ttk_i,\ttk_j]=0$,
 \item[(QLA2)]
 $\ttk_i\ttx^\pm_{j,s}=\ttv^{\pm(\dt_{i,j}-\dt_{i,j+1})}\ttx^\pm_{j,s}\ttk_i,\;
               [\ttk_i,\ttg_{j,s}]=0$,
 \item[(QLA3)] $[\ttg_{i,s},\ttx^\pm_{j,t}]
               =\begin{cases}0,\;\;&\text{if $i\not=j,\,j+1$};\\
                  \pm \ttv^{-js}\frac{[s]}{s}\ttx^\pm_{j,s+t},\;\;\;&\text{if $i=j$};\\
                  \mp \ttv^{-js}\frac{[s]}{s}\ttx_{j,s+t}^\pm,\;\;\;&\text{if $i=j+1$,}
                \end{cases}$
 \item[(QLA4)] $[\ttg_{i,s},\ttg_{j,t}]=0$,
 \item[(QLA5)]
 $[\ttx_{i,s}^+,\ttx_{j,t}^-]=\dt_{i,j}\frac{\phi^+_{i,s+t}
 -\phi^-_{i,s+t}}{\ttv-\ttv^{-1}}$,
 \item[(QLA6)] $\ttx^\pm_{i,s}\ttx^\pm_{j,t}=\ttx^\pm_{j,t}\ttx^\pm_{i,s}$, for $|i-j|>1$, and
 $[\ttx_{i,s+1}^\pm,\ttx^\pm_{j,t}]_{\ttv^{\pm c_{ij}}}
               =-[\ttx_{j,t+1}^\pm,\ttx^\pm_{i,s}]_{\ttv^{\pm c_{ij}}}$,
 \item[(QLA7)]
 $[\ttx_{i,s}^\pm,[\ttx^\pm_{j,t},\ttx^\pm_{i,p}]_\ttv]_\ttv
 =-[\ttx_{i,p}^\pm,[\ttx^\pm_{j,t},\ttx^\pm_{i,s}]_\ttv]_\ttv\;$ for
 $|i-j|=1$,
\end{itemize}
 where $[x,y]_a=xy-ayx$, $[s]=\frac{v^s-v^{-s}}{v-v^{-1}}$  and $\phi_{i,s}^\pm$ are defined via the
 generating functions in indeterminate $u$ by
$$\Phi_i^\pm(u):={\ti\ttk}_i^{\pm 1}
\exp\bigl(\pm(\ttv-\ttv^{-1})\sum_{m\geq 1}\tth_{i,\pm m}u^{\pm
m}\bigr)=\sum_{s\geq 0}\phi_{i,\pm s}^\pm u^{\pm s}$$ with
$\ti\ttk_i=\ttk_i/\ttk_{i+1}$ ($\ttk_{n+1}=\ttk_1$) and $\tth_{i,\pm
m}=\ttv^{\pm(i-1)m}\ttg_{i,\pm m}-\ttv^{\pm(i+1)m}\ttg_{i+1,\pm
m}\,(1\leq i<n).$
\end{Def}
Let $\afUsl$ be
the subalgebra of $\afUgl$ generated by all $\ttx^\pm_{i,s}$, $\ti\ttk_i^{\pm1}$ and $\tth_{i,t}$ for $1\leq i<n$, $s\in\mbz$ and $t\in\mbz\backslash\{0\}$.
For $s\geq 1$ and $1\leq j<n$ let
\begin{equation}\label{ths}
\th_{\pm s} =\mp\frac1{[s]}(\ttg_{1,\pm s}+\cdots+\ttg_{n,\pm s}),\quad
\ttx_j^+=\ttx_{j,0}^+,\quad\ttx_j^-=\ttx_{j,0}^-
\end{equation}
Furthermore let
\begin{equation}\label{ttX}
\begin{split}
\ttX& =v[\ttx_{n-1,0}^-,[\ttx_{n-2,0}^-,\cdots,
[\ttx_{2,0}^-,\ttx_{1,1}^-]_{\ttv^{-1}}\cdots
]_{\ttv^{-1}}]_{\ttv^{-1}}\ti\ttk_n,\\
\ttY& =v^{-1}\ti\ttk_n^{-1}[\cdots[[\ttx_{1,-1}^+,\ttx_{2,0}^+]_\ttv,\ttx_{3,0}^+]_\ttv,
 \cdots,\ttx_{n-1,0}^+]_\ttv.
 \end{split}
\end{equation}
The algebra $\afUgl$ is generated by the elements  $\ttx_{i}^+$, $\ttx_{i}^-$, $\ttk_i^{\pm 1}$ and $\th_{\pm s}$ for $1\leq i\leq n$ and $s\geq 1$.
The following result was given in \cite[Cor. 2.3.5]{DDF}.
\begin{Prop}\label{Hopf algebra}
The algebra $\afUgl$ is a Hopf algebra with
comultiplication $\Dt$, counit $\ep$, and antipode $\sg$ defined
by
\begin{eqnarray*}\label{Hopf}
&\Delta(\ttx_i^+)=\ttx_i^+\otimes\ti \ttk_i+1\otimes
\ttx_i^+,\quad\Delta(\ttx_i^-)=\ttx_i^-\otimes
1+\ti \ttk_i^{-1}\otimes \ttx_i^-,&\\
&\Delta(\ttk^{\pm 1}_i)=\ttk^{\pm 1}_i\otimes \ttk^{\pm 1}_i,\quad
\Delta(\th_{\pm s})=\th_{\pm s}\otimes1+1\otimes
\th_{\pm s};&\\
&\ep(\ttx_i^+)=\ep(\ttx_i^-)=0=\ep(\th_{\pm s}),
\quad \ep(\ttk_i)=1;&\\
&\sg(\ttx_i^+)=-\ttx_i^+\ti\ttk_i^{-1},\quad \sg(\ttx_i^-)=-\ti\ttk_i\ttx_i^-,\quad
\sg(\ttk^{\pm 1}_i)=\ttk^{\mp 1}_i,&\\
&\text{and}\;\;\sg(\th_{\pm s})=-\th_{\pm s},&
\end{eqnarray*}
where $1\leq i\leq n$ and $s\in \mbz^+$.
\end{Prop}

Let $\afUglp$ (resp. $\afUglm$) be the subalgebra of $\afUgl$ generated by the elements $\ttx_{i}^+$, $\th_s$ (resp. $\ttx_{i}^-$, $\th_{-s}$) for $1\leq i\leq n$ and $s\geq 1$. Furthermore let $\afUglz$ be the subalgebra of $\afUgl$ generated by the elements $\ttk_i$ for $1\leq i\leq n$. Then we have $\afUgl\cong\afUglp\ot\afUglz\ot\afUglm$.

Let $\bfUgln$ be the subalgebra of $\afUgl$ generated by the elements
$\ttx_{j}^+$, $\ttx_{j}^-$ and $\ttk_i^{\pm 1}$ for $1\leq j<n$ and $1\leq i\leq n$.
Then $\bfUgln$ is the quantized enveloping algebra of $\frak{gl}_n$.
Let $\bfUglnp=\bfUgln\cap\afUglp$, $\bfUglnm=\bfUgln\cap\afUglm$ and
$\bfUglnz=\afUglz$. Then we have $\bfUgln\cong\bfUglnp\ot\bfUglnz\ot\bfUglnm$.

\subsection{The double Ringel--Hall algebra $\dbfHa$}

Let $\tri$ ($n\geq 1$) be
the quiver of type $\ti A_{n-1}$
with vertex set $I=\mbz/n\mbz=\{1,2,\ldots,n\}$ and arrow set
$\{i\to i+1\mid i\in I\}$. Note that $\ti A_{0}$ has one vertex and one loop.

Let $\field_q$ be a finite field. For $i\in I$ let $S_i$
denote the irreducible representation of $\tri$ over $\field_q$ with $(S_i)_i=\field_q$ and $(S_i)_k=0$ for $i\neq k$.
Let
\begin{equation}\label{afThnp}
\begin{split}
\afThnp&=\{A\in\afThn\mid a_{i,j}=0\text{ for }i\geq j\},
\end{split}
\end{equation}
For $A=(a_{i,j})\in\afThnp$, let
\begin{equation}\label{isoclass}
M(A) =M_{\field_q}(A)=\bop_{1\leq i\leq n,i<j}a_{i,j}M^{i,j},
\end{equation}
where $M^{i,j}$ is the unique indecomposable nilpotent representation for $\tri$ of dimension $j-i$ with top $S_i$.

Let $\afmbzn=\{(\la_i)_{i\in\mbz}\mid
\la_i\in\mbz,\,\la_i=\la_{i-n}\ \text{for}\ i\in\mbz\}$ and $\afmbnn=\{(\la_i)_{i\in\mbz}\in \afmbzn\mid \la_i\ge0\text{ for  }i\in\mbz\}.$ We will identify $\afmbzn$ with $\mbz I=\mbz^n$ via the  natural bijection $\flat:\afmbzn\lra\mbz^n$
defined by sending $\bfj$ to $\flat(\bfj)=(j_1,\cdots,j_n).$
Define an order relation $\leq$ on $\afmbzn$ by
\begin{equation*}\label{order on afmbzn}
\la\leq\mu  \iff\la_i\leq \mu_i\,(1\leq i\leq n).
\end{equation*}
We say that $\la<\mu$ if $\la\leq\mu$ and $\la\not=\mu$.

For $\la\in\afmbnn$ let $$S_\la=\sum_{1\leq i\leq n}\la_i\afE_{i,i+1}\in\afThnp.$$
Then $M(S_\la)=\oplus_{1\leq i\leq n}\la_i S_i$ is the semisimple representation of the cyclic quiver $\tri$.

By \cite{Ri93}, for $A,B,C\in\afThnp$,
there exists $\vi^{C}_{A,B}\in\mbz[\up^2]$  such
that, for any finite field $\field_q$,
$\vi^{C}_{A,B}|_{\up^2=q}$ is equal to the number of submodules $N$ of
$M_{\field_q}(C)$ satisfying $N\cong M_{\field_q}(B)$ and $M_{\field_q}(C)/N\cong M_{\field_q}(A)$.
For each $A=(a_{i,j})\in \afThnp$, there is a polynomial
$\fka_A=\fka_A(\up^2)\in\sZ$ in $\up^2$ such that, for each finite
field $\field$ with $q$ elements,
$\fka_A|_{\up^2=q}=|\Aut(M_{\field}(A))|$ (cf. \cite[Cor.~2.1.1]{Peng}).

For $A\in\afThnp$ let $\bfd(A)\in\mbz I=\mbz^n$ be the dimension vector of $M(A)$.
For $\bfa,\bfb \in\mbz I$, the  Euler form  associated with $\tri$ is the bilinear form $\lan-,-\ran:\mbz I\times\mbz I\ra\mbz$ defined by
\begin{equation*}
\lan \bfa,\bfb\ran=\sum_{i\in I}a_ib_i-\sum_{i\in I}a_ib_{i+1}.
\end{equation*}

Let $\dbfHa$ be the double Ringel--Hall algebra of $\tri$ over $\mbq(v)$
(cf. \cite[(2.1.3.2)]{DDF}). By \cite[2.6.1, 2.6.3(5) and 3.9.2]{DDF} we obtain the following.
\begin{Lem} \label{presentation-dbfHa}
The algebra $\dbfHa$ is the algebra over $\mbq(\up)$ generated by
$u_A^+$, $K_{i}^{\pm 1}$, $u_A^-$ $(A\in\afThnp,\,i\in I)$ subject to
the following relations:
\begin{itemize}
\item[(1)]
$K_iK_j=K_jK_i$, $K_iK_i^{-1}=K_i^{-1}K_i=1$, $u_0^+=u_0^-=1$;
\item[(2)]
$K\su{\bfj} u_A^+=\up^{\lr{\bfd(A),\bfj}}u_A^+K\su\bfj$,
$u_A^-K\su\bfj=\up^{\lr{\bfd(A),\bfj}}K\su\bfj u_A^-$, where
$K\su\bfj=K_1^{j_1}\cdots K_n^{j_n}$ for $\bfj\in\afmbzn$;
\item[(3)]
$u_A^+u_B^+=\sum_{C\in\afThnp}\up^{\lan \bfd(A),\bfd(B)\ran}\vi_{A,B}^C u_C^+$;
\item[(4)]
$u_A^-u_B^-=\sum_{C\in\afThnp}\up^{\lan \bfd(B),\bfd(A)\ran}\vi_{B,A}^C u_C^-$;
\item[(5)] {\rm commutator relations}:  for all $\la,\mu\in\afmbnn$,
\begin{equation*}
\aligned
\up^{\lan\mu,\mu\ran}&\sum_{\al,\bt\in\afmbnn\atop\la-\al=\mu-\bt\geq 0}\vi_{\la,\mu}^{\al,\bt}
\up^{\lan \bt,\la+\mu-\bt\ran}\ti K\su{\mu-\bt}u_{S_\bt}^-u_{S_\al}^+
=\up^{\lan\mu,\la\ran}\sum_{\al,\bt\in\afmbnn\atop\la-\al=\mu-\bt\geq 0}
{\vi_{\la,\mu}^{\al,\bt}}\up^{\lan \mu-\bt,\al\ran+\lan \mu,\bt\ran}
\ti K\su{\bt-\mu}u_{S_\al}^+u_{S_\bt}^-,\endaligned
\end{equation*}
\end{itemize}
where $\ti K\su\nu :=(\ti K_1)^{\nu_1}\cdots(\ti K_n)^{\nu_n}$ with $\ti K_i=K_iK_{i+1}^{-1}$ for $\nu\in\afmbzn$, and
\begin{equation*}
\vi_{\la,\mu}^{\al,\bt}=\up^{2\sum_{1\leq i\leq n}(\la_i-\al_i)(1-\al_i-\bt_i)}\prod_{1\leq i\leq n\atop 0\leq s\leq\la_i-\al_i-1}\frac{1}{\up^{2(\la_i-\al_i)}-\up^{2s}}.
\end{equation*}
\end{Lem}

Let $\dbfHap$ (resp., $\dbfHam$) be the $\mbq(\up)$-subalgebra
of $\dbfHa$ generated by $u_A^+$ (resp., $u^-_A$) for all
$A\in\afThnp$.
For $m\geq 1$, let
\begin{equation}\label{central elements-pm}
c_m^\pm=(-1)^m \up^{-2nm}\sum_{A}(-1)^{{\rm dim}\,{\rm
End}(M(A))}\fka_A u^\pm_A \in\dbfHa ,\end{equation}
where the sum is taken over all $A\in \afThnp$ such that $\bfd (A)=m\dt$ and
${\rm soc}\,M(A)$ is square-free. By \cite{Hub}, the elements $c^+_m$
and $c^-_m$ are central in $\dbfHap$ and $\dbfHam$,
respectively.
Following \cite[\S4]{Hub}, let $C^\pm(u)=1+\sum_{m\geq
1}c^\pm_mu^m$ be the generating functions in indeterminate $u$ associated with the sequence $\{c_m^\pm\}_{m\geq1}$
and define elements $x_m^\pm$ by
$$X^\pm(u)=\sum_{m\geq 1}x_m^\pm u^{m-1}=\frac{\rm d}{{\rm d}u}{\rm log\,}C^\pm(u)=\frac1{C^\pm(u)}\frac {\rm d}{{\rm d}u}C^\pm(u).$$
We further set
$$\sfz_m^\pm=\frac{\up^{nm}}{\up^m-\up^{-m}}x_m^\pm\in\dbfHapm,\;\;\;\text{for $m\geq
1$}.$$

For $i\in I$ let $$E_i=u_{\afE_{i,i+1}}^+,\quad F_i=u_{\afE_{i,i+1}}^-.$$
Beck \cite{Be} proved that $\afUsl$ is isomorphic to   the subalgebra of $\dbfHa$ generated by $E_i$,  $F_i$,   $\ti K_i^{\pm 1}$ for $i\in I$.
The following result was given in {\cite[Prop. 4.4.1]{DDF}}.
\begin{Prop}  \label{DDFIsoThm}
There is a Hopf  algebra isomorphism
 $$\mathsf F:\dbfHa\lra
\afUgl$$ such that
$$\aligned
&K_i^{\pm1}\lm\ttk_i^{\pm1},\;\quad
E_i\lm \ttx^+_{i},\;\quad
F_i\lm \ttx^-_{i}
\;(1\leq i\leq n),\;\\
& \quad
\sfz^\pm_s\lm
\mp s\ttv^{\pm s}\th_{\pm
s}\;(s\geq 1),
\endaligned$$
where $\th_{\pm s}$ and $\ttx_i^\pm$ are defined in \eqref{ths} and \eqref{ttX}.
\end{Prop}
We will identify $\dbfHa$ with $\afUgl$, and hence identify $E_i$ with $\ttx_i^+$, etc., in the sequel.

Let $(\mbn_\vtg^*)^n=\{\la\in\afmbnn\mid\la_i>0,\,\forall i\in\mbz\}$.
For $1\leq i\leq n$ let  $\afbse_i\in\afmbnn$ be the element
satisfying $(\afbse_i)_j=\dt_{\bar i,\bar j}$ for $j\in\mbz$. Here
$\bar i$ is the congruence class of $i$ modulo $n$.
Let
$$\widetilde I=\{\afbse_1,\afbse_2,\cdots,\afbse_n\}\cup(\mbn_\vtg^*)^n.$$
Let $\ti\Sg$ be the set of words on the alphabet $\ti I$.
Any word $w=\bsa_1\bsa_2\cdots\bsa_m$ in $\ti\Sg$ can be uniquely expressed in the {\it tight form} $w=\bsb_1^{x_1}\bsb_2^{x_2}\cdots\bsb_t^{x_t}$ where $x_i=1$ if $\bsb_i\in(\mbn_\vtg^*)^n$, and $x_i$ is the number of consecutive occurrences of $\bsb_i$ if $\bsb_i\in\{\afbse_1,\afbse_2,\cdots,\afbse_n\}$.

For $A\in\afThnp$, let $$\ti u_A^\pm=\up^{\dim \End(M(A))-\dim M(A)}u_A^\pm.$$
For $\la\in\afmbnn$ let $\ti u_\la^\pm=\ti u_{S_\la}^\pm$.
For $w=\bsa_1\bsa_2\cdots\bsa_m\in\ti\Sg$ with the tight form
$\bsb_1^{x_1}\bsb_2^{x_2}\cdots\bsb_t^{x_t}$, let
\begin{equation*}
\ti u_{(w)}^\pm =\ti u_{x_1\bsb_1}^\pm\ti u_{x_2\bsb_2}^\pm\cdots\ti u_{x_t\bsb_t}^\pm\in\dbfHa.
\end{equation*}

Following \cite[3.5]{BLM} we may define the order relation $\pr$
on $\aftiThn$ as follows.
For $A\in\aftiThn$ and $i\not=j\in\mbz$, let\vspace{-2ex}
$$\sg_{i,j}(A)=\begin{cases}\sum\limits_{s\leq i,t\geq j}a_{s,t},\;&\text{ if $i<j$};\\
\sum\limits_{s\geq i,t\leq j}a_{s,t},\; & \text{ if
$i>j$}.\end{cases}\vspace{-2ex}$$
 For $A,B\in\aftiThn$, define
\begin{equation}\label{order pref}
B\pr A \text{ if and only if } \sg_{i,j}(B)\leq\sg_{i,j}(A) \text{ for all } i\not=j.
\end{equation}
Put $B\p A$ if $B\pr A$ and, for some pair $(i,j)$ with $i\not=j$,
$\sg_{i,j}(B)<\sg_{i,j}(A)$.

The following result is given in \cite[(9.2)]{DDX} (cf. \cite[2.2]{DF18}).

\begin{Prop}\label{tri Hall}
For $A\in\afThnp$, there exist $w_{A}\in\ti\Sg$ such that
\begin{equation}\label{eq tri Hall}
\ti u_{(w_{A})}^+=\ti u_A^++\sum_{B\in\afThnp\atop B\prec A,\,\bfd(A)=\bfd(B)}f_{{B,A}}\ti u_B^+.\end{equation}
where $f_{{B,A}}\in\sZ$.
\end{Prop}

Let
 $$\afThnm=\{A\in\afThn\mid a_{i,j}=0\text{ for }i\leq j\}.$$
Furthermore, let
\begin{equation}\label{Thnp}
\Thnp=\afThnp\cap\Thn,\,\Thnm=\afThnm\cap\Thn,
\end{equation}
where
\begin{equation}\label{Thn}
\Thn=\{A\in\afThn\mid a_{i,j}=0 \text{ for }1\leq i\leq n,\,j\not\in\{1,2,\cdots,n\}\}.
\end{equation}
Let $\Ga$ be the set of words on the alphabet $\{\afbse_1,\afbse_2,\cdots,\afbse_{n-1}\}$. Then $\Ga$ is a subset of $\ti\Sg$. By \cite[5.5(c)]{BLM} we have the following result.
\begin{Prop}\label{tri quantum gln}
For $A\in\Thnp$, there exist $w_{A}\in\Ga$ such that
\begin{equation}\label{eq tri Hall}
\ti u_{(w_{A})}^-=\ti u_A^-+\sum_{B\in\Thnp\atop B\prec A,\,\bfd(A)=\bfd(B)}f_{{B,A}}\ti u_B^-.\end{equation}
where $f_{{B,A}}\in\sZ$.
\end{Prop}

\subsection{The affine quantum Schur algebra $\afbfSr$}
We now  recall the geometric definition of affine
quantum Schur algebras introduced in \cite{GV,Lu99}.
Let $V$ be a free $\field[\ep,\ep^{-1}]$ module of rank $r$,
where
$\field$ is a field and $\ep$ is an indeterminate.
Let $\afFn={\mathscr F}_{\vtg,n}$ be the set of all collections $\bfL=(L_i)_{i\in\mbz}$, where each $L_i$ is a lattice in $V$ such that
$L_{i-1}\han L_i$ and $L_{i-n}=\ep L_i$, for all $i\in\mbz$.
Let $G$ be the group of automorphisms of the $\field[\ep,\ep^{-1}]$-module $V$.
The group $G$ acts on $\afFn\times\afFn$  by
$g: (\bfL,\bfL')\mapsto((g(L_i))_{i\in\mbz},(g(L_i'))_{i\in\mbz})$.

For $A\in\aftiThn$ and $r\in\mbn$,
let
$\sg(A)=\sum_{1\leq i\leq n,\,
j\in\mbz}a_{i,j}$ and let
\begin{equation*}\label{afThnr}
\afThnr=\{A\in\afThn\mid\sg(A)=r\}.
\end{equation*}
By \cite[1.5]{Lu99}
there is a bijection between the set of $G$-orbits in
$\afFn\times\afFn$ and the matrix set $\afThnr$ by sending $(\bfL,\bfL')$
to $A=(a_{i,j})_{ij\in\mbz}$,
where $$a_{i,j}=\dim_\field\frac{L_i\cap L_j'}{L_{i-1}\cap L_j'+L_i\cap L_{j-1}'}.$$
For $A\in\afThnr$, let $\sO_A\han\afFn\times\afFn$ be
the $G$-orbit corresponding to $A$.

Assume now that $\field=\field_q$ is the finite field of $q$
elements. For any fixed $(\bfL,\bfL'')\in\sO_{A''}$ let
\begin{equation*}
\afg_{A,A',A'';q}=|\{\bfL'\in\afFn \mid (\bfL,\bfL')
\in\sO_{A},(\bfL',\bfL'')\in\sO_{A'}\}|
\end{equation*}
By \cite[1.8]{Lu99}, there exists a polynomial
$\nu_{A,A',A''}\in\sZ$  in $\up^2$ such that, for
each finite field $\field$ with $q$ elements,
$\nu_{A,A',A'';q}=\nu_{A,A',A''}|_{\up^2=q}$.

For $A\in\aftiThn$, let
$\ro(A)=\bigl(\sum_{j\in\mbz}a_{i,j}\bigr)_{i\in\mbz}$ and
$\co(A)=\bigl(\sum_{i\in\mbz}a_{i,j}\bigr)_{j\in\mbz}.$
Let $\afSr$ be the free $\sZ$-module with
basis $\{e_{A}\mid A\in\afThnr\}$. There is a unique associative $\sZ$-algebra structure on $\afSr$ with multiplication
\begin{equation*}
e_{A}e_{{A'}}=
\begin{cases}
\sum_{A''\in\afThnr}\nu_{A,A',A''}e_{{A''}},&
\text{if $\co(A)=\ro(A')$;}\\
0,&\text{otherwise.}
\end{cases}\end{equation*}
Let $\afbfSr=\afSr\ot_\sZ\mbq(v)$. The algebras $\afSr$ and $\afbfSr$ are called affine quantum Schur algebras (see \cite{GV,Gr99,Lu99}).

\subsection{The algebra homomorphism $\zr$}
The quantum affine $\frak{gl}_n$ and the affine quantum Schur algebra $\afbfSr$ are
related by a surjective algebra homomorphism $\zeta_r$, which we now describe.

For $A\in\afThnr$ let
\begin{equation*}
[A]=\up^{-d_A}e_{A},\quad\text{ where } \quad
d_{A}=\sum_{1\leq i\leq n\atop i\geq k,j<l}a_{i,j}a_{k,l}.
\end{equation*}
For $r\in\mbn$  we set
\begin{equation*}
\afLanr=\{\la\in\afmbnn\mid\sg(\la)=r\},
\end{equation*}
where
$\sg(\la)=\sum_{1\leq i\leq n}\la_i$.
Let
\begin{equation*}\label{afThnpm}
\afThnpm=\{A\in\afThn\mid a_{i,i}
=0\text{ for all $i$}\}\end{equation*}
For $A\in\afThnpm$, $\bfj\in\afmbzn$, let
\begin{equation*}\label{A(j,la)}
\begin{split}
A(\bfj,r)&=\sum_{\mu\in\afLa(n,r-\sg(A))}\up^{\mu\centerdot\bfj}
[A+\diag(\mu)]\in\afbfSr,
\end{split}
\end{equation*}
where $\mu\centerdot\bfj=\sum_{1\leq i\leq n}\mu_ij_i.$
The following result was given in \cite[3.6.3, 3.8.1]{DDF}.
\begin{Thm}  \label{zr}
For $r\in\mbn$, there is a surjective algebra homomorphism $\zr:\afUgl\ra \afbfSr$ such that
$$\zr(K^\bfj)=0(\bfj,r),\;\zr(\ti u_A^+)=A(\bfl,r),\;\;\text{and}\;\;
\zr(\ti u_A^-)=(\tA)(\bfl,r),$$
for all $\bfj\in \afmbzn$, $A\in \afThnp$, where $\tA$ is the transpose of $A$.
\end{Thm}

\section{The quantum current algebra $\bfUn$}
\subsection{The quantum current algebra $\bfUn$}

For a positive integer $n$ let
$\cgl=\frak{gl}_n\ot\mbc[t]$
be the current algebra of $\frak{gl}_n$. The current algebra $\cgl$ is a parabolic subalgebra of $\afgl$.
We have the following direct sum decomposition
$$\cgl=\afglp\oplus\fh\oplus\glm$$
where $\afglp$, $\fh$ are defined in \eqref{tri decomposition},
and $\glm=\spann\{E_{i,j}\mid 1\leq j<i\leq n\}$.

We now introduce the the quantum current algebra of $\frak{gl}_n$ as a certain parabolic subalgebra of quantum affine $\frak{gl}_n$ as follows.
Let $\afUglpz=\afUglp\afUglz$ be the Borel subalgebra of $\afUgl$. Furthermore
let $$\bfUn=\afUglpz\bfUglnm=\afUglp\afUglz\bfUglnm\han\afUgl.$$
Then $\bfUn$ is the subalgebra of $\afUgl$ generated by the elements
$\ttx_{i}^+$, $\ttx_{j}^-$, $\ttk_i^{\pm 1}$ and $\th_{s}$ for  $1\leq i\leq n$, $1\leq j<n$ and $s\geq 1$. By Proposition \ref{DDFIsoThm}, the algebra $\bfUn$ is generated by the elements $E_i$, $F_j$, $K_i^{\pm 1}$ and $\sfz_{s}^+$ for  $1\leq i\leq n$, $1\leq j<n$ and $s\geq 1$.
We refer to $\bfUn$ as the quantum current algebra of $\frak{gl}_n$.

Recall the notation $\afThnp$ defined in \eqref{afThnp} and the notation $\Thnp$ defined in \eqref{Thnp}.
By Proposition \ref{DDFIsoThm} the algebra $\bfUn$ is  spanned by the elements
$u_A^+K^\bfj u_B^-$ for $A\in\afThnp$, $\bfj\in\afmbzn$ and $B\in\Thnp$.
For $i\in I$ and $t\in\mbn$, let
$$\bigg[{K_i;0 \atop t}\bigg]=
\prod_{s=1}^t \frac{K_i\up^{-s+1}-K_i^{-1}\up^{s-1}}{\up^s-\up^{-s}}.$$
Let $\Un$ be the $\sZ$-submodule of $\bfUn$ spanned by the elements
$u_A^+\prod_{1\leq i\leq n}K_i^{j_i}\big[{K_i;0 \atop \la_i}\big] u_B^-$  for $A\in\afThnp$, $\bfj\in\afmbzn$, $\la\in\afmbnn$ and $B\in\Thnp$. We will prove in Theorem \ref{realization of Un} that
$\Un$ is a $\sZ$-subalgebra of $\bfUn$.

Let $\sU(\cgl)$ be the universal enveloping algebra of the current algebra $\cgl$.
We will prove in Proposition \ref{vi} that
$$\sU(\cgl)\cong\Un\ot_\sZ\mbc/\lan K_i -1 \mid 1\leq i\leq n\ran,$$
where $\mbc$ is regarded as a $\sZ$-module by specializing $v$ to $1$.

Clearly, we have the following result.
\begin{Lem}\label{basis of bfUn}
The set $\{u_A^+K^\bfj u_B^-\mid A\in\afThnp,\,\bfj\in\afmbzn,\,B\in\Thnp\}$ forms a $\mbq(v)$ basis of $\bfUn$.
\end{Lem}

\begin{Lem}
The algebra $\bfUn$ is generated by the elements $K_i^{\pm 1}$, $u^+_{\afE_{i,j}}$, $u^-_{\afE_{k,l}}$ for $1\leq i\leq n$, $j\in\mbz$ with $i<j$, and $1\leq k<l\leq n$.
\end{Lem}
\begin{proof}
Let $\bfUnp$ be the subspace of $\bfUn$
spanned by the elements $u_A^+$ for $A\in\afThnp$.
Let $\bfUnm$ be the subspace of $\bfUn$
spanned by the elements  $u_B^-$ for $B\in\Thnp$. Let  $\bfUnz$ be the subalgebra of $\bfUn$
generated by the elements $K_i^{\pm 1}$ for $1\leq i\leq n$. Then we have $$\bfUn=\bfUnp\bfUnz\bfUnm.$$
By \cite[1.4.5]{DDF} we know  that
$\bfUnp$ is the subalgebra of $\bfUn$ generated by the elements $u^+_{\afE_{i,j}}$ for $i,j\in\mbz$ with $i<j$. Furthermore, $\bfUnm$ is the subalgebra of $\bfUn$
generated by the elements $u^-_{\afE_{i,j}}$ for $1\leq i<j\leq n$.
The assertion follows.
\end{proof}

\subsection{A presentation of $\bfUn$}
 We now describe a presentation for
$\bfUn$ as follows.

\begin{Prop}\label{presentation dHallAlg}
If $n\geq 2$, then the  algebra $\bfUn$ is the $\mbq(v)$-algebra generated by
$E_i,\ F_j,\  K_i,\ K_i^{-1},\ \sfz^+_m,$ $i\in I$, $1\leq j\leq n-1$,
$m\in\mbz^+$  with relations:
\begin{itemize}
\item[(1)]
$K_{i}K_{j}=K_{j}K_{i},\ K_{i}K_{i}^{-1}=K_{i}^{-1}K_i=1$,
$K_{i}E_j=\up^{\dt_{i,j}-\dt_{i,j+1}}E_jK_{i}$ for $i,j\in I$;

\item[(2)]
$K_{i}F_j=\up^{-\dt_{i, j}+\dt_{ i,j+1}} F_jK_i$,
$E_iF_j-F_jE_i=\delta_{i,j}\frac
{\ti K_{i}-{\ti K_{i}}^{-1}}{\up-\up^{-1}}$ for $i\in I$, $1\leq j\leq n-1$;

\item[(3)]
$\displaystyle\sum_{a+b=1-c_{i,j}}(-1)^a\leb{1-c_{i,j}\atop a}\rib
E_i^{a}E_jE_i^{b}=0$ for $i\not=j\in I$;

\item[(4)]
$\displaystyle\sum_{a+b=1-c_{i,j}}(-1)^a\leb{1-c_{i,j}\atop a}\rib
F_i^{a}F_jF_i^{b}=0$ for $1\leq i\not=j\leq n-1$;

\item[(5)]
$\sfz^+_mK_i=K_i\sfz^+_m$, $\sfz^+_mE_i=E_i\sfz^+_m$, $\sfz^+_mF_j=F_j\sfz^+_m$ for $i\in I$, $1\leq j\leq n-1$,
$m\in\mbz^+$

\item[(6)]
$\sfz^+_m\sfz^+_s=\sfz^+_s\sfz^+_m$ for $m,s\in\mbz^+$.
\end{itemize}
\end{Prop}
\begin{proof}
Let $\sU$ be the
$\mbq(v)$-algebra generated by
$E_i,\ F_j,\  K_i,\ K_i^{-1},\ \sfz^+_m$ ($1\leq i\leq n$, $1\leq j\leq n-1$, $m\in\mbz^+$) with the defining relations (1)-(6).
There is a surjective algebra homomorphism
$\Phi:\sU\ra\bfUn$ satisfying $\Phi(E_i)=E_i$, $\Phi(F_j)=F_j$,
$\Phi(K_i)=K_i$, $\Phi(\sfz^+_m)=\sfz^+_m$ for $i\in I$, $1\leq j\leq n-1$ and $m\in\mbz^+$.

Let $\sU_1$ (respectively, $\sU_2$) be the subalgebra of $\sU$ generated by the elements $E_i$, $i\in I$ (respectively $F_j$, $1\leq j\leq n-1$).
Let $\sU_3$ be  the subalgebra of $\sU$ generated by the elements
$K_i^{\pm 1}$, $\sfz^+_m$ for $i\in I$ and $m\in\mbz^+$. Let $\sB_i$  be a $\mbq(v)$-basis of $\sU_i$ for $1\leq i\leq 3$.
Let $\sB=\sB_1\sB_2\sB_3$. Then we have
$$\sU=\spann_{\mbq(v)}\sB.$$
Let $\bfU^+(\afsl)$ be the subalgebra of $\bfUn$ generated by the elements
$u_i^+$ for $i\in I$. By \cite[33.1.3]{Lubk} we see that
there is a natural algebra homomorphism $\iota:\bfU^+(\afsl)\ra\sU$ such that $\iota(u_i^+)=E_i$ for $i\in I$. Since $\Phi\circ\iota$ is injective and $\sU_1=\iota(\bfU^+(\afsl))$, we conclude that the map $\Phi|_{\sU_1}:\sU_1\ra\bfUn$ is injective.  Similarly, $\Phi|_{\sU_2}:\sU_2\ra\bfUn$ is injective.
By Lemma \ref{basis of bfUn} and \cite{Hub}
 we have $$\bfUn=\bfU^+(\afsl)\ot\mbq(v)[\sfz^+_m\mid m\in\mbz^+]\ot\bfUnz\ot\bfUnm.$$
Hence $\Phi(\sB)$ is the basis of $\bfUn$. It follows that $\Phi$ is an algebra isomorphism. The proposition is proved.
\end{proof}
\begin{Rem}
If $n=1$, then the  algebra $\bfUn$ is the $\mbq(v)$-algebra generated by
$K_1,\ K_1^{-1},\ \sfz^+_m$, $i\in I$,
$m\in\mbz^+$ with relations:
\begin{itemize}
\item[(1)]
$K_{1}K_{1}^{-1}=K_{1}^{-1}K_{1}=1$,

\item[(2)]
$\sfz^+_mK_1=K_1\sfz^+_m$, for $m\in\mbz^+$.

\item[(3)]
$\sfz^+_m\sfz^+_s=\sfz^+_s\sfz^+_m$ for $m,s\in\mbz^+$.
\end{itemize}
\end{Rem}
\begin{Prop}\label{Hopf subalgebra}
The algebra $\bfUn$ is a $\mbn$-graded Hopf algebra with $\deg(E_j)=\deg(F_j)=\deg(K_i)=0$ for $i\in I$, $1\leq j\leq n-1$, $\deg(E_n)=1$, $\deg(\sfz^+_m)=m$ for $m\geq 1$,
comultiplication $\Dt$, counit $\ep$,
and antipode $\sg$ defined
by
\begin{eqnarray*}
&\Delta(E_i)=E_i\otimes\ti K_i+1\otimes
E_i,\quad\Delta(F_j)=F_j\otimes
1+\ti K_j^{-1}\otimes F_j,&\\
&\Delta(K^{\pm 1}_i)=K^{\pm 1}_i\otimes K^{\pm 1}_i,\quad
\Delta(\sfz_s^+)=\sfz_s^+\otimes1+1\otimes
\sfz_s^+;&\\
&\varepsilon(E_i)=\varepsilon(F_j)=0=\varepsilon(\sfz_s^+),
\quad \varepsilon(K_i)=1;&\\
&\sg(E_i)=-E_i\ti K_i^{-1},\quad \sg(F_j)=-\ti K_jF_j,\quad
\sg(K^{\pm 1}_i)=K^{\mp 1}_i,&\\
&\text{and }\;\;\sg(\sfz_s^+)=-\sfz_s^+,&
\end{eqnarray*}
 where $i\in I$, $1\leq j\leq n-1$ and $s\geq 1$.
\end{Prop}
\begin{proof}
By Proposition \ref{Hopf algebra}, we see that the algebra $\bfUn$ is a Hopf subalgebra of $\afUgl$. Furthermore by Proposition \ref{presentation dHallAlg} we conclude that the algebra $\bfUn$ is a $\mbn$-graded algebra with $\deg(E_j)=\deg(F_j)=\deg(K_i)=0$ for $i\in I$, $1\leq j\leq n-1$, $\deg(E_n)=1$, $\deg(\sfz^+_m)=m$ for $m\geq 1$. Clearly $\Delta$ and $\varepsilon$ are graded algebra homomorphisms, whereas
 $\sg$ is a graded anti-algebra homomorphism. The proposition is proved.
\end{proof}

\subsection{The modified quantum current algebra $\dbfUn$}
 Let  $\afPin=\{\afal_j:=\afbse_j-\afbse_{j+1}\mid 1\leq j\leq n\}.$
The algebra $\bfUn$ is a $\mbz\afPin$-graded algebra
$$\bfUn=\bop_{\nu\in\mbz\afPin}\bfUn(\nu),$$
with $u_A^+\in\bfUn({\sum\limits_{1\leq i\leq n}d_i\afal_i})$, $u_B^-\in\bfUn(-
\sum_{1\leq i\leq n}d_i'\afal_i)$ and $K_i^{\pm 1}\in\bfUn(0)$.
for $A\in\afThnp$, $B\in\Thnp$ and $1\leq i\leq n$, where $(d_i)_{i\in\mbz}=\bfd(A)$ and
$(d_i')_{i\in\mbz}=\bfd(B)$.
By \cite[3.5.2]{Fu13}, we have
$$u_A^+\in\bfUn(\ro(A)-\co(A)),\
 u_B^-\in\bfUn(\co(B)-\ro(B))$$
for $A\in\afThnp$, $B\in\Thnp$.

Following \cite{Lubk}, we introduce the modified quantum current algebra $\dbfUn$ as follows.
For $\la,\mu\in\afmbzn$ let
$$\la\centerdot\mu=\la_1\mu_1+\la_2\mu_2+\cdots+\la_n\mu_n.$$
For $\la,\mu\in\afmbzn$ we set ${}_\la\bfUn_\mu=\bfUn/{}_\la I_\mu$, where
\begin{equation*}
{}_\la I_\mu=\sum_{\bfj\in\afmbzn}(\Kbfj-
 \up^{\la\centerdot\bfj})\bfUn+\sum_{\bfj\in\afmbzn}\bfUn(\Kbfj
 -\up^{\mu\centerdot\bfj}).
 \end{equation*}
Let $\pi_{\la,\mu}:\bfUn\ra{}_\la\bfUn_\mu$ be the canonical projection.
Let
$$\dbfUn:=\bop_{\la,\mu\in\afmbzn}{}_\la\bfUn_\mu.$$

For $\la',\mu',\la'',\mu''\in\afmbzn$ with
$\la'-\mu',\la''-\mu''\in\mbz\afPin$ and any $x\in\bfUn({\la'-\mu'})$,
$y\in\bfUn({\la''-\mu''})$,  define
\begin{equation*}
\pi_{\la',\mu'}(x)\pi_{\la'',\mu''}(y)=
\begin{cases}
\pi_{\la',\mu''}(xy)& \text{if  $\mu'=\la''$},\\
0 &\text{otherwise.}
\end{cases}
\end{equation*}
Then $\dbfUn$ becomes an associative $\mbq(\up)$-algebra. The algebra $\dbfUn$ is naturally a $\bfUn$-bimodule defined by
$x'\pi_{\la',\la''}(y)x''=\pi_{\la'+\nu',\la''-\nu''}(x'yx'')$,
for $x'\in\bfUn({\nu'})$, $y\in\bfUn$, $x''\in\bfUn({\nu''})$ and $\la',\la''\in\afmbzn$.

By Lemma \ref{basis of bfUn}, we have the following result.
\begin{Lem}
The set $\{u_A^+1_\la u_B^-\mid A\in\afThnp,\,\la\in\afmbzn,\,B\in\Thnp\}$ forms a $\mbq(v)$ basis of $\dbfUn$, where $1_\la=\pi_{\la,\la}(1)$.
\end{Lem}

Let $\dUn$ be the $\sZ$-submodule of $\dbfUn$ spanned by the elements $u_A^+1_\la u_B^-$ for $A\in\afThnp$, $B\in\Thnp$
and $\la\in \afmbzn$.

\section{Schur--Weyl reciprocity for the quantum current algebra $\bfUn$}

\subsection{The algebra $\bfUnr$}
Let $\Unr$ be the $\sZ$-submodule of the affine quantum Schur algebra $\afbfSr$ spanned by the set
$\{[A]\mid A\in\Xinr\}$, where $$\Xinr=\{A\in\afThnr\mid a_{i,j}=0 \text{ for }1\leq i\leq n,\, j<1\}.$$
Let $$\bfUnr=\spann_{\mbq(v)}\{[A]\mid A\in\Xinr\}\han\afbfSr.$$
We will prove in Proposition \ref{subalgebra} that $\Unr$  is a $\sZ$-subalgebra of $\afbfSr$. Thus,
$\bfUnr$ is a $\mbq(v)$-subalgebra of $\afbfSr$.
Furthermore, we shall prove in Theorem \ref{quotient} that the algebra $\bfUnr$ is a quotient algebra of $\bfUn$. These algebras $\bfUnr$ will be used to give a BLM realization of $\bfUn$ in \S5.
The algebra $\bfUnr$ plays the role of the affine quantum Schur algebra $\afbfSr$.
Categorifications of affine quantum Schur algebras $\afbfSr$ for $n\geq r$ was given in \cite{MT15,MT17}.
It would be interesting to investigate the categorification of the algebra $\bfUnr$.

Let
\begin{equation}\label{Xin}
\Xin=\{A\in\afThn\mid a_{i,j}=0 \text{ for }1\leq i\leq n,\, j<1\}\text{ and }
\Xinpm=\Xin\cap\afThnpm.
\end{equation}
The following result can be easily proved.
\begin{Lem}
For $1\leq i_0\leq n$
The set $\{A(\bfj,r)\mid A\in\Xinpm,\,\bfj\in\afmbnn,\,j_{i_0}=0,\,\sg(A)+\sg(\bfj)\leq r\}$
forms a $\mbq(v)$-basis for $\bfUnr$.
\end{Lem}

For $T=(t_{i,j})\in\aftiThn$ let $\dt_T=(t_{i,i})_{i\in\mbz}\in\afmbzn,$ and $\ti T=(\ti t_{i,j})$,
where $\ti t_{i,j}=t_{i-1,j}$ for all $i,j\in\mbz$.
Let $\bar\ :\sZ\ra\sZ$ be the ring homomorphism defined by $\bar \up=\up^{-1}$.
To prove that $\Unr$ is a $\sZ$-subalgebra of the affine quantum Schur algebra $\afbfSr$, we need the following multiplication formulas.
\begin{Prop}\label{[B][A]}
Let $A\in\Xinr$ and  $\al,\ga\in\afmbnn$ with $\ga_n=0$.

$(1)$ If $B\in\afThnr$ satisfies that $B-\sum\limits_{1\leq i\leq n}\al_i\afE_{i,i+1}$ is a diagonal matrix and $\co(B)=\ro(A)$, then in $\Unr$:
$$[B][A]=\sum_{T\in\afThn,\,\ro(T)=\al \atop
A+T-\ti T\in\Xinr}f_{T,A}[A+T-\ti T],$$
where $$f_{T,A}=v^{\bt(T,A)}\prod_{1\leq i\leq n\atop j\in\mbz}\ol{\dleb{a_{i,j}+t_{i,j}-t_{i-1,j}\atop t_{i,j}}\drib}$$
and
$\bt(T,A)=\sum_{1\leq i\leq n,\,j\geq l}(a_{i,j}-t_{i-1,j})t_{i,l}-\sum_{1\leq i\leq n,\,j>l}(a_{i+1,j}-t_{i,j})t_{i,l}$.

$(2)$ If $C\in\afThnr$ satisfies that $C-\sum_{1\leq i\leq n}\ga_i\afE_{i+1,i}$ is a diagonal matrix and $\co(C)=\ro(A)$, then in $\Unr$:
$$[C][A]=\sum_{T\in\afThn,\,\ro(T)=\ga\atop
A-T+\ti T\in\Xinr}f'_{T,A}[A-T+\ti T],$$
where $$f'_{T,A}=v^{\bt'(T,A)}\prod_{1\leq i\leq n\atop j\in\mbz}\ol{\dleb{a_{i,j}-t_{i,j}+t_{i-1,j}\atop t_{i-1,j}}\drib}$$ and
$\bt'(T,A)=\sum_{1\leq i\leq n,\,l\geq j}(a_{i,j}-t_{i,j})t_{i-1,l}-\sum_{1\leq i\leq n,\,l>j}(a_{i,j}-t_{i,j})t_{i,l}$.
\end{Prop}
\begin{proof}
Since $A\in\Xinr$, we have $a_{i,j}=0$ for $1\leq i\leq n$ and $j<1$.
By \cite[3.6]{DF13} we have
$$[B][A]=\sum_{T\in\afThn,\,\ro(T)=\al \atop
A+T-\ti T\in\afThnr}f_{T,A}[A+T-\ti T].$$
If $f_{T,A}\not=0$ then we have $$a_{i,j}+t_{i,j}-t_{i-1,j}\geq t_{i,j}$$ for $i,j\in\mbz$.
It follows that $0=a_{i,j}\geq t_{i-1,j}$ for $1\leq i\leq n$ and $j<1$.
This implies that $t_{n,j}=t_{0,j-n}=0$ for $j<1$. Hence we have
$A+T-\ti T\in\Xinr$. The assertion (1) follows.

By \cite[3.6]{DF13}  we have
$$[C][A]=\sum_{T\in\afThn,\,\ro(T)=\ga\atop
A-T+\ti T\in\afThnr}f'_{T,A}[A-T+\ti T].$$
If $f_{T,A}'\not=0$ then we have $$a_{i,j}-t_{i,j}+t_{i-1,j}\geq t_{i-1,j}$$ for $i,j\in\mbz$.
This implies  that $0=a_{i,j}\geq t_{i,j}$ for $1\leq i\leq n$ and $j<1$.
Furthermore since $\ga_0=\ga_n=0$ and  $\ro(T)=\ga$, we have $t_{0,s}=0$ for $s\in\mbz$.
Hence we have $A-T+\ti T\in\Xinr$. The assertion (2) follows.
\end{proof}

\begin{Coro}\label{cor of formulas}
Let $\al,\ga\in\afmbnn$ with $\ga_n=0$.
Assume that $B,C\in\afThnr$ is such that $B-\sum_{1\leq i\leq n}\al_i\afE_{i,i+1}$ is a diagonal matrix
and $C-\sum_{1\leq i\leq n}\ga_i\afE_{i+1,i}$ is a diagonal matrix. Then
we have
$[B]\Unr\han\Unr$ and
$[C]\Unr\han\Unr.$
\end{Coro}

For $A\in\aftiThn$ with $\sg(A)=r$, we denote $[A]=0\in\afSr$ if $a_{i,i}<0$ for some $i\in\mbz$.
For $A\in\aftiThn$ let $$\bfsg(A)=(\sg_i(A))_{i\in\mbz}\in\afmbnn$$ where $\sg_i(A)=a_{i,i}+\sum_{j<i}(a_{i,j}+a_{j,i})$.
For $A,B\in\aftiThn$ define
\begin{equation*}\label{order sqsubset}
B\sqsubseteq A\text{ if and only if $B\pr A$, $\co(B)=\co(A)$ and $\ro(B)=\ro(A)$.}
\end{equation*}
 Put $B\sqsubset A$ if $B\sqsubseteq A$ and $B\not=A$.

For $A\in\afThnpm$, we write
\begin{equation*}
A=A^++A^-,
\end{equation*}
 where $A^+\in\afThnp$ and
$A^-\in\afThnm$. The following triangular relation in affine quantum Schur algebras is given in
{\cite[3.7.7]{DDF}}.
\begin{Prop} \label{tri-affine Schur algebras}
For $A\in\afThnpm$ and $\la\in\afLanr$, we have
\begin{equation*}
A^+(\bfl,r)[\diag(\la)]A^-(\bfl,r)
=[A+\diag(\la-\bfsg(A))]+f,
\end{equation*}
where $f$ is a $\sZ$-linear combination of $[B]$ with $B\sqsubset A+\diag(\la-\bfsg(A))$.
\end{Prop}

\begin{Coro}\label{tri2-affine Schur algebras}
Let $A\in\Xinpm$. There exists $w_{A^+}\in\ti\Sg$ and $w_{{}^t(A^-)}\in\Ga$ such that
$$\zeta_r(\ti u^+_{(w_{A^+})})[\diag(\la)]\zeta_r(\ti u^-_{(w_{{}^t(A^-)})})=[A+\diag(\la-\bfsg(A))]+f,$$
for $\la\in\afLanr$, where $f$ is  a $\sZ$-linear combination of $[B]$ with $B\in\Xinr$ and $B\sqsubset A+\diag(\la-\bfsg(A))$.
\end{Coro}
\begin{proof}
By Proposition \ref{tri Hall} and \ref{tri quantum gln} we see that
there exists $w_{A^+}\in\ti\Sg$ and $w_{{}^t(A^-)}\in\Ga$ such that
\begin{equation*}
\begin{split}
\zeta_r(\ti u^+_{(w_{A^+})})&=A^+(\bfl,r)+g_1\\
\zeta_r(\ti u^-_{(w_{{}^t(A^-)})})&=A^-(\bfl,r)+g_2,
\end{split}
\end{equation*}
where
$g_1$ is a $\sZ$-linear combination of $B(\bfl,r)$ for $B\in\afThnp$ with $B\p A^+$, and
$g_2$ is a $\sZ$-linear combination of $C(\bfl,r)$ for $C\in\Thnm$ with $C\p A^-$. It follows that
$$\zeta_r(\ti u^+_{(w_{A^+})})[\diag(\la)]\zeta_r(\ti u^-_{(w_{{}^t(A^-)})})=A^+(\bfl,r)[\diag(\la)]A^-(\bfl,r)+g,$$
where $g$ is a $\sZ$-linear combination of $B^+(\bfl,r)[\diag(\la)]B^-(\bfl,r)$ for $B\in\Xinpm$ with $B\p A$.
Hence by Proposition  \ref{tri-affine Schur algebras} we have
$$\zeta_r(\ti u^+_{(w_{A^+})})[\diag(\la)]\zeta_r(\ti u^-_{(w_{{}^t(A^-)})})=[A+\diag(\la-\bfsg(A))]+f,$$
where $f$ is  a $\sZ$-linear combination of $[B]$ with $B\in\afThnr$ and $B\sqsubset A$. From Corollary \ref{cor of formulas} we see that $\zeta_r(\ti u^+_{(w_{A^+})})[\diag(\la)]\zeta_r(\ti u^-_{(w_{{}^t(A^-)})})\in\Unr$. Thus, $f$ must be a $\sZ$-linear combination of $[B]$ with $B\in\Xinr$ and $B\sqsubset A$.
\end{proof}

For $A\in\aftiThn$, let
$$\ddet{A}=\sum_{i<j\atop1\leq i\leq n}{j-i+1\choose 2}(a_{i,j}+a_{j,i}).$$
We are now prepared to prove that $\Unr$ is a $\sZ$-subalgebra of the affine quantum Schur algebra $\afbfSr$.

\begin{Prop}\label{subalgebra}
$\Unr$ is a $\sZ$-subalgebra of $\afbfSr$ generated by the elements
$[\sum_{1\leq i\leq n}\al_i\afE_{i,i+1}+\diag(\la)]$ and
$[\sum_{1\leq i\leq n}\ga_i\afE_{i+1,i}+\diag(\la)]$ for
$\al,\la,\ga\in\afmbnn$ with $\sg(\al)+\sg(\la)=r$,
$\sg(\ga)+\sg(\la)=r$ and $\ga_n=0$. In particular, $\bfUnr$ is a $\mbq(v)$-subalgebra of $\afbfSr$.
\end{Prop}
\begin{proof}
Let $\Unr'$ be the $\sZ$-subalgebra of $\afbfSr$ generated by the indicated elements. From Corollary \ref{cor of formulas} we see that
$$\Unr'\han \Unr'\Unr\han\Unr.$$
We shall show by induction on
$\ddet A$ that $[A]\in\Unr'$ for $A\in\Xinr$.
If $\ddet A=0$ then $A$ is a diagonal matrix. Hence we have
$[A]\in\Unr'$. Now we assume that  $\ddet A>0$ and $[B]\in\Unr'$
for $B\in\Xinr$ with $\ddet B<\ddet A$.
By Corollary \ref{tri2-affine Schur algebras} and \cite[3.7.6]{DDF}, there exists $w_{A^+}\in\ti\Sg$ and $w_{{}^t(A^-)}\in\Ga$ such that
$$\zeta_r(\ti u^+_{(w_{A^+})})[\diag(\bfsg(A))]\zeta_r(\ti u^-_{(w_{{}^t(A^-)})})=[A]+f,$$
where $f$ is a $\sZ$-linear combination of $[B]$ with $B\in\Xinr$ and $\ddet B<\ddet A$. By the induction hypothesis, we have $f\in\Unr'$ and it follows that $[A]\in\Unr'$. The proposition is proved.
\end{proof}

\subsection{Schur--Weyl reciprocity for $\bfUn$}
The classical Schur--Weyl reciprocity is a fundamental result in representation theory, establishing a deep connection between the irreducible representations of the symmetric group and the irreducible polynomial representations of the general linear group of a complex vector space. This reciprocity  is known to hold over any infinite field (see \cite{CL,CP,Gr80,Doty}). The quantum Schur--Weyl reciprocity between quantum $\frak{gl}_n$ and Hecke algebras of type $A$ was established  in the generic case by Jimbo \cite{Ji86}.
Over the years, numerous variations and generalizations of Schur--Weyl reciprocity have been developed; see, for example,  \cite{Ar,ATY,DDH,DH,SS,Hu}.

We now discuss Schur--Weyl reciprocity for $\bfUn$.
We prove that the algebra $\bfUnr$ is a quotient algebra of $\bfUn$.

\begin{Prop}\label{quotient}
For $r\in\mbn$, we have $\zeta_r(\bfUn)=\bfUnr$ and $\zeta_r(\Un)=\Unr$.
\end{Prop}
\begin{proof}
By Lemma \ref{basis of bfUn} and \cite[3.7.4(2)]{DDF} we have
$$\zeta_r(\Un)=\spann_\sZ\{A^+(\bfl,r)[\diag(\la)]A^-(\bfl,r)\mid A\in\Xinpm,\,\la\in\afLanr\}.$$
Furthermore, by Proposition \ref{tri-affine Schur algebras} and \ref{subalgebra}, for $A\in\Xinpm$ and $\la\in\afLanr$, we have
\begin{equation}\label{tri Unr}
A^+(\bfl,r)[\diag(\la)]A^-(\bfl,r)
=[A+\diag(\la-\bfsg(A))]+f,
\end{equation}
where $f$ is a $\sZ$-linear combination of $[B]$ for $B\in\Xinr$ with $B\sqsubset A+\diag(\la-\bfsg(A))$. Thus we have $\zeta_r(\Un)=\Unr$ and $\zeta_r(\bfUn)=\bfUnr$.
\end{proof}

\subsection{The algebra $\msbfHr$}
For any algebra
$\mathscr A$ over $\mbq(v)$, the notation ${\scr A}\hmod$
represents the category of all finite dimensional left $\scr
A$-modules.
We will establish in Proposition \ref{category equivalence}
an equivalence of categories between the
categories $\bfUnr\hmod$ and $\msbfHr\hmod$ when $n\geq r$, where $\msbfHr$ is a certain subalgebra of the
extended affine Hecke algebra of type $A$ defined in \eqref{msbfHr}.

Let $\afsygr$ be the group consisting of all permutations
$w:\mbz\ra\mbz$ such that $w(i+r)=w(i)+r$ for $i\in\mbz$.
The extended affine Hecke algebra $\afHr$ over $\sZ$ associated to
$\affSr$ is the (unital) $\sZ$-algebra with basis
$\{T_w\}_{w\in\affSr}$, and multiplication defined by
\begin{equation*}
\begin{cases} T_{s_i}^2=(\up^2-1)T_{s_i}+\up^2,\quad&\text{for }1\leq i\leq r\\
T_{w}T_{w'}=T_{ww'},\quad&\text{if}\
\ell(ww')=\ell(w)+\ell(w'),
\end{cases}
\end{equation*}
where $s_i\in\affSr$ is defined by
setting
$s_i(j)=j$ for $j\not\equiv i,i+1\mnmod r$, $s_i(j)=j-1$ for
$j\equiv i+1\mnmod r$ and $s_i(j)=j+1$ for $j\equiv i\mnmod r$. Let $\afbfHr=\afHr\ot_\sZ\mbq(\up)$.

Let
\begin{equation}\label{msbfHr}
\begin{split}
\msHr&=\spann_{\sZ}\{T_w\mid w\in\affSr,\,w^{-1}(i)>0
\text{ for }1\leq i\leq r\}\han\afHr\\
\msbfHr&=\spann_{\mbq(v)}\{T_w\mid w\in\affSr,\,w^{-1}(i)>0
\text{ for }1\leq i\leq r\}\han\afbfHr.
\end{split}
\end{equation}
We will prove in Lemma \ref{subalgebra Hr} that $\msHr$ is a $\sZ$-subalgebra of $\afHr$.

For $\la\in\afLanr$, let
$\afmsD_\la=\{d\mid d\in\affSr,\ell(wd)=\ell(w)+\ell(d)\text{ for
$w\in\fS_\la$}\}$ and $\afmsD_{\la,\mu}=\afmsD_{\la}\cap{\afmsD_{\mu}}^{-1}$.
There is
a bijective map
\begin{equation*}
{\jmath_\vtg}:\{(\la, d,\mu)\mid
d\in\afmsD_{\la,\mu},\la,\mu\in\afLanr\}\lra\afThnr
 \end{equation*}
sending $(\la, d,\mu)$ to the matrix $A=(|R_k^\la\cap dR_l^\mu|)_{k,l\in\mbz}$, where
\begin{equation*}
R_{i+kn}^{\nu}=\{\nu_{k,i-1}+1,\nu_{k,i-1}+2,\ldots,\nu_{k,i-1}+\nu_i=\nu_{k,i}\}\; \text{ with }\;\nu_{k,i-1}=kr+\sum_{1\leq t\leq i-1}\nu_t,
\end{equation*}
for all $1\leq i\leq n$, $k\in\mbz$ and $\nu\in\afLa(n,r)$.

Assume $n\geq r$. Let $$e_\og=[\diag(\og)]\in\bfUnr,$$ where $\og=(\ldots,1^r,0^{n-r},1^r,0^{n-r},\ldots)\in\afLanr$.
Clearly there is an algebra isomorphism
$$\theta_{n,r}:\afbfHr\ra e_\og\afbfSr e_\og$$
such that $\theta_{n,r}(T_d)=e_A$ for $d\in\fSr$, where $A=\jmath_\vtg(\og,d,\og)\in\afThnr$.

\begin{Lem}\label{subalgebra Hr}
Assume $n\geq r$. Then we have $\theta_{n,r}(\msHr)=e_\og\Unr e_\og$. In particular, $\msHr$ is a $\sZ$-subalgebra of $\afHr$.
\end{Lem}
\begin{proof}
By definition we have $$e_\og\Unr e_\og=\spann_{\sZ}\{[A]\mid \ro(A)=\co(A)=\og,\,A\in\Xinr\}.$$
For $A=\jmath_\vtg(\og,d,\og)$ with $d\in\affSr$, we have
\begin{equation*}
\begin{split}
A\in\Xinr&\Leftrightarrow R_i^\og\cap dR_j^\og
=\emptyset\text{ for $1\leq i\leq r$, $j<1$}\\
&\Leftrightarrow\{1,2,\cdots,r\}\cap\{d(j)\mid j\in\mbz,\,j\leq 0\}
=\bin_{1\leq i\leq r\atop j\in\mbz,\,j<1}R_i^\og\cap dR_j^\og=\emptyset\\
&\Leftrightarrow d^{-1}(i)>0 \text{ for } 1\leq i\leq r.
\end{split}
\end{equation*}
It follows that $\theta_{n,r}(\msHr)=e_\og\Unr e_\og$.
Therefore by Proposition \ref{subalgebra} we conclude that
$\msHr$ is a $\sZ$-subalgebra of $\afHr$.
\end{proof}

With the above result, we will identify $\msHr$ with $e_\og\Unr e_\og$ for $n\geq r$.

\subsection{An equivalence of categories}
We are now ready to show that the categories $\bfUnr\hmod$ and
$\msbfHr\hmod$ are equivalent.
Recall the notation $\Thn$ introduced in \eqref{Thn}. Let
\begin{equation}\label{bfSr}
\bfSr =\spann_{\mbq(v)}\{[A]\mid A\in\Thnr\}\han\bfUnr,
\end{equation}
where $\Thnr=\{A\in\Thn\mid\sg(A)=r\}$. Then
$\bsS(n,r)$ is the $q$-Schur algebra over $\mbq(v)$.
Let $\bsH(r)$ be the subalgebra of $\afbfHr$ generated by the elements
$T_{s_i}$ for $1\leq i\leq r-1$.
\begin{Lem}\label{module isomorphism}
Assume $n\geq r$. Then there is a right $\msbfHr$-module isomorphism
$$\bsS(n,r)e_\og\ot_{\bsH(r)}\msbfHr\stackrel{\sim}{\lra}\bfUnr e_\og,\;
x\ot h\longmapsto xh.$$
\end{Lem}
\begin{proof}
Clearly there is a right $\msbfHr$-module isomorphism
$$\vi:\bsS(n,r)e_\og\ot_{\bsH(r)}\msbfHr\stackrel{\sim}{\lra}\bfUnr e_\og,\;
x\ot h\longmapsto xh.$$
For $\la\in\afLanr$, $d\in\affSr$ and $A=\jmath_\vtg(\la,d,\og)$ we have
\begin{equation*}
\begin{split}
A\in\Xinr&\Leftrightarrow \{1,2,\cdots,r\}\cap\{d(j)\mid j\in\mbz,\,j\leq 0\}=\bin_{1\leq i\leq n\atop j\in\mbz,\,j<1}R_i^\la\cap dR_j^\og=\emptyset\\
 &\Leftrightarrow d^{-1}(i)>0 \text{ for } 1\leq i\leq r.
\end{split}
\end{equation*}
It follows that
\begin{equation}\label{eq module isomorphism}
\begin{split}
\bfUnr e_\og&=\spann_{\mbq(v)}\{[A]\mid A\in\Xinr,\,\co(A)=\og\}\\
&=\spann_{\mbq(v)}\{[\jmath_\vtg(\la,d,\og)]\mid \la\in\afLanr,\,d^{-1}(i)>0,\,\text{ for }1\leq i\leq r\}.
\end{split}
\end{equation}
Furthermore we have
$\bsS(n,r)e_\og\ot_{\bsH(r)}\msbfHr=\spann_{\mbq(v)}\sX$,
where
$$\sX=\{[\jmath_\vtg(\la,1,\og)]\ot T_d\mid\la\in\afLanr,\,
d\in\afmsD_\la,\,d^{-1}(i)>0,\,1\leq i\leq r\}.$$
By \eqref{eq module isomorphism}, we see that
the set $\vi(\sX)$ forms a basis for $\bfUnr e_\og$.
Therefore $\vi$ is a right $\msbfHr$-module isomorphism.
\end{proof}

\begin{Prop}\label{category equivalence}
Assume $n\geq r$. Then the categories $\bfUnr\hmod$ and
$\msbfHr\hmod$ are equivalent.
\end{Prop}
\begin{proof}
We define the following two functors
\begin{equation*}\label{functor sfF}
\begin{split}
&\sfF:\msbfHr\hmod\lra\bfUnr\hmod,\;L\longmapsto\bfUnr e_\og\ot_{\msbfHr}L\\
&\sfG:\bfUnr\hmod\lra\msbfHr\hmod,\; M\longmapsto
e_\og M.
\end{split}
\end{equation*}
Here we have identified $ e_\og \bfUnr e_\og $ with
$\msbfHr$.

Clearly, for any left $\bsS(n,r)$-module $M$, there
is a left $\bsS(n,r)$-module isomorphism
\begin{equation}\label{the iso f}
f:\bsS(n,r) e_\og  \ot_{\bsH(r)} e_\og M\cong M
\end{equation}
defined by $f(x\ot m)=xm$ for $x\in\bsS(n,r) e_\og $ and
$m\in  e_\og M$. By Lemma \ref{module isomorphism},
\eqref{the iso f} induces a left $\bsS(n,r)$-module
isomorphism
$$g:\bfUnr e_\og
\ot_{\msbfHr} e_\og M\cong M$$
 satisfying $g(x\ot m)=xm$, for $x\in\bsS(n,r) e_\og $
and $m\in  e_\og M$. By \eqref{eq module isomorphism} we have
$$\bfUnr e_\og =\bop_{\la\in\afLanr}
[\jmath_\vtg(\la,1,\og)] \msbfHr.$$ Furthermore, for any $\la\in\afLanr$, $h\in\msbfHr$,
and $m\in e_\og M$,
we have
$$g([\jmath_\vtg(\la,1,\og)]h\ot m)=g([\jmath_\vtg(\la,1,\og)]\ot
hm)=([\jmath_\vtg(\la,1,\og)] h) m.$$
Hence, $g(x\ot
m)=xm$, for all $x\in\bfUnr e_\og $ and $m\in  e_\og M$. Thus,
$g$ is an $\bfUnr$-module isomorphism. Therefore, we have $\sfF\circ
\sfG\cong \mathsf{id}_{\bfUnr\hmod}$. In addition by \cite[(6.2d)]{Gr80} we have
$\sfG\circ \sfF\cong \mathsf{id}_{\msbfHr\hmod}$. The proposition is proved.
\end{proof}

\section{BLM realization of the quantum current algebra $\bfUn$}
In this section, we will use the algebras $\bfUnr$ introduced in \S4.1 to give a BLM realization of $\bfUn$, and prove that
$\Un$ is a $\sZ$-subalgebra of $\bfUn$.

\subsection{The algebras $\afbfLn$ and $\afhbfLn$}
Let $v'$ be an indeterminate independent of $v$.
Let $\sZ_1$ be the subring of $\mbq(\up)[\up']$ generated  by
$\prod_{1\leq i\leq t}\frac{\up^{-2(a-i)}\up'^2-1} {\up^{-2i}-1}$
and $\up^j$
for all $a\in\mbz$, $t\geq 1$ and $j\in\mbz$.
Let $\sZ_2$ be the subring of $\mbq(\up)[\up',v'^{-1}]$ generated  by
$\prod_{1\leq i\leq t}\frac{\up^{-2(a-i)}\up'^2-1} {\up^{-2i}-1}$, $\prod_{1\leq i\leq t}\frac{\up^{2(a-i)}\up'^{-2}-1} {\up^{2i}-1}$,
and $\up^j$
for all $a\in\mbz$, $t\geq 1$ and $j\in\mbz$.

For $A\in\aftiThn$ and $p\in\mbz$, let
$${}_pA=A+pI$$
where $I\in\afThn$ is the identity matrix.
If $p$ is large enough, then we have ${}_pA\in\afThn$.
Let
$$\tiXin=\{A\in\aftiThn\mid a_{i,j}=0 \text{ for }1\leq i\leq n,\, j<1\}.$$

\begin{Prop}\label{stabilization property}
Let $A,B\in\tiXin$ and assume $\co(B)=\ro(A)$. There exist $X_1,\cdots,X_m\in\tiXin$, elements $P_1(\up,\up'),\cdots,P_m(\up,\up')\in\sZ_1$ and an integer $p_0\geq 0$ such that, in $\UpA$,
\begin{equation}\label{eq stabilization}
[{}_pB][{}_pA]=\sum_{1\leq i\leq m}P_i(\up,\up^{-p})[{}_pX_i]
\end{equation}
for all $p\geq p_0$.
\end{Prop}
\begin{proof}
By \cite[6.3]{DF18}, there exist $X_1,\cdots,X_m\in\aftiThn$, elements $P_1(\up,\up'),\cdots,P_m(\up,\up')\in\sZ_1$ and an integer $p_0\geq 0$ such that in $\afSpA$,
\begin{equation*}
[{}_pB][{}_pA]=\sum_{1\leq i\leq m}P_i(\up,\up^{-p})[{}_pX_i]
\end{equation*}
for all $p\geq p_0$. Since $A,B\in\tiXin$, by Proposition \ref{subalgebra}, we conclude that for $p\geq p_0$, $[{}_pB][{}_pA]$ is a $\sZ$-linear combination of $[C]$ for $C\in\XinpA$. Therefore, if $X_i\not\in\tiXin$ for some $i$, then $P_i(\up,\up^{-p})=0$ for $p\geq p_0$ and hence $P_i(v,v')=0$. The proposition is proved.
\end{proof}

Let $\aftiLn$ be the free $\sZ_1$-module with basis $\{A\mid A\in\tiXin\}$. By Proposition \ref{stabilization property}, there is a unique
associative $\sZ_1$-algebra structure on
 $\aftiLn$ such that
\begin{equation*}\label{tiKn}
B\cdot A=\begin{cases}\sum_{1\leq i\leq m}P_i(\up,\up')X_i, &\text{  if $\co(B)=\ro(A)$};\\
0,&\text{ otherwise.}\end{cases}
\end{equation*}
Let
\begin{equation}\label{afLn}
\afLn=\aftiLn\ot_{\sZ_1}\sZ,
\end{equation}
where $\sZ$ is regarded as a $\sZ_1$-module by specializing $\up'$ to $1$.
For $A\in\tiXin$ let $[A]=A\ot 1\in\afLn$.
Then $\afLn$ is an associative $\sZ$-algebra with basis $\{[A]\in\tiXin\}$.
Let $\afbfLn=\afLn\ot_\sZ\mbq(\up)$.

Following \cite[5.1]{BLM}, let $\afhbfLn$   be
the vector space of all formal (possibly infinite) $\mbq(\up)$-linear combinations
$\sum_{A\in\tiXin}\beta_A\dob A\dcb$ such that,
for ${\bf x}\in\afmbzn$, the sets
${\{A\in\tiXin\ |\ \beta_A\neq0,\ \ro(A)={\bf
x}\}}$ and ${\{A\in\tiXin\ |\ \beta_A\neq0,\ \co(A)={\bf
x}\}}$ are finite. For $\sum_{A\in\tiXin}\beta_A\dob A\dcb, \sum_{B\in\tiXin}\gamma_B\dob B\dcb\in\afhbfLn$, we define
$$\sum_{A\in\tiXin}\beta_A\dob A\dcb\sum_{B\in\tiXin}\gamma_B\dob B\dcb=\sum_{A,B\in\tiXin}\beta_A\gamma_B\dob A\dcb\dob B\dcb,$$
where $[A][B]$ is the product in $\afLn$. Then $\afhbfLn$ becomes  an associative algebra over $\mbq(v)$.

Recall the notation $\Xinpm$ introduced in \eqref{Xin}.
For $A\in\Xinpm$, $\bfj\in\afmbzn$ and $\la\in\afmbnn$, let
\begin{equation}
\begin{split}
A(\bfj)&=\sum_{\mu\in\afmbzn}\up^{\mu\centerdot\bfj}
[A+\diag(\mu)]\in\afhbfLn,\\
A(\bfj,\la)&=\sum_{\mu\in\afmbzn}\up^{\mu\centerdot\bfj}
\leb{\mu\atop\la}\rib[A+\diag(\mu)]\in\afhbfLn,
\end{split}
\end{equation}
where  $[{\mu\atop\la}]=
\prod_{1\leq i\leq n}[{\mu_i\atop\la_i}]$.
By Proposition \ref{stabilization property}, the algebra homomorphisms $\zr$ given in Theorem \ref{zr} induce
an algebra homomorphism
\begin{equation}\label{zeta}
\zeta:\bfUn\ra\afhbfLn
\end{equation}
such that $\zeta(\ti u_A^+)=A(\bfl)$, $\zeta(\ti u_B^-)=(\tB)(\bfl)$ and
$\zeta(K^{\bfj})=0(\bfj)$ for $A\in\afThnp$, $B\in\Thnp$ and $\bfj\in\afmbzn$.

\subsection{The algebra $\afbfWn$}
Let $\afbfWn$ be the $\mbq(\up)$-subspace of $\afhbfLn$ spanned by
$A(\bfj)$ for $A\in\Xinpm$ and $\bfj\in\afmbzn$. We shall prove in Lemma \ref{Wn-subalgebra} that $\afbfWn$ is a $\mbq(v)$-subalgebra of $\afhbfLn$. Furthermore, we will prove in Theorem \ref{realization of Un} that
$\afbfWn$ is a realization of $\bfUn$. We need several preliminary lemmas.
\begin{Lem}
For $A\in\Xinpm$, $\bfj\in\afmbzn$ and $\la\in\afmbnn$
we have $A(\bfj,\la)\in\afbfWn$.
\end{Lem}
\begin{proof}
Clearly
we have
\begin{equation}\label{0{la}A{0}}
\begin{split}
0(\bfj,\la) A(\bfl)&=
\up^{\ro(A)\centerdot(\bfj+\la)}A(\bfj,\la)+
\sum_{\bfj'\in\afmbnn\atop\bfj'<\la}\up^{\ro(A)\centerdot(\bfj+\bfj')}
\leb{\ro(A)\atop\la-\bfj'}\rib A(\bfj+\bfj'-\la,\bfj').
\end{split}
\end{equation}
It follows that $A(\bfj,\la)\in\spann_{\mbq(v)}\{0(\dt,\mu)A(\bfl)\mid\dt\in\afmbzn,
\,\mu\in\afmbnn\}$.
In addition, we have
\begin{equation*}
0(\dt,\mu)=\prod_{1\leq i\leq n}\bigg(0(\afbse_i)^{\dt_i}
\prod_{1\leq s\leq\mu_i}\frac{0(\afbse_i)v^{-s+1}-0(-\afbse_i)v^{s-1}}{v^s-v^{-s}}
\bigg).
\end{equation*}
Hence we have $A(\bfj,\la)\in\spann_{\mbq(v)}\{0(\dt)A(\bfl)\mid\dt\in\afmbzn\}
=\spann_{\mbq(v)}\{A(\dt)\mid\dt\in\afmbzn\}$. The lemma is proved.
\end{proof}

Let $\afWn$ be the $\sZ$-submodule of $\afbfWn$ spanned by
$A(\bfj,\la)$ for $A\in\Xinpm,\,\bfj\in\afmbzn,\,\la\in\afmbnn.$
Clearly we have the following result.
\begin{Lem}\label{basis for afVn}
The set $\{A(\bfj,\la)\mid A\in\Xinpm,\,
\bfj=(j_i)_{i\in\mbz},\la\in\afmbnn,\,j_i\in\{0,1\},\forall i\}$ forms a $\sZ$-basis for $\afWn.$
\end{Lem}

\begin{Lem}\label{formula in Wn}
For $\al,\la\in\afmbnn$, $\bfj\in\afmbzn$,
we have
$(\sum_{1\leq i\leq n}\al_i\afE_{i,i+1})(\bfl)\afWn\han\afWn$,
$(\sum_{1\leq i\leq n-1}\al_i\afE_{i+1,i})(\bfl)\afWn\han\afWn$ and $0(\bfj,\la)\afWn\han\afWn$.
\end{Lem}
\begin{proof}
The assertion follows from Proposition \ref{[B][A]}, \ref{stabilization property} and \cite[4.2]{DF18}.
\end{proof}

\begin{Lem}\label{tri Wn}
For $A\in\Xinpm$,
there exists $w_{A^+}\in\ti\Sg$ and $w_{{}^t(A^-)}\in\Ga$ such that
$$\zeta(\ti u^+_{(w_{A^+})})\zeta(\ti u^-_{(w_{{}^t(A^-)})})=
A(\bfl)+g$$
where $g$ is a $\sZ$-linear combination of $B(\bfj,\dt)$ such that $B\in\Xinpm$, $B\p A$, $\dt\in\afmbnn$ and $\bfj\in\afmbzn$.
\end{Lem}
\begin{proof}
By Corollary \ref{tri2-affine Schur algebras} and Proposition \ref{stabilization property},  there exists $w_{A^+}\in\ti\Sg$ and $w_{{}^t(A^-)}\in\Ga$ such that
\begin{equation}\label{tri in Ln}
\zeta(\ti u^+_{(w_{A^+})})[\diag(\la)]\zeta(\ti u^-_{(w_{{}^t(A^-)})})=[A+\diag(\la-\bfsg(A))]+f,
\end{equation}
for $\la\in\afmbzn$, where $f$ is  a $\sZ$-linear combination of $[B]$ with $B\in\tiXin$ and $B\sqsubset A+\diag(\la-\bfsg(A))$. It follows that
$$\zeta(\ti u^+_{(w_{A^+})})\zeta(\ti u^-_{(w_{{}^t(A^-)})})=\sum_{\la\in\afmbzn}
\zeta(\ti u^+_{(w_{A^+})})[\diag(\la)]\zeta(\ti u^-_{(w_{{}^t(A^-)})})=A(\bfl)+g$$
where $g$ is a $\sZ$-linear combination of $[B]$ with $B\in\tiXin$ and $B\sqsubset A+\diag(\la-\bfsg(A))$.
By Lemma \ref{formula in Wn} we conclude that  $g$ must be a $\sZ$-linear combination of $B(\bfj,\dt)$ such that $B\in\Xinpm$, $B\p A$, $\dt\in\afmbnn$ and $\bfj\in\afmbzn$.
\end{proof}
\begin{Lem}\label{Wn-subalgebra}
The $\sZ$-module $\afWn$ is a $\sZ$-subalgebra of $\afhbfLn$. In particular, $\afbfWn$ is a $\mbq(v)$-subalgebra of $\afhbfLn$.
\end{Lem}
\begin{proof}
Let $\afWn'$ be the $\sZ$-subalgebra of $\afhbfLn$ generated by the  elements $(\sum_{1\leq i\leq n}\al_i\afE_{i,i+1})(\bfl)$, $(\sum_{1\leq i\leq n-1}\al_i\afE_{i+1,i})(\bfl)$, $0(\bfj,\la)$ for  $\al,\la\in\afmbnn$, $\bfj\in\afmbzn$.
By Lemma \ref{formula in Wn}, we have $$\afWn'\han\afWn'\afWn\han\afWn.$$
Furthermore by \eqref{0{la}A{0}},
we have
$$\spann_{\sZ}\{0(\bfj,\la) A(\bfl)\mid \bfj\in\afmbzn, \la\in\afmbnn\}
=\spann_{\sZ}\{A(\bfj,\la)\mid \bfj\in\afmbzn, \la\in\afmbnn\}$$
for $A\in\Xinpm$.
This together with Lemma \ref{tri Wn} implies that
$$\afWn\han\afWn'.$$
Hence we have $\afWn=\afWn'$. The assertion follows.
\end{proof}

\subsection{A realization of the quantum current algebra $\bfUn$}
We are now prepared to prove that $\afbfWn$ is isomorphic to the quantum current algebra $\bfUn$.
\begin{Lem}\label{zeta(Un)}
We have $\zeta(\Un)=\afWn$.
\end{Lem}
\begin{proof}
By Proposition \ref{tri Hall}, \ref{tri quantum gln} and Lemma \ref{tri Wn},
for $A\in\Xinpm$, we have
\begin{equation}\label{tri2 Wn}
\zeta(\ti u_{A^+})\zeta(\ti u_{{}^t(A^-)})=A(\bfl)+g
\end{equation}
where $g$ is a $\sZ$-linear combination of $B(\bfj,\dt)$ such that $B\in\Xinpm$, $B\p A$, $\dt\in\afmbnn$ and $\bfj\in\afmbzn$.
Clearly we have $$\zeta\bigg(\prod_{1\leq i\leq n}
K_i^{j_i}\bigg[{K_i;0\atop\la_i}\bigg]\bigg)=0(\bfj,\la)$$
for $\bfj\in\afmbzn$ and $\la\in\afmbnn$.
Therefore we have
$$\zeta(\Un)=\spann_\sZ\{\zeta(\ti u_{A^+})0(\bfj,\la)
\zeta(\ti u_{\tAm})\mid A\in\Xinpm,\,\bfj\in\afmbzn,\,\la\in\afmbnn\}=\afWn.$$
The proposition is proved.
\end{proof}

\begin{Thm}\label{realization of Un}
$(1)$ The map $\zeta$ defined in \eqref{zeta} induces an algebra isomorphism
$\bfUn\overset\zeta\cong\afbfWn$.

$(2)$ $\Un$ is a $\sZ$-subalgebra of $\bfUn$ isomorphic to $\afWn$. Moreover, $\Un$ is a Hopf subalgebra of $\bfUn$.

$(3)$ $\dUn$ is a $\sZ$-subalgebra of $\dbfUn$.
\end{Thm}
\begin{proof}
Clearly the set $\{A(\bfj)\mid A\in\Xinpm,\,
\bfj\in\afmbzn\}$ forms a $\mbq(v)$-basis for $\afbfWn.$
It follows from \eqref{tri2 Wn} that the set $\{\zeta(\ti u_{A^+}K^{\bfj}\ti u_{{}^t(A^-)})A\in\Xinpm,\,
\bfj\in\afmbzn\}$ forms a $\mbq(v)$-basis for $\afbfWn.$
Hence $\zeta$ is injective and $\zeta(\bfUn)=\afbfWn$. Now the result follows from Lemma \ref{Wn-subalgebra}, \ref{zeta(Un)} and \cite{Gr95}.
\end{proof}
With the above result, we will identify $\bfUn$ with $\afbfWn$.

\section{Canonical bases for the modified quantum current algebra $\dbfUn$}

Recall the algebra $\bfUnr$ introduced in \S4.1.
In this section, we first construct a canonical basis of $\bfUnr$ in Proposition \ref{canonical basis for Unr}, and then prove  in Theorem \ref{canonical basis for afLn} that
the canonical basis of $\bfUnr$ can be lifted to a canonical basis of the modified quantum current algebra $\dbfUn$.

\subsection{Canonical bases for $\bfUnr$}
For $r\in\mbn$, let
$$\bfB(n,r):=\{\{A\}\mid A\in\afThnr\}$$
be the canonical basis of $\afSr$ defined by Lusztig in \cite[4.1(d)]{Lu99}.
Let $\bar\ :\afSr\ra\afSr$  be
the (involutive) group homomorphism defined by
$$\ol{v^j\{A\}}=v^{-j} \{A\}$$
for $A\in\afThnr$ and $j\in\mbz$. By \cite[4.13]{Lu99},
 the map $\bar\ :\afSr\ra\afSr$ is a ring involution.

\begin{Lem}[{\cite[7.2]{DF18}}]\label{bar monomial} For $\al,\beta\in\afmbnn$, let $A=S_\al+\diag(\bt)\in\afThnr$. Then, in $\afSr$, $\ol{[A]}=[A]$ and $\ol{[\tA]}=[\tA]$.
\end{Lem}

By Proposition \ref{subalgebra} and Lemma \ref{bar monomial} we
see that the restriction of $\bar\ :\afSr\ra\afSr$ gives an
involution $$\bar\ :\Unr\ra\Unr.$$

\begin{Prop}\label{canonical basis for Unr}
The set $\Bnr:=\{\{A\}\mid A\in\Xinr\}$ forms a $\sZ$-basis for $\Unr$. In particular,
we have
$\Bnr=\bfB(n,r)\cap\Unr$.
\end{Prop}
\begin{proof}
By Corollary \ref{tri2-affine Schur algebras} for $A\in\Xinr$, there exists $w_{A^+}\in\ti\Sg$ and $w_{{}^t(A^-)}\in\Ga$ such that
$$\mpm^{(A)}:=\zeta_r(\ti u^+_{(w_{A^+})})[\diag(\bfsg(A))]\zeta_r(\ti u^-_{(w_{{}^t(A^-)})})=[A]+f,$$
where $f$ is a $\sZ$-linear combination of $[B]$ for $B\in\Xinr$ with $B\sqsubset A$. It follows that for $A\in\Xinr$, we have
$$[A]=\mpm^{(A)}+g,$$
where $g$ is a $\sZ$-linear combination of $\mpm^{(B)}$ for $B\in\Xinr$ with $B\sqsubset A$.
By Lemma \ref{bar monomial}, we have
$\ol{\mpm^{(A)}}=\mpm^{(A)}$ for $A\in\Xinr$.
Thus, for $A\in\Xinr$,
$\ol{[A]}-[A]$ is a $\sZ$-linear combination of $[B]$ for $B\in\Xinr$ with $B\sqsubset A$.
By
\cite[7.10]{Lu90b}, we conclude
that there is a unique $\sZ$-basis $\{\th_{\Ar}\mid A\in\Xinr\}$ for $\Unr$ such that $\ol{\th_{\Ar}}=\th_{\Ar}$ and
\begin{equation}\label{th_Ar}
\th_{A,r}-[A]\in\sum_{B\in\Xinr\atop B\sqsubset A}v^{-1}\mbz[v^{-1}][B],
\end{equation}
for $A\in\Xinr$. In addition by \cite[7.6]{DF18} we have
$$\{A\}-[A]\in\sum_{B\in\afThnr\atop B\sqsubset A}v^{-1}\mbz[v^{-1}][B],$$
for $A\in\afThnr$.
Hence for $A\in\Xinr$ we have $\ol{\th_{\Ar}-\{A\}}=\th_{\Ar}-\{A\}$ and
$$\th_{\Ar}-\{A\}=\sum_{B\in\afThnr\atop B\sqsubset A}h_{B,A,r}[B],$$
where $h_{B,A,r}\in\up^{-1}\mbz[\up^{-1}]$. Therefore by a standard argument, we conclude that ${\th_{\Ar}-\{A\}}=0$
for $A\in\Xinr$. The proposition is proved.
\end{proof}

Let $\{C_w'\mid w\in \affSr\}$ be the canonical basis of $\afbfHr$ defined in \cite[1.1(c)]{KL79}. Combining Lemma \ref{subalgebra Hr} and Proposition \ref{canonical basis for Unr}, we have the following result.
\begin{Coro}
The set $\{C_w'\mid w\in\affSr,\,w^{-1}(i)>0\text{ for }1\leq i\leq r\}$ forms
a $\sZ$-basis of $\msHr$.
\end{Coro}

\subsection{A realization of $\dbfUn$}
We now prove that the algebra $\afbfLn$ introduced in \S5.1 is a realization of $\dbfUn$.
\begin{Thm}\label{realization of dbfUn}
The linear map $\Phi:\dbfUn\ra\afbfLn$ sending $\pi_{\la,\mu}(u)$ to
$\dob\diag(\la)\dcb u\dob\diag(\mu)\dcb$ for all $u\in\bfUn$ and
$\la,\mu\in\afmbzn$, is an algebra isomorphism. Furthermore we have $\Phi(\dUn)=\afLn$.
\end{Thm}
\begin{proof}
For $u\in\bfUn_\nu$ and $\la\in\afmbzn$ we have $$ [\diag(\la)]u=u[\diag(\la-\nu)].$$
This implies that
$[\diag(\la)]u[\diag(\mu)]=0$ for $u\in\bfUn_\nu$,
$\la,\mu,\nu\in\afmbzn$, with $\nu\not=\la-\mu$.
For $\la',\mu',\la'',\mu''\in\afmbzn$ and  $x\in\bfUn_{\la'-\mu'}$,
$y\in\bfUn_{\la''-\mu''}$,
we have
\begin{equation*}
\begin{split}
\Phi(\pi_{\la',\mu'}(x))\Phi(\pi_{\la'',\mu''}(y))&=
\dt_{\mu',\la''}[\diag(\la')]x[\diag(\mu')]y[\diag(\mu'')]\\
&=
\dt_{\mu',\la''}[\diag(\la')]xy[\diag(\mu'')]\\
&=\Phi(\pi_{\la',\mu'}(x)\pi_{\la'',\mu''}(y)).
\end{split}
\end{equation*}
Hence, $\Phi$ is an algebra homomorphism.
Furthermore,
by \eqref{tri Unr} and Proposition \ref{stabilization property}, we have in $\afLn$,
\begin{equation}\label{dot tri relation}
A^+\dop\bfl\dcp\dob\diag(\bfsg(A))\dcb A^-\dop\bfl\dcp
=\dob A\dcb+f
\end{equation}
for $A\in\tiXin$, where $f$ is a $\sZ$-linear combination of $\dob A'\dcb$ with $A'\sqsubset A$. It follows that the set
$\{\Phi(u_A^+1_\la u_B^-)\mid A\in\afThnp,\,B\in\Thnp,\,\la\in \afmbzn\}$
is a $\sZ$-basis for $\afLn$. Therefore $\Phi$ is an algebra isomorphism. The theorem is proved.
\end{proof}
We will identify $\dbfUn$ with $\afbfLn$ and $\dUn$ with $\afLn$ via the map $\Phi$.

\subsection{Canonical bases for the modified quantum current algebra $\dbfUn$}
We are ready to prove that the canonical bases $\Bnr$ of $\bfUnr$ can be ``glued together'' to form a canonical basis for $\dbfUn$.
\begin{Prop}\label{bar stab}
For $A\in\tiXin$ there exist $C_1,\cdots,C_m\in\tiXin$, elements $H_i(\up,\up')\in\sZ_2$ ($1\leq i\leq m$) and an integer $p_0\geq 0$ such that, in $\UpA$,
$$\ol{[{}_pA]}=\sum_{1\leq i\leq m}H_i(\up,\up^{-p})[{}_pC_i]\quad\text{ for all $p\geq p_0$}.$$
\end{Prop}
\begin{proof}
By \cite[7.3]{DF18}, there exist $C_1,\cdots,C_m\in\aftiThn$, elements $H_i(\up,\up')\in\sZ_2$ ($1\leq i\leq m$) and an integer $p_0\geq 0$ such that, in $\afSpA$,
$$\ol{[{}_pA]}=\sum_{1\leq i\leq m}H_i(\up,\up^{-p})[{}_pC_i]\quad\text{ for all $p\geq p_0$}.$$
Furthermore by Proposition \ref{subalgebra} and Lemma \ref{bar monomial} we have $\ol{\UpA}=\UpA$. Therefore if $C_i\not\in\tiXin$, then $H_i(\up,\up^{-p})=0$ for $p\geq p_0$ and hence $H_i(\up,\up')=0$. The proposition is proved.
\end{proof}

Let $\aftiLnt=\aftiLn\ot\sZ_2$.  By Proposition \ref{bar stab}, there is a ring homomorphism $\bar{}:\aftiLnt\ra\aftiLnt$ such that
$\ol{f(v,v')A}=f(v^{-1},v'^{-1})\sum_{1\leq i\leq m}H_i(\up,\up')C_i$ for $f(v,v')\in\sZ_2$, $A\in\tiXin$  (notation of Proposition \ref{bar stab}).
By definition we have
$$\afLn=\aftiLn\ot_{\sZ_1}\sZ\cong\aftiLnt\ot_{\sZ_2}\sZ,$$ where
$\sZ$ is regarded as a $\sZ_2$-module by specializing $\up'$ to $1$.
The bar involution on $\aftiLnt$ induces a ring involution
\begin{equation}\label{bar on afKn}
\bar\,:\afLn\ra\afLn
\end{equation}
such that
$\ol{v^{j}\dob A\dcb}=v^{-j}\sum_{1\leq i\leq m}H_i(\up,1)\dob C_i\dcb$ for $A\in\tiXin$.

By Lemma \ref{bar monomial} and Proposition \ref{bar stab}, we obtain the following result.
\begin{Lem}\label{bar bimodule}
For $\al\in\afmbnn$, $\beta\in\afmbzn$, if $A=S_\al+\diag(\bt)\in\tiXin$, then $\ol{\dob A\dcb}=\dob A\dcb$ and $\ol{\dob \tA\dcb}=\dob\tA\dcb$.
\end{Lem}

\begin{Thm}\label{canonical basis for afLn}
There exists a unique $\sZ$-basis $\dotBn:=\{\{A\}\mid A\in\tiXin\}$ for $\afLn=\dUn$ such that $\ol{\{A\}}=\{A\}$ and
$\{A\}-\dob A\dcb\in\sum_{B\in\tiXin,B\sqsubset A}v^{-1}\mbz[v^{-1}]\dob B\dcb$.
\end{Thm}
\begin{proof}
By Lemma \ref{tri Wn}, for $A\in\tiXin$ there exists $w_{A^+}\in\ti\Sg$ and $w_{{}^t(A^-)}\in\Ga$ such that
\begin{equation}\label{tri in Ln}
\mpM^{(A)}:=\zeta(\ti u^+_{(w_{A^+})})[\diag(\bfsg(A))]\zeta(\ti u^-_{(w_{{}^t(A^-)})})=[A]
+\sum_{B\sqsubset A\atop B\in\tiXin}h_{A,B}\dob B\dcb,
\end{equation}
where $h_{A,B}\in\sZ$.
This implies that there exist $h_{A,B}'\in\sZ$ such that
$$[A]=\mpM^{(A)}+\sum_{B\in\tiXin\atop B\sqsubset A}h_{A,B}'\mpM^{(B)}.$$
By Lemma \ref{bar bimodule} we have
$\ol{\mpM^{(A)}} =\mpM^{(A)}$.
Consequently,  we have $$\ol{\dob A\dcb}-\dob A\dcb\in\sum_{C\in\aftiThn\atop
C\sqsubset A}\sZ\dob C\dcb.$$ Now the assertion follows from \cite[7.10]{Lu90b}.
\end{proof}
Let $\ddbfHa$ be the modified quantum algebra associated with $\dbfHa$ and let $\ddHa$ be the integral form of $\ddbfHa$ constructed in \cite[(4.0.1)]{Fu13}. Then
$\dUn$ is naturally a subalgebra of $\ddHa$. Let $\dot\bfB(n)$ be the canonical basis of
$\ddHa$ constructed in \cite[7.7]{DF18}. Then
by Theorem \ref{canonical basis for afLn} we have the following result.
\begin{Coro}
We have $\dot\bfB(n)\cap\dUn=\dotBn$.
\end{Coro}
\subsection{The algebra homomorphism $\dzr$}
By Proposition \ref{[B][A]} and \ref{stabilization property}, we obtain the following result (cf. \cite[Lem. 6.4]{DF18}).
\begin{Lem}
There is a surjective algebra homomorphism
\begin{equation}\label{dzr}
\dzr:\dbfUn\ra\bfUnr
\end{equation}
such that
\begin{equation*}\label{dzr([A])}
\dzr([A])=\begin{cases}[A]& \mathrm{if\ }A\in\Xinr;\\
0&  \mathrm{otherwise}\end{cases}
\end{equation*}
for $A\in\tiXin$.
\end{Lem}
We now prove that the algebra homomorphism $\dzr$ preserves
the canonical bases.
\begin{Prop}
We have
\begin{equation}\label{dzr({A})}
\dzr(\{A\})=\begin{cases}\{A\}& \mathrm{if\ }A\in\Xinr;\\
0&  \mathrm{otherwise}\end{cases}
\end{equation}
for $A\in\tiXin$.
In particular we have $\dzr(\dotBn)=\Bnr\cup\{0\}$.
\end{Prop}
\begin{proof}
Clearly we have
\begin{equation}\label{dzr bar commute}
\ol{\dzr(\{A\})}=\dzr(\ol{\{A\}})=\dzr(\{A\})
\end{equation}
 for $A\in\tiXin$. If $A\not\in\Xinr$ then by Theorem \ref{canonical basis for afLn} we have
$$\dzr(\{A\})\in\sum_{B\in\Xinr,B\sqsubset A}v^{-1}\mbz[v^{-1}][B].$$
Hence by \eqref{dzr bar commute} we conclude that $\dzr(\{A\})=0$. Now we assume that $A\in\Xinr$.
Then by Theorem \ref{canonical basis for afLn} we have
$$\dzr(\{A\})-[A]\in\sum_{B\in\Xinr,B\sqsubset A}v^{-1}\mbz[v^{-1}][B]$$
for $A\in\Xinr$. Therefore by \eqref{dzr bar commute} and the uniqueness of canonical bases we conclude that $\dzr(\{A\})=\{A\}$.
\end{proof}

\section{Canonical bases for finite dimensional irreducible graded $\bfUn$-modules}
In this section, we show that the  canonical basis $\dotBn$ of  $\dbfUn$ is well adapted to  finite dimensional graded modules for the quantum current algebra $\bfUn$. In particular, we construct canonical bases for finite dimensional irreducible graded $\bfUn$-modules.

\subsection{The $\mbn$-graded algebras $\bfUnr$, $\bfUn$ and $\dbfUn$}
Recall the algebra $\bfUnr$ introduced in \S4.1.  We first show that the algebra $\bfUnr$ has a natural $\mbn$-grading as an associative algebra.
\begin{Lem}\label{graded algebra bfUnr}
The algebra $\bfUnr$ is a $\mbn$-graded algebra with $$\deg([A])=\sum_{s>0,\,1\leq i,j\leq n}sa_{i,j+sn}$$ for $A\in\Xinr$.
\end{Lem}
\begin{proof}
Let $A\in\Xinr$. Assume that $B=\sum_{1\leq i\leq n}\al_i\afE_{i,i+1}+\diag(\la)$, $C=\sum_{1\leq i\leq n}\ga_i\afE_{i+1,i}+\diag(\mu)$ is such that $\co(B)=\co(C)=\ro(A)$, where $\al,\la,\mu,\ga\in\afmbnn$ with $\sg(\al)+\sg(\la)=r$, $\sg(\ga)+\sg(\mu)=r$ and $\ga_n=0$. By Proposition \ref{subalgebra}, it is enough to show that $\deg([B]\cdot[A])=\deg([B])+\deg([A])$ and $\deg([C]\cdot[A])=\deg([C])+\deg([A])$.

By Proposition \ref{[B][A]} we have
$$[B][A]=\sum_{T\in\afThn,\,\ro(T)=\al \atop
A+T-\ti T\in\Xinr}f_{T,A}[A+T-\ti T]\text{ and }
[C][A]=\sum_{T\in\afThn,\,\ro(T)=\ga\atop
A-T+\ti T\in\Xinr}f'_{T,A}[A-T+\ti T]$$
where $f_{T,A}$ and $f'_{T,A}$ are given as in Proposition \ref{[B][A]}. If $f_{T,A}\not=0$ for some
$T\in\afThn$ with $\ro(T)=\al$ and $A+T-\ti T\in\Xinr$, then we have
\begin{equation*}
\begin{split}
\deg([A+T-\ti T])
&=\deg([A])+\sum_{s>0\atop 1\leq j\leq n}st_{n,j+sn}
-\sum_{s>0\atop 1\leq j\leq n}st_{0,j+sn}\\
&=\deg([A])+\sum_{s\geq 0\atop 1\leq j\leq n}(s+1)t_{n,j+(s+1)n}
-\sum_{s>0\atop 1\leq j\leq n}st_{n,j+(s+1)n}\\
&=\deg([A])+\sum_{j\geq n+1}t_{n,j}.
\end{split}
\end{equation*}
Since $f_{T,A}\not=0$ we have $a_{i,j}+t_{i,j}-t_{i-1,j}\geq t_{i,j}$ for all $i,j$. It follows that $0=a_{1,j}\geq t_{0,j}\geq 0$ for $j<1$ since $A\in\Xinr$. Hence, since $\ro(T)=\al$, we have $\al_n=\sum_{j\in\mbz}t_{n,j}=\sum_{j\geq n+1}t_{n,j}$. This implies that $\deg([A+T-\ti T])=\det([A])+\al_n=\deg([A])+\deg([B]).$
Therefore we have $$\deg([B]\cdot[A])=\deg([A])+\deg([B]).$$
In addition, if $f'_{T,A}\not=0$ for some
$T\in\afThn$ with $\ro(T)=\ga$ and $A-T+\ti T\in\Xinr$, then we have
\begin{equation*}
\begin{split}
\deg([A-T+\ti T])
&=\deg([A])-\sum_{s>0\atop 1\leq j\leq n}st_{n,j+sn}
+\sum_{s>0\atop 1\leq j\leq n}st_{0,j+sn}\\
&=\deg([A])-\sum_{s\geq 0\atop 1\leq j\leq n}(s+1)t_{n,j+(s+1)n}
+\sum_{s>0\atop 1\leq j\leq n}st_{n,j+(s+1)n}\\
&=\deg([A])-\sum_{j\geq n+1}t_{n,j}.
\end{split}
\end{equation*}
Since $\ro(T)=\ga$ and $\ga_n=0$ we conclude that $t_{n,j}=0$ for all $j$. Hence we have
$\deg([A-T+\ti T])=\deg([A])=\deg([A])+\deg([C])$. Consequently, we have $$\deg([C]\cdot[A])=\deg([A])+\deg([C]).$$
The proof is completed.
\end{proof}
Recall from \eqref{bfSr} the $q$-Schur algebra $\bfSr$ can be regarded as a subalgebra of  $\bfUnr$.
For $k\in\mbn$ let $\bfUnr[k]$ be the $k$-th graded piece of $\bfUnr$. Then we have $\bfUnr[0]=\bfSr$.
In a way similar to the proof of Lemma \ref{graded algebra bfUnr}, we obtain the following result.
\begin{Lem}\label{graded algebra dbfUn}
$(1)$
The algebra $\dbfUn$ is a $\mbn$-graded algebra with $$\deg([A])=\sum_{s>0,\,1\leq i,j\leq n}sa_{i,j+sn}$$ for $A\in\tiXin$.

$(2)$
The algebra $\bfUn$ is a $\mbn$-graded algebra with $$\deg(A(\bfj,\la))=\sum_{s>0}\sum_{1\leq i,j\leq n}sa_{i,j+sn}$$ for $A\in\Xinpm$, $\bfj\in\afmbzn$ and $\la\in\afmbnn$.
\end{Lem}

For $k\in\mbn$ let $\dbfUn[k]$ (respectively $\bfUn[k]$)  be the $k$-th graded piece of $\dbfUn$ (respectively $\bfUn$).
Then we have $\dbfUn[0]=\dbfUgln$ and $\bfUn[0]=\bfUgln$.

\subsection{Evaluation maps}
We now construct an evaluation map $\evr$ from $\bfUnr$ to the $q$-Schur algebra $\bfSr$.
\begin{Lem}\label{evr}
There is a surjective algebra homomorphism $\evr:\bfUnr\ra\bfSr$
such that
\begin{equation*}
\evr([A])=
\begin{cases}
[A]&\text{if $A\in\Thnr$},\\
0&\text{otherwise,}
\end{cases}
\end{equation*}
for $A\in\Xinr$.
\end{Lem}
\begin{proof}
Let $A\in\Xinr$ and $B=\sum_{1\leq i\leq n}\al_i\afE_{i,i+1}+\diag(\la)$ be such that $\co(B)=\ro(A)$, where $\al,\la\in\afmbnn$.
If $A\in\Thnr$ and $\al_n=0$ then we have $[B]\cdot[A]\in\bfSr$, and hence $$\evr([B])\evr([A])=[B]\cdot[A]=\evr([B]\cdot[A]).$$
Now we assume either $A\not\in\Thnr$ or $\al_n\not=0$.
By proposition \ref{[B][A]} we have
\begin{equation*}
\begin{split}
\evr([B][A])&=\sum_{T\in\afThn,\,\ro(T)=\al \atop
A+T-\ti T\in\Xinr}f_{T,A}\evr([A+T-\ti T]),
\end{split}
\end{equation*}
where $f_{T,A}$ is given as in Proposition \ref{[B][A]}.
Let $T\in\afThn$ be such that $\ro(T)=\al$, $A+T-\ti T\in\Xinr$ and $f_{T,A}\not=0$.
Since $f_{T,A}\not=0$, we have
\begin{equation}\label{fTA not=0}
a_{k,l}+t_{k,l}-t_{k-1,l}\geq t_{k,l}
\end{equation}
for all $k,l$.
 If $A\not\in\Thnr$ then we have
$a_{i,j}>0$ for some $1\leq i\leq n$ and $j>n$.
If $a_{i,j}+t_{i,j}-t_{i-1,j}>0$ then we have $A+T-\ti T\not\in\Thnr$. If $a_{i,j}+t_{i,j}-t_{i-1,j}=0$, then by \eqref{fTA not=0} we have
$t_{i,j}=0$ and hence $a_{i,j}=t_{i-1,j}$. It follows from \eqref{fTA not=0} that
$a_{i-1,j}+t_{i-1,j}-t_{i-2,j}\geq t_{i-1,j}=a_{i,j}>0$. This implies that $A+T-\ti T\not\in\Thnr$. Therefore we have $$\evr([B]\cdot[A])=0=\evr([B])\evr([A]).$$
Now we assume that $\al_n\not=0$.
By \eqref{fTA not=0} we have $a_{1,l}\geq t_{n,l+n}$ for all $l$.
Since $A\in\Xinr$ we have $a_{1,l}=0$ for $l<1$, and hence $t_{n,l+n}=0$ for $l<1$.
Therefore, since $\ro(T)=\al$ and $\al_n\not=0$, we have $t_{n,s}>0$ for some $s>n$.
It follows from \eqref{fTA not=0} that $a_{n,s}+t_{n,s}-t_{n-1,s}\geq t_{n,s}>0$. Consequently, we have $A+T-\ti T\not\in\Thnr$.
Therefore we have
$$\evr([B]\cdot[A])=0=\evr([B])\evr([A]).$$

Let $C=\sum_{1\leq i\leq n}\ga_i\afE_{i+1,i}+\diag(\mu)$ be such that $\co(C)=\ro(A)$, where $\mu,\ga\in\afmbnn$ with $\sg(\ga)+\sg(\mu)=r$ and $\ga_n=0$.
If $A\in\Thnr$ then we have $$\evr([C]\cdot[A])=[C]\cdot[A]=\evr([C])\evr([A])$$ since $\ga_n=0$. If $A\not\in\Thnr$ then we have $a_{i,j}>0$ for some $1\leq i\leq n$ and $j>n$.
By Proposition \ref{[B][A]} we have
$$[C][A]=\sum_{T\in\afThn,\,\ro(T)=\ga\atop
A-T+\ti T\in\Xinr}f'_{T,A}[A-T+\ti T],$$
where $f'_{T,A}$ is as given in Proposition \ref{[B][A]}.
Let $T\in\afThn$ be such that $\ro(T)=\ga$, $A-T+\ti T\in\Xinr$ and $f_{T,A}'\not=0$.
Since $f_{T,A}'\not=0$, we have
\begin{equation}\label{fTA' not=0}
a_{k,l}-t_{k,l}+t_{k-1,l}\geq t_{k-1,l}
\end{equation}
for all $k,l$. If $a_{i,j}-t_{i,j}+t_{i-1,j}>0$ then we have $A-T+\ti T\not\in\Thnr$.
If $a_{i,j}-t_{i,j}+t_{i-1,j}=0$, then by \eqref{fTA' not=0} we have $t_{i-1,j}=0$ and hence $t_{i,j}=a_{i,j}>0$. It follows from \eqref{fTA' not=0} that
$a_{i+1,j}-t_{i+1,j}+t_{i,j}\geq t_{i,j}>0$. Furthermore since $\ga_n=0$ and $\ro(T)=\ga$, we have $t_{n,s}=0$ for $s\in\mbz$. Hence, since $t_{i,j}>0$, we have $2\leq i+1\leq n$. So we have $A-T+\ti T\not\in\Thnr$, and hence
$$\evr([C]\cdot[A])=0=\evr([C])\evr([A]).$$

By Proposition \ref{subalgebra}, the algebra $\bfUnr$ is generated by the elements like $[B]$, $[C]$ above. Therefore $\evr$ is an algebra homomorphism.
\end{proof}

Recall the algebra $\bfUgln$ defined in \S \ref{quantum affine gln}.
Furthermore let $\dbfUgln$ be the subspace of $\dbfUn$ spanned by the  elements
$u_A^+1_\la u_B^-$ for $A,B\in\Thnp$ and $\la\in\afmbzn$. Then $\dbfUgln$ is the modified quantum group of $\frak{gl}_n$.

By \cite{BLM} the set $\{A(\bfj)\mid A\in\Thnpm,\,\bfj\in\afmbzn\}$ (respectively $\{[A]\mid A\in\tiThn\}$) forms a basis for $\bfUgln$ (respectively $\dbfUgln$), where $\Thnpm=\afThnpm\cap\Thn$ and $$\tiThn=\{A\in\aftiThn\mid a_{i,j}=0,\,\text{ for }1\leq i\leq n,\,j\not\in\{1,2,\cdots,n\}\}.$$
We now construct an evaluation map $\dev$ (respectively $\ev$) from the modified quantum current algebra $\dbfUn$ (respectively the quantum current algebra $\bfUn$) to the modified quantum group $\dbfUgln$ (respectively the quantum group $\bfUgln$).
\begin{Lem}\label{dev}
$(1)$ There is a surjective algebra homomorphism $\dev:\dbfUn\ra\dbfUgln$
such that
\begin{equation*}
\dev([A])=
\begin{cases}
[A]&\text{if $A\in\tiThn$},\\
0&\text{otherwise.}
\end{cases}
\end{equation*}
for $A\in\tiXin$.

$(2)$ There is a surjective algebra homomorphism $\ev:\bfUn\ra\bfUgln$
such that
\begin{equation*}
\ev(A(\bfj))=
\begin{cases}
A(\bfj)&\text{if $A\in\Thnpm$},\\
0&\text{otherwise.}
\end{cases}
\end{equation*}
for $A\in\Xinpm$ and $\bfj\in\afmbzn$.
\end{Lem}
\begin{proof}
The assertion can be proved in a way similar to the proof of Lemma \ref{evr}.
\end{proof}

By \cite{BLM} the set $\{\{A\}\mid A\in\tiThn\}$ (respectively $\{\{A\}\mid A\in\Thnr\}$) is the canonical basis of the modified quantum group $\dbfUgln$ (respectively the $q$-Schur algebra $\bfSr$). We now prove that the evaluation maps $\evr$ and $\dev$ preserve the canonical bases.
\begin{Prop}\label{dev({A})}
$(1)$ For $A\in\tiXin$ we have
\begin{equation*}
\dev(\{A\})=
\begin{cases}
\{A\}&\text{if $A\in\tiThn$},\\
0&\text{otherwise.}
\end{cases}
\end{equation*}

$(2)$ For $A\in\Xinr$ we have
\begin{equation*}
\evr(\{A\})=
\begin{cases}
\{A\}&\text{if $A\in\Thnr$},\\
0&\text{otherwise.}
\end{cases}
\end{equation*}
\end{Prop}
\begin{proof}
Clearly we have
\begin{equation}\label{dev bar commute}
\ol{\dev(\{A\})}=\dev(\ol{\{A\}})=\dev(\{A\})
\end{equation}
for $A\in\tiXin$.
If $A\not\in\tiThn$, then by Theorem \ref{canonical basis for afLn} and Lemma \ref{dev} we have
$$\dev(\{A\})=\dev(\{A\})-\dev([A])\in\sum_{B\sqsubset  A\atop B\in\tiThn}
v^{-1}\mbz[v^{-1}][B].$$
Hence by \eqref{dev bar commute}, we have $\dev(\{A\})=0$. Now we assume that $A\in\tiThn$. Then by Theorem \ref{canonical basis for afLn} and Lemma \ref{dev} we have
$$\dev(\{A\})-[A]=\dev(\{A\})-\dev([A])\in\sum_{B\sqsubset  A\atop B\in\tiThn}
v^{-1}\mbz[v^{-1}][B].$$
Therefore by \eqref{dev bar commute} and the uniqueness of canonical bases we have $\dev(\{A\})=\{A\}$. The assertion (1) follows. The assertion (2) can be proved similarly.
\end{proof}

\subsection{The categories $\msC$, $\msCr$ and $\dmsC$}
A $\bfUn$-module $V$ is called a representation of type $1$ if $V=\oplus_{\la\in\mbzn}V_\la$, where $V_\la=\{w\in V\mid K_iw=v^{\la_i}w,\,\forall i\}$.
Let $\msC$ be the category whose objects are finite dimensional graded $\bfUn$-modules of type $1$, and where the morphisms are graded maps of $\bfUn$-modules. More precisely, an objects of $\msC$ is a finite dimensional $\bfUn$-modules $V$ of type $1$ such that  $V=\bop_{s\geq 0}V[s],$ where $V[s]$ is a subspace of $V$ such that $\bfUn[k]V[s]\han V[s+k]$ for $s,k\in\mbn$. If $V,W\in\text{Ob}\msC$, then
$\Hom_{\msC}(V,W)=\{f\in\Hom_{\bfUn}(V,W)\mid f(V[s])\han W[s]\text{ for }s\in\mbn\}.$
Similarly, for $r\in\mbn$ let $\msCr$ be the category whose objects are finite dimensional graded $\bfUnr$-modules $V$ and where the morphisms are graded maps of $\bfUnr$-modules.

Following \cite[23.1.4]{Lubk}, a $\dbfUgln$-module $M$ is said to be unital if
\begin{itemize}
\item[(a)]
for any $m\in M$ we have $1_\la m=0$ for all but finitely many $\la\in\afmbzn$;
\item[(b)]
for any $m\in M$ we have $\sum_{\la\in\afmbzn}1_\la m=m$.
\end{itemize}
We define a category $\dmsC$ as follows. An object of $\dmsC$ is a finite dimensional $\dbfUn$-module $V$ such that $V|_{\dbfUgln}$ is a unital $\dbfUgln$-module, and $V=\bop_{s\geq 0}V[s],$ where $V[s]$ is a subspace of $V$ such that $\dbfUn[k]V[s]\han V[s+k]$ for $s,k\in\mbn$. If $V,W\in\text{Ob}\dmsC$, then
$\Hom_{\dmsC}(V,W)=\{f\in\Hom_{\dbfUn}(V,W)\mid f(V[s])\han W[s]\text{ for }s\in\mbn\}.$

Each graded $\bfUnr$-module $V$ can be regarded as a graded $\dbfUn$-module via the map $\dzr$ defined in \eqref{dzr}. Hence $\msCr$ can be regarded as a full subcategory of $\dmsC$. We now prove that the two categories $\msC$ and $\dmsC$ are equivalent.

\begin{Lem}\label{equivalent}
The two categories $\msC$ and $\dmsC$ are equivalent.
\end{Lem}
\begin{proof}
If $M\in\dmsC$, then we have $M=\oplus_{\la\in\afmbzn}1_\la M$. We may regard $M$ as a $\bfUn$-module as follows. The action of $u\in\bfUn$ on $M$ is given by $um=(u1_\la)m$ for $\la\in\afmbzn$ and $m\in 1_\la M$. It is easy to see that $M$ is an object of $\msC$. In this way, we see that to give an object of $\dmsC$ is the same as to give an object of $\msC$.
\end{proof}

\subsection{A classification of the simple objects in the categories $\dmsC$ and $\msCr$}
Let $\dmsP$ be the  category of finite dimensional unital $\dbfUgln$-modules with morphisms being maps of $\dbfUgln$-modules. By Lemma \ref{dev}, we may
define a
covariant functor $$\dev:\dmsP\ra\dmsC$$ by the requirements:
$$\dev(V)[0]=V,\,\dev(V)[s]=0,\,s>0,$$
and with $\dbfUn$-action given by
$$uw=\dev(u)w,\,u\in\dbfUn,\,w\in V$$
and
$$\Hom_{\dmsC}(\dev(V),\dev(W))=\Hom_{\dbfUgln}(V,W)$$
(cf. \cite{CG}).
Similarly, let $\msPr$ be the category of finite dimensional modules for the $q$-Schur algebra $\bfSr$ with morphisms being maps of $\bfSr$-modules. By Lemma \ref{evr} we may define a covariant functor $$\evr:\msPr\ra\msCr$$ by the requirements:
$\evr(V)[0]=V,\,\evr(V)[s]=0,\,s>0,$
and with $\bfUnr$-action given by
$uw=\evr(u)w$, for $u\in\bfUnr$, $w\in V,$
and
$\Hom_{\msCr}(\evr(V),\evr(W))=\Hom_{\bfSr}(V,W)$

For $m\in\mbn$ let $\dtau_m$   be the grading shift given by
$$(\dtau_mV)[k]=V[k-m],\ k\in\mbn$$
for $V\in\text{Ob}\dmsC$. Similarly, for $m\in\mbn$ let $\tau_m$  be the grading shift given by
$$(\tau_mV)[k]=V[k-m],\ k\in\mbn$$
for $V\in\text{Ob}\msCr$. Clearly we have $\dtau_mV\in\text{Ob}\dmsC$ (respectively $\tau_mV\in\text{Ob}\msCr$) for $V\in\text{Ob}\dmsC$ (respectively $V\in\text{Ob}\msCr$).

Let $X^+(n)=\{\la\in\mbz^n\mid\la_1\geq\la_2\geq\cdots\geq\la_n\}$.
For $\la\in X^+(n)$ let $L_v(\la)$ be the irreducible $\bfUgln$-module with highest weight $\la$.
The $\bfUgln$-module $L_v(\la)$ can be naturally regarded as a $\dbfUgln$-module.
For $(\la,m)\in X^+(n)\times\mbn$, let
\begin{equation}\label{Lv(la,m)}
L_v(\la,m)=\dtau_m(\dev(L_v(\la)))\in\text{Ob}\dmsC.
\end{equation}
Let $\La^+(n)=\{\la\in\mbnn\mid\la_1\geq\la_2\geq\cdots\geq\la_n\}$ and $\La^+(n,r)=\{\la\in\La^+(n)\mid\sum_{1\leq i\leq n}\la_i=r\}$.
For $(\la,m)\in\La^+(n,r)\times\mbn$, the graded $\bfUnr$-module $\tau_m(\evr(L_v(\la)))$ can be regarded as a graded $\dbfUn$-module via $\dzr$. Clearly $L_v(\la,m)$ is isomorphic to $\tau_m(\evr(L_v(\la)))$ as a graded $\dbfUn$-module for  $(\la,m)\in\La^+(n,r)\times\mbn$. Therefore we have $L_v(\la,m)\in\text{Ob}\msCr$
for $(\la,m)\in\La^+(n,r)\times\mbn$.

\begin{Lem}\label{lem for classification of bfUn-modules}
Let $V$ be a simple object in the category $\dmsC$. Then we have $V=V[k]$ for some $k\in\mbn$.
\end{Lem}
\begin{proof}
Assume that $V\not=V[s]$ for any $s\in\mbn$. We will show that
this leads to a contradiction. Since $V\not=V[s]$ for any $s\in\mbn$, there exist $k'>k\geq 0$ such that $V[k]\not=0$ and $V[k']\not=0$. The subspace $\bop_{s>k}V[s]$ is a nontrivial proper graded $\dbfUn$-submodule of $V$. Hence $V$ is not simple. This is a contradiction.
\end{proof}

\begin{Prop}\label{classification of simple objects in the category dmsC}
$(1)$ The set $\{L_v(\la,m)\mid(\la,m)\in X^+(n)\times\mbn\}$ is a complete set of non-isomorphic simple objects in the category $\dmsC$.

$(2)$ For $r\in\mbn$ the set $\{L_v(\la,m)\mid(\la,m)\in\La^+(n,r)\times\mbn\}$ is a complete set of non-isomorphic simple objects in the category $\msCr$.
\end{Prop}
\begin{proof}
Clearly, $L_v(\la,m)$ is a simple object in $\dmsC$ for $(\la,m)\in X^+(n)\times\mbn$. Let $V$ be a simple object in $\dmsC$. By Lemma \ref{lem for classification of bfUn-modules}, $V=V[m]$ for some $m\in\mbn$. It follows $(\dbfUn)[s]V\han V[s+m]=0$ for $s>0$. Since $\dbfUn$ is the direct sum of $\dbfUgln$ and $\oplus_{s>0}(\dbfUn)[s]$, we conclude that the restriction of $V$ to $\dbfUgln$ is isomorphic to $L_v(\la)$ for some $\la\in X^+(n)$, and hence $V\cong L_v(\la,m)$ as a $\dbfUn$-module. The assertion (1) follows. The assertion (2) can be proved similarly.
\end{proof}

\subsection{Canonical bases for $L_v(\la,m)$}
We end this section by constructing a canonical basis for $L_v(\la,m)$.
For $(\la,m)\in\La^+(n,r)\times\mbn$ we have $$L_v(\la,m)=\bop_{\mu\in\Lanr} L_v(\la,m)_\mu$$
where $L_v(\la,m)_\mu=[\diag(\mu)]L_v(\la,m)$. Clearly we have $\dim L_v(\la,m)_\la=1$. For $(\la,m)\in\La^+(n,r)\times\mbn$ we choose a nonzero vector $w_{\la,m}$ in $L_v(\la,m)_\la$, and let
\begin{equation}\label{Lz(la,m)}
L_\sZ(\la,m)=\dUn w_{\la,m}.
\end{equation}
\begin{Thm}\label{canonical basis for graded modules of quantum current algebra}
For $(\la,m)\in\La^+(n)\times\mbn$ the set $$\dotBn w_{\la,m}-\{0\}=\{\{A\}w_{\la,m}\not=0\mid A\in\tiXin\}$$  is a basis of $L_\sZ(\la,m)$ over $\sZ$, and of  $L_v(\la,m)$ over $\mbq(v)$.
\end{Thm}
\begin{proof}
By Proposition \ref{dev({A})} we have
$$\{A\}w_{\la,m}=\dev(\{A\})w_{\la,m}
=\begin{cases}
\{A\}w_{\la,m}&\text{if $A\in\tiThn$}\\
0&\text{otherwise}
\end{cases}
$$
for $A\in\tiXin$. Furthermore $L_v(\la,m)$ is isomorphic to $L(\la)$ as a $\dbfUgln$-module. Therefore by \cite[Prop. 4.7]{Fu14} we conclude that the set $\dotBn w_{\la,m}-\{0\}$ is a $\sZ$-basis of $L_\sZ(\la,m)$.
\end{proof}

\section{Canonical bases for $\dot\sU(\cgl)$ and $\bar L(\la,m)$}
In this section, we construct a canonical basis for the algebra $\dot\sU(\cgl)$, and a canonical basis for the finite dimensional irreducible graded $\cgl$-module $\bar L(\la,m)$.
\subsection{The graded algebra $\sU(\cgl)$}

Let $\sU(\cgl)$ be the universal enveloping algebra of the current algebra $\cgl$. Clearly, the algebra
$\sU(\cgl)$ has a presentation with
generators $E_{i,j}\ot t^m$ ($1\leq i,j\leq n$, $m\in\mbn$), and relations
$$[E_{i,j}\ot t^a,E_{k,l}\ot t^b]=\dt_{j,k}E_{i,l}\ot t^{a+b}-\dt_{l,i}E_{k,j}\ot t^{a+b}.$$
The current algebra $\cgl$ is a $\mbn$-graded Lie algebra with the grading
given by powers of $t$. Therefore $\sU(\cgl)$ is a $\mbn$-graded algebra.

Let $$\Unc=\Un\ot_\sZ\mbc,$$
where $\mbc$ is regarded as a $\sZ$-module by specializing $v$ to $1$.
Let
\begin{equation}\label{barUnc}
\barUnc=\Unc/\lan K_i -1 \mid 1\leq i\leq n\ran.
\end{equation}
If $x\in\Un$ then $\bar x$ denote the image of $x$ in $\barUnc$.

\begin{Prop}\label{vi}
There is a graded algebra isomorphism $$\vi:\sU(\cgl)\ra\barUnc$$
such that
\begin{equation}\label{eq vi}
\vi(E_{i,j}\ot t^m)=
\begin{cases}
\ol{u^+_{\afE_{i,j'}}} &\text{if $i<j'$}\\
\ol{u^-_{\afE_{j',i}}}&\text{if $i>j'$}\\
\ol{\big[{K_i;0\atop 1}\big]} & \text{if $i=j'$}
\end{cases}
\end{equation}
for $1\leq i,j\leq n$ and $m\in\mbn$,
where $j'=j+mn$.
\end{Prop}
\begin{proof}
By \cite[(6.1.4.1)]{DDF} and \cite[5.3]{Fu13} we see that
there is an algebra homomorphism $\vi:\sU(\cgl)\ra\barUnc$ satisfying
\eqref{eq vi}. In addition, by \cite[6.4(b)]{Lu90} the set
$$\bigg\{\ol{u_A^+}\prod_{1\leq i\leq n}\vi(E_{i,i})^{j_i}\ol{u_B^-}\,\big|\, A\in\afThnp,\,B\in\Thnp,\,\bfj\in\mbnn\bigg\}$$
forms a $\mbc$-basis for $\barUnc$.
This together with \cite[6.1.4(2)]{DDF} implies that  the set
$$\bigg\{\prod_{1\leq i\leq n\atop i<j,\,
j\in\mbz}\bigg(\ol{u_{\afE_{i,j}}^+}\bigg)^{a_{i,j}}\prod_{1\leq i\leq n}\vi(E_{i,i})^{j_i}\prod_{1\leq i<j\leq n}\bigg(\ol{u_{\afE_{i,j}}^-}\bigg)^{b_{i,j}}\,\big|\, A\in\afThnp,\,B\in\Thnp,\,\bfj\in\mbnn\bigg\}$$
forms a $\mbc$-basis for $\barUnc$.
Hence $\vi$ takes the PBW basis of $\sU(\cgl)$ onto the basis of $\barUnc$. It follows that $\vi$ is an algebra isomorphism.
\end{proof}

\subsection{Canonical bases for $\dot\sU(\cgl)$}
For $\la,\mu\in\afmbzn$ we set ${}_\la\sU(\cgl)_\mu=\sU(\cgl)/{}_\la {\mathcal I}_\mu$, where
\begin{equation*}
{}_\la {\mathcal I}_\mu=\sum_{\bfj\in\afmbnn}\big(\prod_{1\leq i\leq n}
E_{i,i}^{j_i}-\prod_{1\leq i\leq n}\la_i^{j_i}\big)
\sU(\cgl)+\sum_{\bfj\in\afmbnn}\sU(\cgl)\big(\prod_{1\leq i\leq n}
E_{i,i}^{j_i}-\prod_{1\leq i\leq n}\mu_i^{j_i}\big).
 \end{equation*}
 Let $\bar\pi_{\la,\mu}:\sU(\cgl)\ra{}_\la\sU(\cgl)_\mu$ be the canonical projection.
Let
$$\dot\sU(\cgl):=\bop_{\la,\mu\in\afmbzn}{}_\la\sU(\cgl)_\mu.$$
As in the case of $\dbfUn$, there is a natural associative $\mbc$-algebra structure on $\dot\sU(\cgl)$ inherited from
that of $\sU(\cgl)$. Since $\sU(\cgl)$ is a $\mbn$-graded algebra, the algebra $\dot\sU(\cgl)$ has a natural $\mbn$-grading as an associative algebra. Let $\bar 1_\la=\bar\pi_{\la,\la}(1)$.

Let
$$\dUnc=\dUn\ot_\sZ\mbc,$$
where $\mbc$ is regarded as a $\sZ$-module by specializing $v$ to $1$.
Let $\dotBnc$ be the image of $\dotBn$ in $\dUnc$,
where $\dotBn$ is given in Theorem \ref{canonical basis for afLn}.
We shall denote the images of $[A]$ in $\dUnc$ by the same letters.
By Theorem \ref{realization of dbfUn}, \ref{canonical basis for afLn} and Proposition \ref{vi}, we obtain the following result.
\begin{Thm}\label{canonical basis for current algebra}
There is a graded algebra isomorphism $\dot\vi:\dot\sU(\cgl)\ra\dUnc$ such
that $$\dot\vi(\bar\pi_{\la,\mu}(E_{i,j}\ot t^m))=[\afE_{i,j+mn}+\diag(\la-\afbse_i)],\quad \dot\vi(\bar 1_\la)=[\diag(\la)]$$
for $1\leq i,j\leq n$, $m\in\mbn$, $\la,\mu\in\afmbzn$ with $i\not=j+mn$ and  $\mu=\la-\afbse_i+\afbse_j$. Furthermore
the set $$\dot\sB(n):=\dot\vi^{-1}(\dotBnc)$$ forms a $\mbc$-basis for $\dot\sU(\cgl)$.
\end{Thm}

The basis $\dot\sB(n)$ of $\dot\sU(\cgl)$ is called a canonical basis for $\dot\sU(\cgl)$.
We shall prove in Theorem \ref{canonical basis for bar L(la,m)} that the canonical basis $\dot\sB(n)$ is well adapted to finite dimensional graded $\cgl$-modules.

Recall the notations $\ddHa$ and $\dot\bfB(n)$ introduced in \S6.3. Let $\ddHac=\ddHa\ot\mbc,$
where $\mbc$ is regarded as a $\sZ$-module by specializing $v$ to $1$.
Let $\dot\bfB(n)_\mbc$ be the image of  $\dot\bfB(n)$ in $\ddHac$.
Let $\dot\sU(\afgl)$ be the modified algebra associated with the universal enveloping algebra $\sU(\afgl)$. It is easy to see that the map $\dot\vi$ defined in Theorem \ref{canonical basis for current algebra} can be extended to an algebra isomorphism
$\dot\vi:\dot\sU(\afgl)\ra\ddHac$. By Theorem \ref{canonical basis for current algebra} we have the following result.
\begin{Coro}
We have $\dot\sB(n)=\dot\vi^{-1}(\dot\bfB(n)_\mbc)\cap\dot\sU(\cgl)$.
\end{Coro}

\subsection{The category $\msCc$}
A $\cgl$-module $V$ is called a weight module if $V=\oplus_{\la\in\mbzn}V_\la$, where $V_\la=\{w\in V\mid E_{i,i}w=\la_iw\,\forall i\}$.
Let $\msCc$ be the category whose objects are finite dimensional graded $\cgl$ weight modules $V$ and where the morphisms are graded maps of $\gl$-modules.
Let $\msPc$ be the  category of finite dimensional $\gl$ weight modules with morphisms being maps of $\gl$-modules. Following \cite{CG} we
define a
covariant functor
$$\evc:\msPc\ra\msCc$$ by the requirements:
$\evc(V)[0]=V,\,\evc(V)[s]=0,\,s>0,$
and with $\cgl$-action given by
$(xt^k)w=\dt_{k,0}xw,\,x\in\gl,\,w\in V$,
and
$\Hom_{\msCc}(\evc(V),\evc(W))=\Hom_{\gl}(V,W)$. For $m\in\mbn$ let $\tau_{m,\mbc}$   be the grading shift given by
$(\tau_{m,\mbc}V)[k]=V[k-m],\ k\in\mbn$
for $V\in\text{Ob}\msCc$. For $\la\in X^+(n)$ let $\bar L(\la)$ be the irreducible $\gl$-module with highest weight $\la$.
For $(\la,m)\in X^+(n)\times\mbn$ let
\begin{equation}\label{bar L(la,m)}
\bar L(\la,m)=\tau_{m,\mbc}(\evc(\bar L(\la)))\in\text{Ob}\msCc.
\end{equation}
By \cite{CG} we have the following result.
\begin{Prop}
The set $\{\bar L(\la,m)\mid(\la,m)\in X^+(n)\times\mbn\}$ is a complete set of non-isomorphic simple objects in the category $\msCc$.
\end{Prop}
\subsection{Canonical bases for  $\bar L(\la,m)$}
Finally we construct a canonical basis for $\bar L(\la,m)$.
For $(\la,m)\in\La^+(n)\times\mbn$, let $$L_\mbc(\la,m)=L_\sZ(\la,m)\ot_\sZ\mbc,$$ where
$L_\sZ(\la,m)$ is as in \eqref{Lz(la,m)} and $\mbc$ is regarded as a $\sZ$-module by specializing $v$ to $1$. By Theorem \ref{canonical basis for current algebra}, we may regard $L_\mbc(\la,m)$ as a graded $\dot\sU(\cgl)$-module.

Clearly the $\sU(\cgl)$-module $\bar L(\la,m)$ can be naturally regarded as a graded $\dot\sU(\cgl)$-module (see Lemma \ref{equivalent}).  For $(\la,m)\in\La^+(n,r)\times\mbn$ we choose a nonzero vector $\bar w_{\la,m}$ in $\bar L(\la,m)_\la$, where $\bar L(\la,m)_\la=\bar 1_\la \bar L(\la,m)$.
\begin{Thm}\label{canonical basis for bar L(la,m)}
$(1)$ For $(\la,m)\in\La^+(n)\times\mbn$ we have
$L_\mbc(\la,m)$ is isomorphic to $\bar L(\la,m)$ as a graded $\dot\sU(\cgl)$-module.

$(2)$ The set $\dot\sB(n)\bar w_{\la,m}-\{0\}$ is a $\mbc$-basis for $\bar L(\la,m)$.
\end{Thm}
\begin{proof}
By \cite{Lu88} there is a $\gl$-module isomorphism $f:\bar L(\la,m)\ra L_\mbc(\la,m)$. Furthermore, by definition we have $\bar L(\la,m)[k]=0$ and $L_\mbc(\la,m)[k]=0$ for $k\not=m$. Therefore $f$ must be a graded $\dot\sU(\cgl)$-module isomorphism. The assertion (1) follows. Now the assertion (2) follows Theorem \ref{canonical basis for graded modules of quantum current algebra} and \ref{canonical basis for current algebra}.
\end{proof}

\section{Relation with quantum affine \(\mathfrak{gl}_n\) and the Yangian \(Y(\mathfrak{gl}_n)\)}

In this section we investigate the relation between representations of the quantum current algebra \(\bfUnC\) and those of the quantum affine algebra \(\afUglC\), as well as the Yangian \(Y(\mathfrak{gl}_n)\). The main result is a rigidity theorem (Theorem~\ref{classification bfUnC}) which establishes a bijection between the finite dimensional polynomial irreducible modules for \(\bfUn\) and those for \(\afUglC\). Moreover, in Section~9.5 we further relate \(\bfUnC\) to the Yangian \(Y(\mathfrak{gl}_n)\).

\subsection{Finite dimensional representations of  $\afUslC$ }
Let $\afUglC$ be the quantum affine  algebra defined by the
generators $\ttx^\pm_{i,s}$
($1\leq i<n$, $s\in\mbz$), $\ttk_i^{\pm1}$ and $\ttg_{i,t}$ ($1\leq
i\leq n$, $t\in\mbz\backslash\{0\}$)
and relations (QLA1)--(QLA7) with $\mbq(\up)$ replaced by
$\mbc$ and $\up$ by $\ttz\in\mbc^*$ with $\ttz^m\neq 1$ for all $m\geq1$.
Let $\afUslC$ be
the subalgebra of $\afUglC$ generated by all $\ttx^\pm_{i,s}$, $\ti\ttk_i^{\pm1}$ and $\tth_{i,t}$ for $1\leq i<n$, $s\in\mbz$ and $t\in\mbz\backslash\{0\}$.

For $1\leq j\leq n-1$ and $s\in\mbz$, define the elements $\ms
P_{j,s}\in\afUslC$ through the generating functions
\begin{equation*}
\begin{split}
& \ms P_j^\pm(u):=\exp\bigg(-\sum_{t\geq
1}\frac{1}{[t]_\ttz}\tth_{j,\pm t} (\ttz u)^{\pm
t}\bigg)=\sum_{s\geq 0}\ms P_{j,\pm s} u^{\pm
s}\in\afUslC[[u,u^{-1}]].
\end{split}
\end{equation*}
Let $\bfUC\fsl$ be the subalgebra of $\afUslC$ generated by the elements
$\ttx_i^+$, $\ttx_i^-$, $\ti\ttk_i^{\pm 1}$ for $1\leq i\leq n-1$. A finite dimensional representation of $\bfUC\fsl$ is said to be of type $1$ if $V=\oplus_{\la\in\mbz^{n-1}}V_\la$, where
$$V_\la=\{x\in V\mid \ti \ttk_ix=\ttz^{\la_i}x, 1\leq i\leq n-1\}.$$
A finite dimensional representation of $\afUslC$ is said to be of type $1$ if $V|_{\bfUC\fsl}$ is of type $1$.
Following \cite[12.2.4]{CPbk}, a nonzero ($\mu$-weight) vector $w\in V$ is called a   pseudo-highest weight vector if there exist some
$P_{j,s}\in\mbc$ such that
$$
\ttx_{j,s}^+w=0,\quad\ms P_{j,s} w=P_{j,s} w,\quad\text{and}\quad
\ti\ttk_jw=\ttz^{\mu_j}w,
$$
for all $1\leq j\leq n-1$ and $s\in\mbz$.

Let $\mbc^*=\mbc\backslash\{0\}$.  For $f(u)=\prod_{1\leq i\leq m}(1-a_iu)\in\mbc[u]$
with $a_i\in\mbc^*$, let
\begin{equation}\label{f^pm(u)}
f^\pm(u)=\prod_{1\leq i\leq m}(1-a_i^{\pm1}u^{\pm1}).
\end{equation}
Let $\Pn$ be the set of $(n-1)$-tuple polynomials $\bfP=(P_1(u),\ldots,P_{n-1}(u))$ such that $P_i(u)\in\mbc[u]$ and the constant term of $P_i(u)$ is $1$ for $1\leq i\leq n-1$.
For
$\bfP=(P_1(u),\ldots,P_{n-1}(u))\in\sP(n)$, define $P_{j,s}\in\mbc$,
for $1\leq j\leq n-1$ and $s\in\mbz$, as in $P_j^\pm(u)=\sum_{s\geq
0}P_{j,\pm s} u^{\pm s}$, where $P_j^\pm(u)$ is defined by
\eqref{f^pm(u)}.

Let $\Icp(\bfP)$ be the left ideal of $\afUslC$ generated by
$\ttx_{i,s}^+ ,\ms P_{i,s}-P_{i,s},$ and $\ti\ttk_i-\ttz^{\mu_i}$, for $1\leq i\leq n-1$ and $s\in\mbz$, where
$\mu_i=\mathrm{deg}P_i(u)$, and define
$$\Mcp(\bfP)=\afUslC/\Icp(\bfP).$$
Then $\Mcp(\bfP)$ has a unique simple quotient, denoted by
$\Lcp(\bfP)$.
The following result is due to Chari--Pressley (see
\cite{CPbk}).

\begin{Thm}\label{classification of irreducible afUslC-modules}
The modules $\Lcp(\bfP)$ with $\bfP\in\Pn$ are all nonisomorphic
finite dimensional irreducible $\afUslC$-modules of  type $1$.
\end{Thm}

\subsection{Finite dimensional representations of $\bfUnslC$}Let $\bfUnslC$ be the subalgebra of
$\afUslC$ generated by the elements $\ttx_{i}^+$, $\ti\ttk_j^{\pm 1}$ and $\ttx_j^-$ for $1\leq i\leq n$ and $1\leq j<n$.
We refer to $\bfUnslC$ as the quantum current algebra of $\frak{sl}_n$.

The Borel subalgebra $\afUslCpz$ of $\afUslC$ is  the subalgebra of $\afUslC$ generated by the elements $\ttx_{i}^+$, $\ti\ttk_i^{\pm 1}$  for $1\leq i\leq n$.
Benkart--Terwilliger \cite{BT} proved that there is a bijection between finite dimensional irreducible $\bfU_\bfv(\widehat{\frak{sl}}_2)$-modules and finite dimensional irreducible
$\bfU_\bfv^{\geq 0}(\widehat{\frak{sl}}_2)$-modules. This result was generalized to an arbitrary quantum affine
algebra   by Bowman \cite{Bo}.
By \cite{Be} and \cite[Prop. 1.3]{BCP} we have the following result.
\begin{Lem}\label{generators}
The algebra $\afUslCpz$ is generated by the elements $\ttx_{i,s}^+$, $\ttx_{i,t}^-$, $\tth_{i,t}$ and $\ti\ttk_{i}^{\pm 1}$ for $1\leq i\leq n-1$, $s\geq 0$, $t>0$.
\end{Lem}

\begin{Coro}
The algebra $\bfUnslC$ is generated by the elements $\ttx_{i,s}^+$, $\ttx_{i,s}^-$, $\ms P_{i,s}$ and $\ti\ttk_{i}^{\pm 1}$ for $1\leq i\leq n-1$, $s\geq 0$.
\end{Coro}

The following result was given in \cite[Th. VI. 3.5]{Kassel}.
\begin{Lem}\label{Kassel}
If $V$ is a finite dimensional irreducible $\bfUC{\frak{sl}_2}$-module of type $1$,
then there exists a basis $w_0,w_1,\cdots,w_d$ for $V$ such that
$\ti\ttk_1w_i=\ttz^{2i-d}w_i$, $\ttx_1^+ w_i=[i+1]_\ttz w_{i+1}$ and $\ttx_1^-w_i=[d-i+1]_\ttz w_{i-1}$ for $0\leq i\leq d$, where $w_{-1}=w_{d+1}=0$.
\end{Lem}

For $\bfP\in\Pn$ let $\bar J(\bfP)$
be the left ideal of $\bfUnslC$ generated by
$\ttx_{i,s}^+$, $\ms P_{i,s}-P_{i,s}$ and $\ti\ttk_i-\ttz^{\mu_i}$ for $1\leq i\leq n-1$, $s\geq 0$, where
$\mu_i=\mathrm{deg}P_i(u)$. Let
$$\bar N(\bfP) =\bfUnslC/\bar J(\bfP).$$
By Lemma \ref{generators}, the $\bfUnslC$-module $\bar N(\bfP)$ has a unique irreducible quotient $\bfUnslC$-module, which is denoted by $\bar V(\bfP)$.

\begin{Prop}\label{restriction 1}
For $\bfP\in\Pn$ the restriction of $\bar L(\bfP)$ to $\bfUnslC$ is isomorphic to $\bar V(\bfP)$ as a $\bfUnslC$-module.
\end{Prop}
\begin{proof}
Let $w_0$ be a pseudo-highest weight vector in $\bar L(\bfP)$. Then
there is a $\bfUnslC$-module homomorphism
$$f:\bar N(\bfP)\ra\bar L(\bfP)$$ such that $f(\bar 1)=w_0$,
where $\bar 1=1+\bar J(\bfP)$.
By \cite{Bo},  the restriction of $\bar L(\bfP)$ to ${\afUslCpz}$ is irreducible.
This implies that  the restriction of
$\bar L(\bfP)$ to $\bfUnslC$ is irreducible. It follows that $f$ is surjective.
 Hence, since $\bar V(\bfP)$ is the unique
irreducible quotient of $\bar N(\bfP)$, we have $\bar V(\bfP) \cong \bar L(\bfP)|_{\bfUnslC}$.
\end{proof}

\begin{Lem}\label{geq 0}
Let $V$ be a finite dimensional $\bfUC\fsl$-module.
If $w_0$ is a nonzero vector in $V$ such that $\ti\ttk_iw_0=\ttz^{\mu_i}w_0$ and $\ttx_i^+w_0=0$ for some $1\leq i\leq n-1$, then we have $\mu_i\geq 0$.
\end{Lem}
\begin{proof}
By \cite[Prop. 5.1]{Jan}, $\ttx_i^-$ is nilpotent on $V$. Hence
there exists $b\geq 1$ such that $(\ttx_i^-)^{(b-1)}w_0\not=0$ and
 $(\ttx_i^-)^{(b)}w_0 =0$. Hence, since $\ttx_i^+(\ttx_i^-)^{(b)}=
 (\ttx_i^-)^{(b)}\ttx_i^++(\ttx_i^-)^{(b-1)}
 \frac{\ti\ttk_i\ttz^{1-b}-\ti\ttk_i^{-1}\ttz^{b-1}}{\ttz-\ttz^{-1}}$,
 we have $0=\ttx_i^+(\ttx_i^-)^{(b)}w_0=[1-b+\mu_i]_\ttz(\ttx_i^-)^{(b-1)}w_0$.
 It follows that $\mu_i=b-1\geq 0$.
\end{proof}

\begin{Lem}\label{Lem for pseudo-highest module}
Let $V$ be a finite dimensional $\bfUnslC$-module of type $1$. Then
there exists a nonzero vector $w_0\in V$ such that
$ \ttx_{i,s}^+w_0=0$, $\ms P_{i,s} w_0=P_{i,s}x_0$
and $\ti\ttk_iw_0=\ttz^{\mu_i}w_0$
 for $1\leq i< n$ and $s\geq 0$, where  $\mu_i\geq 0$ and $P_{i,s}\in\mbc$.
\end{Lem}
\begin{proof}
Let $$V^0=\{w\in V\mid \ttx_{i,s}^+w=0,\,\text{for }1\leq i<n,\,s\in\mbn\}.$$
Assume for a contradiction that $V^0=0$. Let $w$ be a non-zero joint eigenvector of $\ti\ttk_1,\cdots,\ti\ttk_{n-1}$. Since $V^0=0$, there exist $1\leq i_s<n$ and $j_s\in\mbn$ ($s\geq 1$) such that $\ttx_{i_s,j_s}^+\ttx_{i_{s-1},j_{s-1}}^+\cdots \ttx_{i_1,j_1}w\not=0$ for all $s\geq 1$. Since the vectors
$w$, $\ttx_{i_1,j_1}^+w$, $\ttx_{i_2,j_2}^+\ttx_{i_1,j_1}^+w$, $\cdots$,
have different weights for the action of $\ti\ttk_1,\cdots,\ti\ttk_{n-1}$, they
are linearly independent. This contradicts the finite dimensionality of $V$.

Clearly we have $\ms P_{i,s}V^0\han V^0$ and $\ti\ttk_{i}V^0\han V^0$
for $1\leq i\leq n-1$ and $s\geq 0$. Since the elements $\ms P_{i,s}$
and $\ti\ttk_i$ commute with each other, there exists
$w_0\in V^0$ such that
$\ms P_{i,s} w_0=P_{i,s}x_0$
and $\ti\ttk_iw_0=\ttz^{\mu_i}w_0$
 for $1\leq i< n$ and $s\geq 0$, where  $\mu_i\in\mbz$ and $P_{i,s}\in\mbc$.
 By Lemma \ref{geq 0} we have $\mu_i\geq 0$ for $1\leq i\leq n-1$.
The proof is completed.
\end{proof}

We shall say that a representation $V$ of $\bfUnslC$ is of type $1$ if
$V|_{\bfUC\fsl}$ is of type $1$.
We say that a $\bfUnslC$-module $V$ is a
pseudo-highest weight module with highest weight $\mu$ if
$V=\afUslC w$ for some
nonzero vector $w\in V_\mu$ and there exist
$P_{j,s}\in\mbc$  such that
$\ttx_{j,s}^+w=0$ and $\ms P_{j,s} w=P_{j,s} w,$
for all $1\leq j\leq n-1$ and $s\geq 0$.
By Lemma \ref{Lem for pseudo-highest module} we obtain the following result.
\begin{Prop}\label{pseudo-highest module}
Every finite dimensional irreducible $\bfUnslC$-module of type $1$ is a pseudo-highest weight module with highest weight $\mu\in\mbn^{n-1}$.
\end{Prop}

Let $N^+=\sum_{1\leq i< n,\,s\geq 0}\bfUnslC\ttx_{i,s}^+$.
The following result was given in \cite[Lem. 12.2.7]{CPbk}.
\begin{Lem}\label{CP (i),(iii)}
Let $1\leq i\leq n-1$ and $m\geq 1$. Then we have
$\ms P_{i,m}\equiv(-1)^m\ttz^{m^2}\ttx_{i,0}^{+(m)}
\ttx_{i,1}^{-(m)}\ti\ttk_i^{-m} \pmod {N^+}$.
\end{Lem}
In a way similar to the proof of \cite[Lem. 12.2.7]{CPbk} we obtain the following result.
\begin{Lem}\label{(i),(iii)}
Let $1\leq i\leq n-1$ and $m\geq 1$. Then we have $\ms P_{i,m}\equiv(-1)^m\ttz^{m^2}\ttx_{i,1}^{+(m)}
\ttx_{i,0}^{-(m)}\ti\ttk_i^{-m} \pmod {N^+}$.
\end{Lem}

Recall the Cartan matrix $C$ of affine type $A_{n-1}$ defined in \S2.1.
For $\la\in\mbz^{n-1}$ and $1\leq i\leq n-1$ let
 $s_i(\la)=\la-\la_i(c_{i,1},c_{i,2},\cdots,c_{i,n-1})$.

\begin{Prop}\label{restriction 2}
Let $V$ be a finite dimensional irreducible $\bfUnslC$-module of type $1$ with highest weight $\mu$. Suppose that $(\ttx_{i,1}^-)^{\mu_i}V_\mu\not=0$ for $1\leq i\leq n-1$.
Then the action of $\bfUnslC$ on  $V$
extends uniquely to an action of $\afUslC$ on $V$. The resulting $\afUslC$-module structure on $V$ is irreducible and of type $1$.
\end{Prop}
\begin{proof}
By Proposition \ref{pseudo-highest module} we have
$\dim V_\mu=1$ and $\mu\in\mbn^{n-1}$.
Let $w_0$ be a non-zero vector in $V_\mu$. Then we have
\begin{equation}\label{eq restriction 2}
 \ttx_{i,s}^+w_0=0,\quad\ms P_{i,s} w_0=P_{i,s}x_0,
\end{equation}
for $1\leq i< n$ and $s\geq 0$, where $\mu_i\geq 0$ and $P_{i,s}\in\mbc$. Let $\bfP=(P_1(u),P_2(u),\cdots,P_{n-1}(u))$, where $P_i(u)=\sum_{s\geq 0}P_{i,s}u^s$.
By \cite[Cor. 10.1.6]{CPbk}, the $\bfUC\fsl$-submodule of $V$ generated by $w_0$
is an irreducible $\bfUC\fsl$-module with highest weight $\mu$.
 Therefore by Lemma \ref{Kassel} we have
\begin{equation}
(\ttx_{i,0}^-)^{(\mu_i)}w_0\not=0,\quad (\ttx_{i,0}^-)^{(b)}w_0=0
\end{equation}
for $1\leq i<n$ and $b>\mu_i$.
It follows from Lemma \ref{(i),(iii)} that
\begin{equation}\label{polynomial}
 P_{i,s}=0
\end{equation}
for $1\leq i<n$ and $s>\mu_i$. Furthermore since $(\ttx_{i,1}^-)^{\mu_i}V_\mu\not=0$ and $\dim V_{s_i(\mu)}=\dim V_\mu=1$, we have
$(\ttx_{i,1}^-)^{(\mu_i)}w_0$ is a non-zero multiple of $(\ttx_{i,0}^-)^{(\mu_i)}w_0$.
Hence by Lemma \ref{CP (i),(iii)} we have $P_{i,\mu_i}\not=0$ for $1\leq i\leq n-1$.
So   $P_i(u)$ is a polynomial in $\mbc[u]$ of degree $\mu_i$ for $1\leq i<n$, and hence $\bfP \in\Pn$.
By \eqref{eq restriction 2}, we see that $V$ is an irreducible  quotient module of
$\bar N(\bfP)$. Hence by Proposition \ref{restriction 1} we have $\bar L(\bfP)|_{\bfUnslC}\cong\bar V(\bfP) \cong V$. The proof is completed.
\end{proof}

\subsection{Polynomial representations of $\afUglC$}

For $1\leq i\leq n$ and $s\in\mbz$, define the elements $\ms
Q_{i,s}\in\afUglC$ through the generating functions
\begin{equation*}
\begin{split}
&\quad\qquad\ms Q_i^\pm(u):=\exp\bigg(-\sum_{t\geq
1}\frac{1}{[t]_\ttz}\ttg_{i,\pm t} (\ttz u)^{\pm t}\bigg)=\sum_{s\geq
0}\ms Q_{i,\pm s} u^{\pm s}\in\afUglC[[u,u^{-1}]].
\end{split}
\end{equation*}
Let $\UglC$ be the subalgebra of $\afUglC$ generated by the elements
$\ttx_i^+$, $\ttx_i^-$, $\ttk_j^{\pm 1}$ for $1\leq i\leq n-1$ and $1\leq j\leq n$.
Let $V$ be a finite dimensional representation of $\afUglC$ such that the restriction of $V$ to  $\UglC$ is a polynomial representation
of type $1$. Then $V=\oplus_{\la\in\mbn^{n }}V_\la$, where
$$V_\la=\{x\in V\mid  \ttk_ix=\ttz^{\la_i}x, 1\leq i\leq n \},$$
and  each $V_\la$ is a direct sum
of generalized eigenspaces of the form
\begin{equation}\label{geigenspace}
V_{\la,\gamma}=\{x\in V_\la\mid (\ms
Q_{i,s}-\gamma_{i,s})^px=0\text{ for some $p$}\, (1\leq i\leq
n ,s\in\mbz)\},
\end{equation}
 where $\gamma=(\gamma_{i,s})$ with $\gamma_{i,s}\in\mbc$.
Let  $\Gamma_i^\pm(u)=\sum_{s\geq 0}\gamma_{i,\pm s} u^{\pm
s}$ for $1\leq i\leq n$.

A finite dimensional $\afUglC$-module $V$ is called a
polynomial representation if the restriction of $V$ to
 $\UglC$ is a polynomial representation
of type $1$ and, for every weight $\la=(\la_1,\ldots,\la_n)\in\mbn^n$
of $V$, the formal power series $\Gamma_i^\pm(u)$ associated to the
eigenvalues $(\gamma_{i,s})_{s\in\mbz}$ defining the generalized
eigenspaces $V_{\la,\gamma}$ as given in \eqref{geigenspace}  are
polynomials in $u^\pm$ of degree $\la_i$ so that the zeroes of the
functions $\Gamma_i^+(u)$ and $\Gamma_i^-(u)$ are the same.

Following \cite{FM}, an $n$-tuple of polynomials
$\bfQ=(Q_1(u),\ldots,Q_n(u))$ with constant terms $1$ is called {\it
dominant} if, for each $1\leq i\leq n-1$, the ratio
$Q_i(\ttz^{i-1}u)/Q_{i+1}(\ttz^{i+1}u)$ is a polynomial. Let
$\sQ(n)$ be the set of dominant $n$-tuples of polynomials.

For $\bfQ=(Q_1(u),\ldots,Q_{n}(u))\in\sQ(n)$, define
$Q_{i,s}\in\mbc$, for $1\leq i\leq n$ and $s\in\mbz$, by the
following formula
$$Q_i^\pm(u)=\sum_{s\geq 0}Q_{i,\pm s}u^{\pm s},$$
where $Q_i^\pm(u)$ is defined using \eqref{f^pm(u)}. Let $I(\bfQ)$
be the left ideal of $\afUglC$ generated by $\ttx_{j,s}^+ ,\ms
Q_{i,s}-Q_{i,s},$ and $\ttk_i-\ttz^{\la_i}$, for $1\leq j\leq n-1$,
$1\leq i\leq n$, and $s\in\mbz$, where $\la_i=\mathrm{deg}Q_i(u)$,
and define
$$M(\bfQ)=\afUglC/I(\bfQ).$$
Then $M(\bfQ)$ has a unique simple quotient, denoted by $L(\bfQ)$.
The polynomials $Q_i(u)$ are called {\it Drinfeld
polynomials}\index{Drinfeld polynomials} associated with $L(\bfQ)$.\index{simple representation! $\sim$ of $\afUglC$, $L(\bfQ)$}

\begin{Thm}[\cite{FM}]\label{classification of simple afUglC-modules}
The $\afUglC$-modules $L(\bfQ)$ with $\bfQ\in\sQ(n)$ are all
nonisomorphic finite dimensional simple polynomial representations
of $\afUglC$. Moreover,
$$L(\bfQ)|_{\afUslC}\cong \bar L(\bfP),$$ where
$\bfP=(P_1(u),\ldots,P_{n-1}(u))$ with
$P_i(u)=Q_i(\ttz^{i-1}u)/Q_{i+1}(\ttz^{i+1}u)$.
\end{Thm}

\subsection{Polynomial representations of $\bfUnC$}

Let
$\bfUnC$ be the subalgebra of $\afUglC$ generated by the elements
$\ttx_{i}^+$, $\ttx_{j}^-$, $\ttk_i^{\pm 1}$ and $\th_{s}$ for  $1\leq i\leq n$, $1\leq j<n$ and $s\geq 1$.
By Lemma \ref{generators} we have the following result.
\begin{Lem}\label{generators of bfUnC}
The algebra $\bfUnC$ is generated by the elements $\ttx_{j,s}^+$, $\ttx_{j,s}^-$, $\ttg_{i,t}$ and $\ttk_{i}^{\pm 1}$ for $1\leq j\leq n-1$, $1\leq i\leq n$, $s\geq 0$ and $t>0$.
\end{Lem}

Let $V$ be a finite dimensional $\bfUnC$-module such that the
restriction of $V$ to $\UglC$ is a polynomial
representation of type $1$. Thus
\[
 V=\bigoplus_{\lambda\in\mathbb N^n}V_\lambda,
 \qquad
 V_\lambda=
 \{x\in V\mid \ttk_i x=\ttz^{\lambda_i}x,\ 1\leq i\leq n\}.
\]
Since the elements $\ms Q_{i,s}$ commute, each $V_\lambda$ is a direct
sum of simultaneous generalized eigenspaces
\[
 V_{\lambda,\gamma}
 =
 \left\{
 x\in V_\lambda\ \middle|\
 (\ms Q_{i,s}-\gamma_{i,s})^p x=0
 \text{ for some }p ,\ 1\leq i\leq n,\ s\geq0
 \right\},
\]
where $\gamma=(\gamma_{i,s})$ with $\gamma_{i,s}\in\mbc$.
Let
\[
 \Gamma_i(u)
 =
 \sum_{s\geq0}\gamma_{i,s}u^s
 \qquad(1\leq i\leq n).
\]

A finite dimensional $\bfUnC$-module $V$ is called a
 polynomial representation if
$V|_{\UglC}$ is a polynomial representation of
type $1$ and, for every nonzero generalized eigenspace
$V_{\lambda,\gamma}$, the formal power series $\Gamma_i(u)$ are
polynomials satisfying
\[
 \deg\Gamma_i(u)\leq\lambda_i
 \qquad(1\leq i\leq n).
\]
A polynomial representation $V$ is called
\emph{affine-polynomial} if
\[
 \deg\Gamma_i(u)=\lambda_i
 \qquad(1\leq i\leq n)
\]
for every nonzero $V_{\lambda,\gamma}$.

For $\bfQ\in\Qn$ let $J(\bfQ)$
be the left ideal of $\bfUnC$ generated by
$\ttx_{j,s}^+$, $\ms Q_{i,s}-Q_{i,s}$ and $\ttk_i-\ttz^{\la_i}$ for $1\leq i\leq n$, $1\leq j\leq n-1$, $s\geq 0$, where
$\la_i=\mathrm{deg}Q_i(u)$. Let
$$N(\bfQ) =\bfUnC/J(\bfQ).$$
By Lemma \ref{generators of bfUnC}, the $\bfUnC$-module $N(\bfQ)$ has a unique irreducible quotient $\bfUnC$-module, which is denoted by $V(\bfQ)$.

\begin{Prop}\label{restriction 1 gl}
For $\bfQ\in\Qn$ the restriction of $L(\bfQ)$ to $\bfUnC$ is isomorphic to $ V(\bfQ)$ as a $\bfUnC$-module.
\end{Prop}
\begin{proof}
By \cite{Bo},  the restriction of $L(\bfQ)$ to ${\afUslCpz}$ is irreducible.
It follows that the restriction of $L(\bfQ)$ to $\bfUnC$ is irreducible.
Hence $L(\bfQ)|_{\bfUnC}$
is an irreducible quotient module of $N(\bfQ)$. Therefore we have $V(\bfQ) \cong  L(\bfQ)|_{\bfUnC}$.
\end{proof}

\begin{Lem}\label{restriction of irreducible modules}
Let $V$ be a finite dimensional polynomial irreducible  $\bfUnC$-module of type $1$. Then
the restriction of $V$ to $\bfUnslC$ is irreducible.
\end{Lem}
\begin{proof}
Put $D=\ttk_1\cdots \ttk_n$. Since $D$ is central and $V$ is
irreducible, $D$ acts on $V$ as a scalar. Write
$ V=\bigoplus_{\lambda\in\mbnn}V_\lambda$
for the type $1$ weight decomposition of $V$. Then
$\lambda_1+\cdots+\lambda_n$ is independent of
$\lambda\in\mbnn$. Since $v$ is not a root of unity, two weights
$\lambda,\mu\in\Lambda$ have the same joint eigenvalues for
$\widetilde \ttk_1,\ldots,\widetilde \ttk_{n-1}$ if and only if
$
 \lambda_i-\lambda_{i+1}=\mu_i-\mu_{i+1}
 \ (1\leq i<n).$
Together with
$
 \lambda_1+\cdots+\lambda_n=\mu_1+\cdots+\mu_n,$
this implies $\lambda=\mu$. Thus the joint weight spaces for
$\widetilde \ttk_1,\ldots,\widetilde \ttk_{n-1}$ are precisely
$V_\lambda$.

Let $0\ne W\subseteq V$ be a
$\bfUnslC$-submodule. Since
$\widetilde \ttk_1,\ldots,\widetilde \ttk_{n-1}$ act semisimply on $V$, we have
$ W=\bigoplus_{\lambda\in\mbnn}(W\cap V_\lambda).$
Therefore $W$ is stable under every $\ttk_i^{\pm1}$. Since each
$\theta_s$ is central, it acts on $V$ as a scalar. Hence $W$ is stable
under $\bfUnC$. The irreducibility of $V$ now gives
$W=V$.
\end{proof}
We first retain the rigidity result for the affine-polynomial
subclass.
\begin{Prop}\label{restriction 2 gl}
Let $V$ be a finite dimensional affine-polynomial irreducible $\bfUnC$-module of type $1$. Then $V$ admits a unique extension to a polynomial $\afUglC$-module. The resulting $\afUglC$-module structure on $V$ is a polynomial irreducible $\afUglC$-module of type $1$.
\end{Prop}
\begin{proof}
By the proof of \cite[Prop. 12.2.3]{CPbk} we see that there exists $w_0\in V$ such that
\begin{equation}\label{eq restriction 2 gl}
 \ttx_{i,s}^+w_0=0,\quad\ms Q_{j,s} w_0=Q_{j,s}w_0,\quad \ttk_jw_0=\ttz^{\la_j}w_0
\end{equation}
for $1\leq i< n$, $1\leq j\leq n$ and $s\geq 0$, where $\la_j\geq 0$ and $Q_{j,s}\in\mbc$.
Let $\bfQ=(Q_1(u),\ldots,Q_{n}(u)) $, where
 $Q_j (u)=\sum_{s\geq 0}Q_{j, s}u^{ s}$ for $1\leq j\leq n$.
Since $V$ is a affine-polynomial irreducible $\bfUnC$-module of type $1$, by Lemma \ref{restriction of irreducible modules}  $Q_i(u)$ is a polynomial of degree $\la_i$ for $1\leq i\leq n$
and the restriction of $V$ to $\bfUnslC$ is irreducible.
Hence by \eqref{polynomial}, $P_i(u)$ is a polynomial for $1\leq i\leq n-1$, where
$P_i(u)=Q_i(\ttz^{i-1}u)/Q_{i+1}(\ttz^{i+1}u)$. Therefore $\bfQ\in\Qn$.
By \eqref{eq restriction 2 gl}
we see that $V$ is an irreducible  quotient module of
$ N(\bfQ)$. Hence by Proposition \ref{restriction 1 gl} we have $ L(\bfQ)|_{\bfUnC}\cong V(\bfQ) \cong V$. The proof is completed.
\end{proof}

Combining Theorem \ref{classification of simple afUglC-modules},
Proposition \ref{restriction 1 gl} and \ref{restriction 2 gl} we obtain the following result.
\begin{Thm}\label{classification bfUnC}
The modules $ V(\bfQ)$ with $\bfQ\in\Qn$ are all nonisomorphic
finite dimensional affine-polynomial irreducible $\bfUnC$-modules.
Thus there is a bijection between finite dimensional affine-polynomial irreducible $\bfUnC$-modules and finite dimensional polynomial irreducible  $\afUglC$-modules.
\end{Thm}

We now classify all polynomial irreducible $\bfUnC$-modules.
Recall that $\La^+(n)=\{\la\in\mbnn\mid\la_1\geq\la_2\geq\cdots\geq\la_n\}$, and that, for $\la\in\Lambda^+(n)$,
$L_v(\la)$ denotes the finite dimensional irreducible
polynomial $U_v(\mathfrak{gl}_n)$-module of highest weight
$\la$. Via the evaluation homomorphism
$ \ev:\bfUnC\rightarrow \bfU_\ttz(\mathfrak{gl}_n)$
constructed in Lemma~\ref{dev}, we regard $L_v(\la)$ as a
$\bfUnC$-module and denote it by $E(\la)$.
Since $\ev$ annihilates the strictly positive homogeneous part
of $\bfUnC$, we have
\[
 \ttg_{i,t}E(\la)=0
 \qquad(1\leq i\leq n,\ t>0).
\]
Thus $E(\la)$ is a
polynomial $\bfUnC$-module.

\begin{Lem}\label{tensor evaluation polynomial}
Let $V$ be a finite dimensional affine-polynomial
$\bfUnC$-module and let
$\la\in\Lambda^+(n)$. Then
$ V\otimes E(\la)$
is a polynomial $\bfUnC$-module.
\end{Lem}
\begin{proof}
By the argument of \cite[Lemma~4.1]{FM}, the joint generalized
eigenvalues of the elements $\ttk_i$ and $\ms Q_i^+(u)$ are
multiplicative under tensor products. Since $\ms Q_i^+(u)$ acts as
$1$ on $E(\lambda)$, if $V_{\mu,\gamma}$ has associated series
$\Gamma_i(u)$ and $E(\lambda)_\nu\neq0$, then the corresponding
generalized eigenvalue on $V\otimes E(\lambda)$ has weight
$\mu+\nu$ and associated series $\Gamma_i(u)$. Hence
$ \deg\Gamma_i(u)=\mu_i\leq\mu_i+\nu_i,$
since $V$ is affine-polynomial and $\nu\in\mathbb N^n$.
Thus $V\otimes E(\lambda)$ is polynomial.
\end{proof}

For
$\bfQ \in\Qn$ and  $\mu\in\Lambda^+(n),$
let $J(\bfQ,\mu)$ be the left ideal of $\bfUnC$ generated
by
$ \ttx_{j,s}^+,\  \ms Q_{i,s}-Q_{i,s},\
 \ttk_i-\ttz^{\la_i+\mu_i},$
for $1\leq j<n$, $1\leq i\leq n$ and $s\geq0$, where $\la_i=\deg Q_i(u)$,
and define
\[
 N(\bfQ,\mu)
 =
 \bfUnC/J(\bfQ,\mu).
\]
By Lemma \ref{generators of bfUnC}, the $\bfUnC$-module $ N(\bfQ,\mu)$ has a unique irreducible quotient $\bfUnC$-module, which is denoted by $V(\bfQ,\mu)$.
Clearly,
$ V(\bfQ,0)=V(\bfQ),$
and $V(\bfQ,\mu)$ is affine-polynomial if and only if
$\mu=0$.
\begin{Thm}\label{classification polynomial bfUnC}
For
$ (\bfQ,\mu)\in\Qn\times\Lambda^+(n),$
the module $V(\bfQ,\mu)$ is finite dimensional and
polynomial.
Moreover, the modules
$ V(\bfQ,\mu)$ with
$(\bfQ,\mu)\in\Qn\times\Lambda^+(n)$  are all nonisomorphic finite dimensional
polynomial irreducible $\bfUnC$-modules.
\end{Thm}

\begin{proof}
Let $ (\bfQ,\mu)\in\Qn\times\Lambda^+(n)$.
Let $w_{\bfQ}$ be a pseudo-highest weight vector of $V(\bfQ)$
and let $w_\mu$ be a highest weight vector of $E(\mu)$.
Put
\[
 w=w_{\bfQ}\otimes w_\mu\in V(\bfQ)\ot E(\mu).
\]
Then we have $
 x_{i,s}^+w=0$ for $1\leq i< n$ and $s\geq 0$.
Since
$\ms Q_i^+(u)w_\mu=w_\mu$, by the argument of \cite[Lem.~4.1]{FM},
we have
\[
\ms Q_i^+(u)w=Q_i(u)w,
 \qquad
 \ttk_iw=\ttz^{\la_i+\mu_i}w
\]
for $1\leq i\leq n$
where $
\la_i=\deg Q_i(u)$.
Hence $ \bfUnC w$
is a nonzero finite dimensional quotient of
$N(\bfQ,\mu)$.
By Theorem~\ref{classification bfUnC} and Lemma~\ref{tensor evaluation polynomial}
$ V(\bfQ)\otimes E(\mu)$
is polynomial. Hence
$\bfUnC w$ is a finite dimensional polynomial
quotient of $N(\bfQ,\mu)$. Since $N(\bfQ,\mu)$ has a unique
irreducible quotient, $V(\bfQ,\mu)$ is finite dimensional and
polynomial.

Conversely, let $V$ be a finite dimensional polynomial
irreducible $\bfUnC$-module. By the argument in the proof of
\cite[Prop.~12.2.3]{CPbk}, there exists a nonzero vector
$w_0\in V$ such that
\begin{equation}\label{eq polynomial module}
 \ttx_{i,s}^+w_0=0,\qquad
\ms Q_{j,s}w_0=Q_{j,s}^Vw_0,\qquad
  \ttk_jw_0=\ttz^{\nu_j}w_0
\end{equation}
for $1\leq i< n$, $1\leq j\leq n$ and $s\geq 0$.
Put
\[
 Q_j^V(u)=\sum_{s\geq0}Q_{j,s}^Vu^s,
 \qquad
\bt_j=\deg Q_j^V(u).
\]
Since $V$ is polynomial,
$
\bt_j\leq\nu_j $ for all $j$.
By Lemma~\ref{restriction of irreducible modules},
$V|_{\bfU_\ttz(\mathfrak{sl}_n[t])}$ is irreducible.  Hence by \eqref{polynomial},
\[
 P_i(u)
 =
 \frac{Q_i^V(\ttz^{i-1}u)}
      {Q_{i+1}^V(\ttz^{i+1}u)}
 \in\mathbb C[u]
\]
and
$
\bt_i-\bt_{i+1}=\deg P_i(u)
 \leq
 \nu_i-\nu_{i+1} $ for all $i$.
Thus
$
 \bfQ^V=(Q_1^V,\ldots,Q_n^V)\in\Qn.$
Let $\dt=\nu-\bt$. Then we have $
 \dt \in\Lambda^+(n).$
By~\eqref{eq polynomial module}, $V$ is an irreducible quotient of
$N(\bfQ^V,\dt)$, and hence
$ V\cong V(\bfQ^V,\dt).$ The proof is completed.
\end{proof}

\subsection{Relation with the Yangian \(Y(\mathfrak{gl}_n)\)}

In this subsection we further clarify the relationship between the quantum current algebra \(\bfUnC\) and the Yangian \(Y(\mathfrak{gl}_n)\), the standard quantization of the current algebra \(\mathfrak{gl}_n[t]\) due to Drinfeld.

Let $\Irr_{\rm pol}(\bfUnC)$ denote the set of isomorphism
classes of finite dimensional polynomial irreducible
$\bfUnC$-modules, and let
$\Irr_{\rm pol}(Y(\mathfrak{gl}_n))$ denote the corresponding
set for the Yangian.
For a nonzero polynomial $f(u)$, write
\[
 \operatorname{mon}(f)
 =
 \frac{f(u)}{\operatorname{lc}(f)}
\]
for its monic normalization, where $\operatorname{lc}(f)$ is the leading coefficient of $f$.

\begin{Thm} \label{thm:bijection}
There is a bijection between the set of isomorphism classes of finite dimensional polynomial irreducible modules over the quantum current algebra \(\bfUnC\) and the corresponding set for the Yangian \(Y(\mathfrak{gl}_n)\):
\[
\operatorname{Irr}_{\mathrm{pol}}\bigl(\bfUnC)\bigr) \;\longleftrightarrow\; \operatorname{Irr}_{\mathrm{pol}}\bigl(Y(\mathfrak{gl}_n)\bigr).
\]
More precisely, let
$(\bfQ,\mu) \in
 \Qn\times\Lambda^+(n)
$
be the parameters of
$V(\bfQ,\mu)$ in
Theorem~\ref{classification polynomial bfUnC}.
Then the corresponding polynomial irreducible
$Y(\mathfrak{gl}_n)$-module is parametrized by the monic
polynomials
$ P_i^Y(u)= u^{b_i}\operatorname{mon}(P_i(u))
 \ (1\leq i<n),$
together with
$Q_n^Y(u) = u^{b_n}\operatorname{mon}(Q_n(u))$,
where
$ P_i(u) = \frac{Q_i(v^{i-1}u)}{Q_{i+1}(v^{i+1}u)}
 \ (1\leq i<n)$,  $b_i=\mu_i-\mu_{i+1}
 \ (1\leq i<n),  \ b_n=\mu_n.$
\end{Thm}

\begin{proof}
By the finite dimensional highest weight classification for the Yangian \(Y(\mathfrak{gl}_n)\) recalled in
\cite[\S1.3]{KNS} and  \cite[Theorem~1.2]{KNS}, finite dimensional polynomial irreducible \(Y(\mathfrak{gl}_n)\)-modules are
are parametrized by
$
 (P_1^Y(u),\ldots,P_{n-1}^Y(u),Q_n^Y(u)),$
where all these polynomials are monic. Now the assertion follows from
Theorem~\ref{classification polynomial bfUnC}.
\end{proof}

Thus, the finite dimensional polynomial irreducible modules of \(\bfUnC\) are in canonical bijection with those of the Yangian \(Y(\mathfrak{gl}_n)\). The advantage of \(\bfUnC\) is that it admits a canonical basis (Theorem~\ref{canonical basis for afLn}) and an integral form over \(\sZ=\mathbb{Z}[v,v^{-1}]\) (Theorem~\ref{realization of Un}), structures that are not available for the Yangian.



\end{document}